\documentclass[11pt]{amsart}

\usepackage{amsmath}	
\usepackage{amssymb}	
\usepackage{amsthm}
\usepackage[utf8]{inputenc}

\usepackage{url}
\usepackage{mathtools}
\usepackage[english]{babel}

\usepackage{enumitem}

\usepackage{float}

\usepackage{amsfonts}
\usepackage{verbatim}
\usepackage[T1]{fontenc} 	
\usepackage{amscd}
\usepackage{calc}
\usepackage[
hyperfootnotes=false]{hyperref}

\usepackage{xcolor}
\usepackage{crossreftools}

\usepackage{graphicx}

\newtheorem{thm}{Theorem}[section]
\newtheorem{cor}[thm]{Corollary}
\newtheorem{lem}[thm]{Lemma}
\newtheorem{prop}[thm]{Proposition}

\theoremstyle{definition}

\newtheorem{conj}{Conjecture}

\numberwithin{equation}{section}

\newcommand{\norm}[1]{\left\Vert#1\right\Vert}
\newcommand{\abs}[1]{\left\vert#1\right\vert}

\newcommand{\R}{\mathbb{R}}
\newcommand{\D}{\mathbb{D}}
\newcommand{\C}{\mathbb{C}}
\newcommand{\T}{\mathbb{T}}
\newcommand{\Z}{\mathbb{Z}}
\newcommand{\N}{\mathbb{N}}

\newcommand{\HOLO}{\mathcal H}
\newcommand*\CONJ[1]{\overline{#1}}
\newcommand*\closed[1]{\overline{#1}}

\newcommand*\BOP{\mathcal L}
\newcommand{\AUT}{ {\rm Aut}}

\newcommand{\fAlpha}{f_a^{(\alpha)}
}
\newcommand{\fBeta}{f^{(\beta)}
}
\newcommand{\gAlpha}{g^{(\alpha)}}
\newcommand{\gBeta}{g^{(\beta)}}

\newcommand{\BMO}{   {\rm BMO}  }
\newcommand{\VMO}{   {\rm VMO}  }
\newcommand{\BLOCH}{   {\mathcal B }  }
\newcommand{\BMOA}{   {\rm BMOA}  }
\newcommand{\VMOA}{ {\rm VMOA}  }
\newcommand{\WCO}{  \psi C_{\phi}   }
\newcommand{\PEF}{  h   } 
\newcommand{\WUC}{wuC}

\DeclareMathOperator*{\esssup}{ess\,sup}
\DeclareMathOperator{\ARG}{arg}
\DeclareMathOperator{\EXP}{exp}

\newcounter{Genum}

\newcounter{GenumE}

\newcounter{Genumi}

\makeatletter
\newcommand\MCG{
 \stepcounter{Genum}
\item[(G\arabic{Genum})]\def\@currentlabel{(G\arabic{Genum})}
}
\newcommand\MCGtwo{
 \stepcounter{Genumi}
\item[(G\arabic{Genumi})]\def\@currentlabel{(G\arabic{Genumi})}
}
\newcommand\GItem{%
\stepcounter{GenumE}
 \item[(G\arabic{GenumE})]\def\@currentlabel{(G\arabic{GenumE})}
}
\newcommand\GPrimeItem{%
\stepcounter{GenumE}
 \item[(G\arabic{GenumE}\textquotesingle)]\def\@currentlabel{(G\arabic{GenumE}\textquotesingle)}
}
\makeatother

\newcommand{\MainAssumptions}{
Let \(g\in \HOLO(\C_{\Re\geq\frac{1}{2}})   \) be such that \(g|_{[\frac{1}{2},\infty[} \) is (strictly) positive and almost increasing. Assume also that 

\setcounter{Genum}{0}
\begin{enumerate}
 \MCG There exists \(\epsilon_0 > 0\) such that \(  \sup_{0<x<1}   x\, g(\frac{1}{x})^{2+\epsilon_0}   < \infty \), \label{eq:mainAssump1}
\MCG \( g(1/b) \lesssim g(a/b)  g(1/a)    \) for \(0<b\leq a<2\), \label{eq:mainAssump2}
\MCG \(\abs{g(z)} \gtrsim  g(\abs{z}) , \ z\in \C_{\Re\geq\frac{1}{2}} \) ,  \label{eq:mainAssump3}
\end{enumerate}

and let \(v(z) \asymp  g(\frac{1}{1-\abs{z}}) \).

}

\newcommand{\MainAssumptionsTwo}{
There is a \(g\in \HOLO(\C_{\Re\geq\frac{1}{2}})   \) such that \(g|_{[\frac{1}{2},\infty[} \) is (strictly) positive and almost increasing. Assume also that

\setcounter{Genumi}{0}
\begin{enumerate}
 \MCGtwo There exists \(\epsilon_0 > 0\) such that \(  \sup_{0<x<1}   x\, g(\frac{1}{x})^{2+\epsilon_0}   < \infty \), \label{eq:mainAssumpA1}
\MCGtwo \( g(1/b) \lesssim g(a/b)  g(1/a)    \) for \(0<b\leq a<2\), \label{eq:mainAssumpA2}
\MCGtwo \(\abs{g(z)} \gtrsim  g(\abs{z}) , \ z\in \C_{\Re\geq\frac{1}{2}} \) ,  \label{eq:mainAssumpA3}
\end{enumerate}
such that \(v(z) \asymp  g(\frac{1}{1-\abs{z}}) \).

}

\allowdisplaybreaks

\begin{document}

\title[Compactness and related properties...]{Compactness and related properties of weighted composition operators on weighted BMOA spaces}

\author[Norrbo] {David Norrbo}
\address{David Norrbo. Department of Mathematics and Statistics, School of Mathematical and Physical Sciences, University of Reading, Whiteknights, PO Box 220, Reading RG6 6AX, UK.
\emph{E}-mail: d.norrbo@reading.ac.uk}


\begin{abstract} 

 It is shown that a large class of properties coincide for weighted composition operators on a large class of weighted \(\VMOA\) spaces, including the ones with logarithmic weights and the ones with standard weights \((1-\abs{z})^{-c}, \ 0\leq c< \frac{1}{2}\). Some of these properties are compactness, weak compactness, complete continuity and strict singularity. A function-theoretic characterization for these properties is also given. Similar results are also proved for many weighted composition operators on similarly weighted \(\BMOA\) spaces. The main results extend the theorems given in [Proc. Amer. Math. Soc. 151 (2023), 1195--1207], and new test functions that are suitable for the weighted setting are developed.

\end{abstract}

 \maketitle

\section{Introduction}\label{sec:intro}

Let $\D$ be the open unit disk in the complex plane \(\C\) and let \(\HOLO(\D)\) be the space of all analytic functions \(\D \to \C\). The linear space \(\HOLO(\D)\) is a Fréchet space when equipped with the metrizable topology \(\tau_0\), induced by convergence on compact subsets. If \(\phi,\psi\in \HOLO(\D)\) with \(\phi\colon \D\to \D\), then a weighted composition operator is defined as \( \WCO\colon \HOLO(\D) \to \HOLO(\D) \colon f \mapsto \psi f\circ \phi\), which is the product of a multiplication operator \(f\mapsto \psi f\) and a composition operator \(C_{\phi}\colon f\mapsto f\circ \phi\). The main aim of the  paper is to obtain a complete characterization of compactness of \(\WCO\) acting on the space \(\VMOA_v\) in terms of conditions on \(\psi\) and \(\phi\), and to prove that compactness is equivalent to many other operator-theoretic properties for \(\WCO\) on \(\VMOA_v\). Concerning the operator \(\WCO\) on \(\BMOA_v\), a similar result is obtained under some extra assumptions. One sufficient additional condition for the results involving compactness to hold, is that the composition operators \(f\mapsto f\circ \phi\) is bounded on \(\BMOA_v\) and \(\BMOA_v\not\subset H^\infty\). 

In the literature there are many results on the equivalence of weak compactness and compactness for (weighted) composition operators. For example in \cite{Eklund-2017} Eklund, Galindo, Lindström and Nieminen proved  that weakly compact weighted composition operators from  Bloch type spaces into a wide class of Banach spaces of analytic functions on the open unit disk are always compact.  
 The harder well-known problem of whether every weakly compact composition operator on  the space \(\BMOA\) (and \(\VMOA\)) is compact was
solved by Laitila, Nieminen, Saksman and Tylli \cite{Laitila-2013}. Very recently, Laitila, Lindstr\"om and Norrbo \cite{LaitilaLN-2023} extended the result to weighted composition operators on \(\BMOA\) and simplified the function-theoretic characterization of compactness given in \cite{Laitila-2009}. This is now generalized to a large class of weighted \(\BMOA\)  and \(\VMOA\) spaces, contained in \(\BMOA\). In addition to the immediate changes the weight \(v\) presents, the known proof of the invariance of \(p\) in \(\BMOA_{v,p}\) demands some constraints on \(v\). As a consequence, some new estimates are developed (see Proposition \ref{prop:implicationOfJohnNirenberg}) and the previous function-theoretic characterization is changed in a nontrivial way.

For the standard weights \((1-\abs{z})^{-c}, \ 0\leq c< \frac{1}{2}\), boundedness and compactness of composition operators was characterized in \cite{Xiao-2014} by Xiao and Xu. The logarithmic BMOA-space appears naturally in the study of Toeplitz and Hankel operators on \(\BMOA\) and \(H^1\) (see for example \cite{Virtanen-2008} and the references therein), and hence, some of the obtained results on the space \(\BMOA_v\) may prove useful outside the study of weighted composition operators. 

Next, some relevant vector spaces are introduced. Let \(\T\subset \C\) be the unit circle and \(dm(e^{it}) = \frac{dt}{2\pi}, \ t\in [0,2\pi[ \) be the normalized rotationally invariant Haar measure on \(\T\).  For \( 0< p < \infty \), the Hardy space, \(H^p\),  is the linear space of functions \( f\in \HOLO(\D) \) that satisfy
\[
\norm{ f }^p_{H^p} := \sup_{r\in [0,1[} \int_\T \abs{ f(rw) }^p \, dm(w) < \infty
\]
and
\[
H^\infty := \big\{ f\in \HOLO(\D) :  \norm{f}_\infty := \sup_{z\in\D} \abs{f(z)} <\infty \big\}.
\]
For any function \(f\in H^1\), the nontangential limit \(\lim_{z\to w} f(z)\) exists for almost every \(w\in\T\) (see e.g. \cite[Section 2 until Theorem 2.2]{Duren-1970}). A useful composition operator which will be used several times is \(T_c \colon f  \mapsto  [z\mapsto f(cz)], \ c\in \closed{\D}\), which is a rotational shift of argument when \(c\in\T\) and a dilation when \(c\in[0,1]\). The symbol \(\hat{\phi}\) will always represent an analytic automorphism of \(\D\), that is, a function of the form \( z\mapsto b (a-z)/(1-\CONJ{a}z) \), where \(b\in \T\) and \(a\in \D\). The set of these functions is denoted by \( \AUT \). Given \( 0<p<\infty\) and a weight \(v\), that is, an integrable function \( \D \to ]0,\infty[\), the space \( \BMOA_{v,p} \) consists of all functions \(f\in H^p\) such that 
\[
\norm{f}_{*,v,p} :=  \sup_{\hat{\phi} \in {\AUT} } v(\hat{\phi}(0))  \norm{  f\circ\hat{\phi} -   f(\hat{\phi}(0))     }_{H^p} <\infty.
\]
The subscript \(*\) stands for semi-norm and as a superscript \(*\) means the dual space. As in the classical case, the subspace \(\VMOA_{v,p}\subset \BMOA_{v,p}\) is defined as
\[
\VMOA_{v,p} := \Big\{  f\in \BMOA_{v,p} :  \lim_{\substack{  \abs{\hat{\phi}(0)}\to 1 \\ \hat{\phi} \in {\AUT} }  }  v(\hat{\phi}(0))  \norm{  f\circ\hat{\phi} -   f(\hat{\phi}(0))     }_{H^p} = 0 \Big\}.
\]

For \(a,z \in\D\), let \(\sigma_a(z) := (a -z)/(1 -\CONJ{a} z)\), and note that \(\sigma_a(\sigma_a(z)) = z\). Since 
\[
 b \frac{a-z}{1-\CONJ{a}z}  = \frac{ba-bz}{1-\CONJ{ba}bz}  , \ b\in \T 
\]
and \(\norm{T_c f }_{H^p} = \norm{f }_{H^p},\ c\in \T\), it follows that
\[
\norm{f}_{*,v,p} =  \sup_{a\in \D } v(a)  \norm{  f\circ\sigma_a -   f(a)     }_{H^p} =  \sup_{a\in \D } v(a) \left(  \int_{\T} \abs{  f(w) -   f(a)     }^p  P_a(w)\,  dm(w)\right)^{\frac{1}{p}},
\]
where \( P_a(z):= (1-\abs{a}^2)/\abs{ 1- \CONJ{a}z}^2   \) is the Poisson kernel. 
The following well-known identities will be used throughout the paper without any explicit mentioning:  
\[
1-\abs{\sigma_a(z)}^2  = \frac{(1-\abs{z}^2)(1-\abs{a}^2)}{      \abs{ 1-\CONJ{a}z}^2   } =  (1-\abs{z}^2) \abs{\sigma_a'(z)} = (1-\abs{z}^2) P_a(z) ,\ a,z \in \D.
\]

Before we proceed, we introduce some elementary notations. The space bounded linear operators on a Banach space \(X\) is denoted by \(\BOP(X)\). The evaluation maps, \( \HOLO(\D) \ni f \mapsto f(z) \), are denoted by \(\delta_z ,  z\in \D\). For two quantities \(A,B\geq 0\) the notation \(A\lesssim B\) or \(B\gtrsim A\) means that there exists a constant \(C>0\) such that \(A\leq CB\). The constant will quite often depend on an exponent \(1\leq p,q<\infty\) (John-Nirenberg related) or the weight \(v\) (or the related function \(g\)). Dependencies will be mentioned by a subscript. Moreover, \( A \asymp B \) means that both \(A\lesssim B\) and \(A\gtrsim B\) hold, and in this case we say that \(A\) is equivalent to \(B\). The function \(\chi_A\) will appear inside integrals and is a function of all relevant integration variables. The subscript \(A\) can be a set, in which case the function takes the value \(1\) if the integration variables are inside the set and \(0\) else. The subscript \(A\) can  also be a logic expression, which works as an abbreviation for the set of points satisfying the expression. Henceforth, \(\BMOA_v := \BMOA_{v,2}\) and for \(v=1\), \(\BMOA := \BMOA_{1}\). For convenience, we also define
\[
\gamma(f,a,p) := \norm{f\circ \sigma_a - f(a)}_{H^p} , \quad a\in \D , \, 1\leq p<\infty .
\] 
A real-valued, nonnegative function \(f\) is said to be almost increasing if there is a constant \(C\geq 1\) such that \(f(y) \leq C f(x)\) whenever \(y\leq x \). If \(C=1\), then the word increasing is used. For a set \(M\subset \C\) and a number \(c\in \C\), the notation \(cM := Mc := \{cx\in \C \colon x\in M  \}\) is used. When a function \(f\) defined on \(M\) is considered on a subdomain \(M'\subset M\), the restriction of \(f\) to \(M'\) is denoted by \(f|_{M'}\). The notation \(\C_{\Re\geq c} := \{z\in\C : \Re z \geq  c \}\), \(c\in\mathbb R\), denotes a right half plane and \(\HOLO(M)\) is the linear space of functions analytic on some domain containing \(M\).

In this paragraph we assume that the weight \(v\colon \D \to ]0,\infty[\) is radial (rotationally invariant), that is, \(v(z)=v(\abs{z}), \ z\in \D\). In the classical case, where \( v\equiv 1\), the fact that \( \AUT \) is a group with respect to composition yields  \( \norm{   \smash{   f\circ\hat{\phi}  }  }_{*,v,p} = \norm{f}_{*,v,p} \). For more general weights, it holds that \( \norm{ \smash{  f\circ\hat{\phi}   }  }_{*,v,p} = \norm{f}_{*,v,p} \) for every \(\hat{\phi}\in \AUT\), if and only if \(v\) is constant. In general, for every \(\hat{\phi}\in \AUT \) it holds that
\[
 \norm{f\circ\hat{\phi}}_{*,v,p} =  \norm{f}_{*,v\circ \hat{\phi}^{-1},p}.
\]

If the weight \(v\) is almost increasing, equivalent to a radial weight and satisfies \(v(b)\lesssim v(\frac{b-a}{1-a}) v(a) ,\ 0\leq a\leq b <1\), it still holds that 
\[
f\in \BMOA_{v,p} \ \Longrightarrow  f\circ\hat{\phi}\in \BMOA_{v,p}
\]
(see Lemma \ref{lem:compWithAutomorph}). In general, it is not evident that this is true (see Conjecture \ref{conj:invariantCompAut} in Section \ref{sec:proofsOfMainResults}).

A useful tool is the Hardy-Stein estimates (see for example \cite[Theorem 4.22]{Zhu-2005}), from which it follows that for \(0<p<\infty\)
\begin{equation}\label{eq:LittlewoodPaley}
\norm{ f\circ \sigma_a - f(a) }_{H^p}^p \asymp \int_{\D} \abs{f'(z)}^2 \abs{f(z)-f(a)}^{p-2}  \frac{(1-\abs{z}^2)(1-\abs{a}^2)}{      \abs{ 1-\CONJ{a}z}^2   } \, dA(z) \quad f\in H^p.
\end{equation}
The case \(p=2\) is also a consequence of the well-known Littlewood-Paley identity.

If the weight satisfies 
\[
\sup_{a\in\D} v(a) (1-\abs{a})^{\frac{1}{p} -\epsilon }<\infty
\] 
for some \(\epsilon>0\) and \(p\geq 2\), we have by \eqref{eq:LittlewoodPaley} and \cite[Theorem 1.12]{Zhu-2005} that all bounded functions \(f\in \HOLO(\D) \) with \smash{\( \abs{f'(z)}^2\lesssim (1-\abs{z})^{-(1+ p\epsilon)}\)} belong to \(\BMOA_{v,p}\). By inclusion, the same functions \(f \in \BMOA_{v,q}, 1\leq q <2\).

The main function-theoretic characterization of compactness concerns the following two functions, \(\alpha\) and \(\beta\). For \(\psi,\phi\in H^1\)  and \(\phi \colon \D\to \D, \ a\in\D\), we define
\[
\alpha(\psi,\phi,a) :=  \frac{v(a)}{  v( \phi(a) )   }  \abs{\psi(a)} \norm{\phi_a}_{H^2},
\]
where \(\phi_a:=\sigma_{\phi(a)}\circ \phi \circ \sigma_a\). The function \(\alpha\) is sufficient for a characterization of boundedness and compactness for composition operators, but for weighted composition operators, a complementary function
\[
\beta(\psi,\phi,a) :=   \norm{  \delta_{\phi(a)}  }_{ (\BMOA_v)^* }   v(a)  \gamma(\psi,a,1)
\]
is also needed. It follows that for \(\sup_{a\in \D} \beta(\psi,\phi,a) < \infty \) to hold, it is necessary that \(\psi\in\BMOA_{v,1}\).

\subsection{Main results}\label{subsec:MainResults}

\noindent
\MainAssumptions

If \(v\) is such a weight, also called admissible weight, we have the following two theorems.

\begin{thm}\label{thm:MainThm}

\[
\WCO\in\BOP(\BMOA_v)   \  \Longleftrightarrow  \  \sup_{a\in\D} (\alpha(\psi,\phi,a) +\beta(\psi,\phi,b)) <\infty
\]
and
\begin{align*}
\WCO\in\BOP(\VMOA_v)   \  &\Longleftrightarrow  \   \sup_{a\in\D} (\alpha(\psi,\phi,a) + \beta(\psi,\phi,a))<\infty \text{ and } \psi,\psi\phi\in\VMOA_v\\
&\Longleftrightarrow  \   \sup_{a\in\D} (\alpha(\psi,\phi,a) + \beta(\psi,\phi,a))<\infty, \psi\in\VMOA_v \text{ and }\\
& \hspace{2cm} \lim_{\abs{a}\to 1} v(a) \psi(a)  \gamma(\phi,a,2)  = 0. 
\end{align*}
More specifically, for \(X=\BMOA_v\) or \(X=\VMOA_v\) and \(\WCO\colon X\to X\), it holds that
\[
\norm{\WCO}_{\BOP(X)} \asymp_{v,g}    \abs{\psi(0)} \norm{\delta_{\abs{\phi(0)}}}_{X^*} +    \sup_{a\in\D} \alpha(\psi,\phi,a) +   \sup_{a\in\D} \beta(\psi,\phi,a),
\]
where the evaluation map satisfies
\[
\norm{\delta_{z}}_{X^*} \asymp_{v,g}  1 + \int_0^{\abs{z}} \frac{dt}{ (1-t) v(t) }.
\]

\end{thm}

\begin{thm}\label{thm:MainThmCpt}

Let \(X= \BMOA_v\) or \(X=\VMOA_v\). For \(\WCO\in\BOP(X)\) with at least one of the following properties:
\begin{itemize}
\item \( C_{\phi} \in \BOP(\BMOA_v)\), and \(\BMOA_v \not\subset  H^\infty\) or \(\psi\in\VMOA_v\),
\item \(\WCO|_{\VMOA_v} \in \BOP(\VMOA_{v} )\),
\end{itemize} 
the following are equivalent:
\begin{enumerate}
\item \(\WCO \) is compact,    \label{eq:Main_Cpt}
\item \(\WCO \) is weakly compact, \label{eq:Main_wCpt}
\item \(\WCO \) is completely continuous, \label{eq:Main_cCts}
\item \(\WCO \) does not fix a copy of \(c_0\) \ (\(c_0\)-singular), \label{eq:Main_fixC0}
\item \(\WCO \) is unconditionally converging, \label{eq:Main_UC}
\item \(\WCO \) is strictly singular, \label{eq:Main_SS}
\item \(\WCO \) is finitely strictly singular, \label{eq:Main_FSS}
\item \( \limsup_{\abs{\phi(a)}\to 1} (\alpha(\psi,\phi,a) +\beta(\psi,\phi,a)) = 0 \),\label{eq:Main_FuncTheorCharact} 
\item \( \limsup_{\abs{\phi(a)}\to 1} \Big(  \norm{\WCO\fAlpha}_{\BMOA_v}  +\beta(\psi,\phi,a) \Big) = 0 \),\label{eq:Main_FuncTheorCharactTwo} 
\end{enumerate}
where 
\[
\fAlpha\colon z \mapsto  \frac{\sigma_{\phi(a)}(z) - \phi(a) }{v(\phi(a))}.
\]
In the case \(\WCO|_{\VMOA_v} \in \BOP(\VMOA_v )\), the following are equivalent:
\begin{enumerate}[resume]
\item[\eqref{eq:Main_wCpt}] \(\WCO \) is weakly compact,
\item \(\WCO(\BMOA_v)\subset \VMOA_v \) (here \(\WCO\) is interpreted as an operator on \(\HOLO(\D)\)),\label{eq:Main_Gantmacher}
\item \(\forall \epsilon>0 \ \exists C>0 : \norm{\WCO f}_{\BMOA_v} \leq N\norm{f}_{H^2} + \epsilon \norm{f}_{\BMOA_v}  , \quad f\in \VMOA_v. \)\label{eq:Main_AC}
\end{enumerate}

\end{thm}

This yields a complete characterization of the properties listed for \(\WCO\) on the space \(\VMOA_v\), but not on \(\BMOA_v\). Note that for \(v\equiv 1\) the function \(\beta\) in this work is smaller than the one used in \cite{Laitila-2009}. Since \( \norm{\delta_{\phi(a)} }_{X^*} \), where \( X= \BMOA_v \) or \( X= \VMOA_v \), is not necessarily equivalent to a radial function, it is unclear if the reverse relation holds, see Conjecture \ref{conj:ComparingBetas}.

Concerning the multiplication operator, S. Ye characterized, in \cite{Ye-2020}, boundedness and compactness of the multiplication operators on \(\BMOA_v\) and \(\VMOA_v\) with \(v(z) = \ln (2/(1-\abs{z}^2))\). For more general weights, S. Janson described the multipliers in \cite[Theorem 2]{Janson-1976} using a different proof. The following result can essentially be compared to Theorem 2 by Janson:
\begin{cor}\label{cor:MainCorMult}
Let \(X=\BMOA\) or \(X=\VMOA\). The multiplication operator \(M_\psi \colon f\mapsto \psi f \) is bounded on \(X_v\) if and only if
\[
\psi\in H^\infty \cap \BMOA_w,
\]
where 
\[
 w(a)=  v(a)\bigg(1 + \int_0^{\abs{a}} \frac{dt}{  (1-t) v(t) }  \bigg) \asymp v(a) \norm{\delta_a}_{(X_v)^*}     ,\quad a\in \D. 
\]
Especially, if \(\sup_{a\in\D}\int_0^{\abs{a}} \frac{dt}{  (1-t) v(t) } <\infty \), or equivalently \(X_v\subset H^\infty\), then \(X_v=X_w\) is an algebra.

Moreover, the operator \(M_\psi\) is compact (or satisfies any of the other equivalent properties   \eqref{eq:Main_wCpt}--\eqref{eq:Main_FSS}) on \(X_v\) if and only if \( \psi \equiv 0\).
\end{cor}

Although the statement about compactness does not following immediately from Theorem \ref{thm:MainThmCpt}, it holds that \( C_{\phi} \in \BOP(\BMOA_v)\) and \(\limsup_{\abs{\phi(a)}\to 1} (\alpha(\psi,\phi,a) > 0\) unless \(\psi\equiv 0\), so the same proof as for the theorem applies in this case.

Concerning the composition operator, we have the following generalization of the part of the work \cite{Xiao-2014} by J. Xiao and W. Xu, which concerns boundedness and compactness of \(C_\phi\) on the Analytic Lipschitz spaces, \( \BMOA_v\), where \( v(a)=(1-\abs{a}^2)^{-c}, \ c\in]0,1/2[, \, a\in\D \).

\begin{cor}\label{cor:MainCorComp}
Let \(X=\BMOA\) or \(X=\VMOA\) and let \(v\) be an admissible weight. The composition operator \(C_\phi\) is bounded on \(X_v\) if and only if 
\[
\sup_{a\in\D} \frac{ v(a) }{ v(\phi(a)) }\norm{\phi_a}_{H^2} < \infty,
\]
and \(C_\phi\in\BOP(X_v)\) is a compact operator if and only if 
\[
\limsup_{\abs{\phi(a)}\to 1} \frac{ v(a) }{ v(\phi(a)) }\norm{\phi_a}_{H^2} = 0.
\]

Furthermore, it is necessary that \(\phi\in X_v\) for \(C_{\phi}\) to be bounded.
\end{cor}

\vspace{1cm}

Some examples of admissible weights are \(v(z) = (1-\abs{z})^{-c} \), that is, \( g(z)= z^c  , \ 0\leq c <1/2\) and   
\[
v(z) = (\ln(\frac{e}{1-\abs{z}}) )^c , \text{ that is, }  g(z)= (\ln (ez))^c  , \  c>0,
\]
where the branch cuts are chosen appropriately (e.g. along the negative real axis). To see that condition \ref{eq:mainAssump2} holds for \( g(z)= (\ln (ez))^c  , \  c>0 \), we substitute \(a=e^{-A}\) and \(b=e^{-B}\), and use the fact that
\[
 \frac{1+B}{(1+A)(B-A+1)} \leq \frac{1}{(1+A)} + \frac{1}{(B-A+1)}\leq  1 + \frac{1}{\ln\frac{e}{2}} \quad\quad \Big(\ln \frac{1}{2} < A\leq B <\infty \Big).
\]

In view of Theorem \ref{thm:MainThm} and the evaluation map, the logarithmic weights mentioned above yield a space \(\BMOA_v\subset H^\infty\) if and only if \(c > 1\). Concerning the standard weights, \(\BMOA_v\subset H^\infty\) for all \(c>0\). Moreover, if \(v\) is an admissible weight and \(u\colon \D\to [a,b]\) for some \(0<a<b<\infty\), then \(vu\) is an admissible weight. This is clear from the fact that the norms \(\norm{\cdot}_{\BMOA_v}\) and \(\norm{\cdot}_{\BMOA_{vu} }\) are equivalent. An example of a non-radial admissible weight is \( z\mapsto \abs{2+z}(1-\abs{z})^{-1/4}\). Another useful property of the set of admissible weights is given in the paragraph where \eqref{eq:wOne} appears, that is, the product of two admissible weights is admissible if it satisfies the growth restriction \ref{eq:mainAssump1}.

The article is structured as follows: Section \ref{sec:Prelim} contains some more definitions and some preliminary results. The most important new result in this section is Proposition \ref{prop:implicationOfJohnNirenberg}, which in addition to \cite[Proposition 2.6]{dyakonov-1998} (\(\BMOA_{v,p} =  \BMOA_{v,1}\) for suitable \(v,1<p<\infty\)), contains some more precise estimates, which are, for example, used to prove that the given function-theoretic condition is sufficient for \(\WCO\in \BOP(\VMOA_v)\) to be compact (Theorem \ref{thm:SuffForCptVMOA}). Another important tool is the denseness of polynomials in \(\VMOA_v\), which is given in Proposition \ref{prop:ConvergenceInVMOAV}. In Theorem \ref{thm:Gantmacher} it is shown that \(\VMOA_v^{**}\) is isometrically isomorphic to \(\BMOA_v\). The section ends with a brief discussion regarding the demands for the weight \(v\) to be admissible. The two following sections contain some preparatory results, of which a few might be of interest on their own. Section \ref{sec:ImportantLemmas} contains three important results: Lemmas \ref{lem:coolStuff} and  \ref{lem:GoodOne} concern the main function-theoretic characterization for boundedness and compactness, and they are related to the functions \(\alpha\) and \(\beta\) respectively. The third important result is Corollary \ref{cor:EvaluationBoundedOnBMOAv}, which is an estimate for the evaluation map. In Section \ref{sec:FuncAlphaBeta}, the test functions are developed and proofs of important properties for these functions are given. 

Section \ref{sec:WrappingItUp} contains, in contrast to Section \ref{sec:Prelim}, the main parts of the main theorems, whose proofs make heavy use of the fact that the operator is a weighted composition operator and it acts on \(\VMOA_v\) (or \(\BMOA_v\)), where \(v\) is admissible. The function-theoretic characterization for boundedness is proved followed by the remaining implications (the ones that are not of a more general type) to conclude that Theorem \ref{thm:MainThmCpt} holds. Section \ref{sec:Examples} contains some examples of symbols \(\psi\) and \(\phi\) making \(\WCO\) bounded or even compact. Using the obtained results, the proofs for the three main results are completed and summarized in Section \ref{sec:proofsOfMainResults}. Section \ref{sec:proofsOfMainResults} also contains some conjectures.

\section{Preliminaries}\label{sec:Prelim}

This section contains some more definitions and preliminary results. We begin the section by showing that the polynomials in \(\HOLO(\D)\) are often dense in a proper subspace of \(\BMOA_{v,p}\). The section ends with a brief discussion concerning the conditions \ref{eq:mainAssump1} and \ref{eq:mainAssump2}.  

\begin{prop}\label{prop:polynomialsInBMOAV}
Let \(1< p<\infty\). If \(  \sup_{a\in\D}  v(a)  (1-\abs{a})^\frac{1}{p} < \infty \), then the polynomials belong to \(\BMOA_{v,p}\). If \(  \lim_{a\to 1}  v(a)  (1-\abs{a})^\frac{1}{p} = 0 \), then \(\VMOA_{v,p} \subset \BMOA_{v,p}\) is a closed subspace containing the analytic polynomials. If \(p=1\) the same statements are valid when  \( v(a)  (1-\abs{a})^\frac{1}{p} \) is replaced by \( v(a)  (1-\abs{a}) \ln \frac{e}{1-\abs{a}} \). 
\end{prop}
\begin{proof}
Let \({Q}(z)=\sum_{k=0}^n c_k z^k\) be a polynomial on the disk. Now
\begin{align*}
{Q}(z) - {Q}(a)  &=  \sum_{k=0}^n c_k z^k - \sum_{k=0}^n c_k a^k =   \sum_{k=1}^n c_k (z^k - a^k) = (z - a) \sum_{k=1}^n c_k \sum_{j=0}^{k-1}  z^j a^{k-1-j}. 
\end{align*}
There exists a constant \(C({Q},p)\)  only dependent of the polynomial \({Q}\) and \(1\leq p<\infty\)  such that for \(1<p<\infty\),
\begin{align*}
\gamma({Q},a,p)^p  &= \int_\T \abs{(z - a)}^p \abs{   \sum_{k=1}^n c_k \sum_{j=0}^{k-1}  z^j a^{k-1-j}     }^p  P_a(z) \, dm(z)  \leq C({Q},p)  (1-\abs{a}^2). 
\end{align*}

Assuming \(  \lim_{a\to 1}  v(a)  (1-\abs{a})^\frac{1}{p} = 0 \), this yields
\[
v(a)  \gamma({Q},a,p) \leq  C'({Q},p) v(a)  (1-\abs{a})^\frac{1}{p} \stackrel{\abs{a}\to 1}{\longrightarrow} 0,
\]
that is, \({Q}\in \VMOA_{v,p}\). If \((f_n)\subset \VMOA_{v,p}\) is a sequence, converging with respect to \(\norm{\cdot}_{\BMOA_{v,p}}\) to an analytic function \(f\), and \(\epsilon>0\), then for \(n\) large enough 
\begin{equation*}
v(a)  \gamma(f,a,p) \leq  v(a)  \gamma(f_n,a,p)  + v(a)  \gamma(f_n-f,a,p)   \leq  v(a)  \gamma(f_n,a,p)  + \epsilon.
\end{equation*}
Letting first \(\abs{a}\to 1\), then \(\epsilon\to 0\), we obtain  
\[
\lim_{ \abs{a}\to 1 } v(a)  \gamma(f,a,p)  = 0.
\]
Hence, \(\VMOA_{v,p}\) is a closed subspace of \(\BMOA_{v,p}\), containing the polynomials. 

Finally, if \(p=1\),
\[
\gamma({Q},a,p) = (1-\abs{a}^2 )   \int_0^{2\pi}   J(e^{it}) \abs{ 1- a e^{-it} }^{p-1}   \frac{dt}{ 2\pi }  \lesssim C({Q},p)  (1-\abs{a}^2) \ln \frac{e}{1-\abs{a}},
\]
 and the statement follows similarly to the proof above.
\end{proof}

The first part of Lemma \ref{lem:EquivBasisC0} yields that \(\BMOA_v\) is a Banach space and since \(\VMOA_v\) is closed, it is also a Banach space. 

Proposition \ref{prop:polynomialsInBMOAV} shows some similarities between the weighted spaces \(\BMOA_v\) and \(\VMOA_v\), and their unweighted variants. The following Proposition shows some contrast between the spaces, which renders some classical approaches ineffective. The function \(z\mapsto z^n\) is one common tool in characterizing compactness in the unweighted setting, and can be found in, for example, \cite{Laitila-2015,Smith-1999,Wulan-2009} (and in some form also in \cite[(3.13)]{Laitila-2009}).

\begin{prop}
Let \(1< p<\infty\) and assume \(  \lim_{ \abs{a}\to 1 }  v(a)  (1-\abs{a})^\frac{1}{p} = 0 \). The family \( F \) consisting of \(f_n\colon z\mapsto z^n, n\in\mathbb N \) belongs to \(\VMOA_{v,p}\), and  \(\sup_n\norm{f_n}_{*,v,p}<\infty\) if and only if \(v\) is bounded. 
\end{prop}
\begin{proof}
The fact that \(F\subset \VMOA_{v,p}\) follows from Proposition \ref{prop:polynomialsInBMOAV} and the trivial estimate \(\norm{f_n}_{\BMOA_{1,p}} \leq 3\norm{f_n}_\infty = 3 \) proves one of the implications in the remaining statement. Now, let \(t_n = 1-n^{-\frac{1}{2}}, n\in\mathbb N\). By \eqref{eq:LittlewoodPaley} and the fact that \(\int_{\T} P_b \, dm = 1\) for every \(b\in\D\), we have
\begin{align*}
\norm{ f_n\circ \sigma_a - f_n(a) }_{H^2}^2  &\asymp \int_{\D} n^2\abs{z}^{2(n-1)}  (1-\abs{\sigma_a(z)}^2) \, dA(z)    \geq \int_{t_n}^1 n^2 x^{n-1}  \frac{(1-x)(1-\abs{a}^2)}{     1-\abs{a}^2 t_n   } \, dx.
\end{align*}
Moreover, by the well-known asymptotics for the classical beta function, we have

\begin{align*}
n^{-2} \stackrel{ n\to\infty }{\sim}   \int_{t_n}^1  x^{n-1} (1-x) \, dx + O\big(  2^{-\sqrt{n}}  \big).
\end{align*}
This yields that
\[
\limsup_{n\to\infty} \norm{ f_n\circ \sigma_a - f_n(a) }_{H^2}^2  \gtrsim 1.
\]
Finally,
\[
\sup_n\norm{f_n}_{*,v,p} \geq  \sup_{a\in\D} \limsup_{n\to\infty}  v(a) \norm{ f_n\circ \sigma_a - f_n(a) }_{H^2} \gtrsim   \sup_{a\in\D}  v(a).  
\]
\end{proof}

\subsection{Consequences of John-Nirenberg's result}

\

We begin by introducing the spaces \(\BMO_{v,p}\) and \(\VMO_{v,p}\). (Note that the following representations are not standard and these spaces will only be used in this section.) We define
\[
\norm{f}_{*,\BMO_{v,p}} :=  \sup_{ \substack{I\subset\T\\ I\text{ arc} }  }  v(1-m(I))  \eta(f,I,p),
\] 
where 
\[
\eta(f,I,p) := \left( \int_I  \abs{  f(z) - m_I(f)  }^p \, \frac{dm(z)}{  m(I)  } \right)^{\frac{1}{p}},
\]
and \( m_I(f) :=  \int_I  f(w)  \frac{ dm(w)}{m(I)}   \) is the mean of the function on the arc \(I\subset \T\). Now

\[
\BMO_{v,p} := \Big\{  f\in H^p : \norm{f}_{*,\BMO_{v,p}} < \infty   \Big\}
\]
and
\[
\VMO_{v,p} := \Big\{  f\in \BMO_{v,p} :  \lim_{ \substack{m(I) \to 0 \\ I\text{ arc} }  } v(1-m(I))\eta(f,I,p) = 0 \Big\}.
\]

The following well-known results can be found in, for example, \cite{Baernstein-1979} (see also \cite{Garnett-1981,Girela-2001}). If \(v\asymp 1\), then \(\BMOA_{v,p} = \BMOA_v , \ 0<p<\infty\) and for  \(1\leq p<\infty\), any of the semi-norms \( \norm{f}_{*,v,p} \) is comparable with any of \(\norm{f}_{*,\BMO_{v,q}} , \  1\leq q <\infty\). The independency of the parameter \(p\) is, in the work of Baernstein (\cite{Baernstein-1979}), proved in the conformally invariant setting, \(\BMOA\), making use of the group structure of \( \AUT \). This yields a better result than the classical approach carried out in \cite{Garnett-1981} and \cite{Girela-2001}, which proves independency in the classical \(\BMO\) setting and then apply the result that for a fixed \(p, \ \norm{f}_{*,\BMO_{v,p}}\asymp\norm{f}_{*,v,p} \) when \(v\asymp 1\). It is, however, no surprise that stronger results can exist in the analytic setting compared to the measurable setting due to additional structure. When \(v\) is almost increasing, we have the following results, pointed out in \cite{dyakonov-1998} by Dyakonov. (The weights \(\varphi(t)\) in \cite{dyakonov-1998} and \cite{Spanne-1965} are comparable to \(v(1-t)^{-1}\) for \( t\in]0,1[\).) The following proposition contains \cite[Proposition 2.6]{dyakonov-1998}, but also some new crucial estimates for the proofs of the main results.

\begin{prop}\label{prop:implicationOfJohnNirenberg}
Let \(1\leq q\leq p < \infty\) and \(v\colon \D \to ]0,\infty[\) be radial. Assume \(v|_{[0,1[}\) is almost increasing and there is an \(\epsilon_0>0\) such that \(x\mapsto v|_{[0,1[}(1-x) x^{\frac{1}{p}-\epsilon_0} \) is almost increasing. Then
\[
\BMOA_{v,p} = \BMO_{v,q}  \quad \text{ and } \quad \norm{f}_{*,\BMO_{v,q}} \asymp_{v,p,q,\epsilon_0} \norm{f}_{*,v,p}.  
\]
Moreover, for \(R\in]0,1[\)
\begin{equation}\label{eq:Rineq}
    \sup_{m(I) \leq  R} v(1-m(I)) \eta(f,I,p) \asymp_{v,p,q}     \sup_{m(I) \leq  R} v(1-m(I)) \eta(f,I,q)  \lesssim \sup_{ \abs{a}\geq 1-R } v(a) \gamma(f,a,q) 
\end{equation}
and for any \(R_{\BMO}\in]0,1[\) and \(R_A\in]0,R_{\BMO}/2]\), we have 
\begin{equation}\label{eq:R_BMOineq}
\sup_{\abs{a}\geq 1-R_A} v(a) \gamma(f,a,q)  \lesssim_{v,q,\epsilon_0}     \sup_{m(I) \leq  R_{\BMO}} v(1-m(I)) \eta(f,I,q) + \norm{f}_{*,\BMO_{v,q}} \bigg(  \frac{2R_A}{R_{\BMO}}   \bigg)^{\epsilon_0   }.  
\end{equation}
Hence,
\begin{equation}\label{eq:VMOequvalentNorms}
\VMOA_{v,p} = \VMO_{v,q} \quad  \text{ and }  \quad   \limsup_{\abs{a}\to 1} v(a) \gamma(f,a,p)  \asymp_{v,p,q,\epsilon_0}  \limsup_{m(I)\to 0} v(1-m(I)) \eta(f,I,q).    
\end{equation}
\end{prop}
Note that if \(x\mapsto v|_{[0,1[}(1-x) x^{\frac{1}{p}-\epsilon_0} \) almost increasing for some \(p=p_0\), it is almost increasing for any \(0< p\leq p_0\).

The proof is split into a few results and most of them can be found, in some form, in \cite{Baernstein-1979,Garnett-1981} and \cite{Girela-2001}. The following proposition follows immediately from the proof in the classical unweighted \(\BMO\) setting (see e.g. \cite[p.~73]{Girela-2001}).
\begin{prop}\label{prop:ConformInvBMOAimplyintBMOA}
For \(1\leq p <\infty\) and \(\inf_{x\in]0,1[} v(x) >0\) and all \(R>0\), it holds that
\[
 \sup_{ \substack{m(I)\leq R \\ I\text{ arc} }  }  v(1-m(I))  \eta(f,I,p) \leq 2 \sup_{\abs{a}\geq 1-R}  v(a) \gamma(f,a,p).
\]
Furthermore,
\[
\norm{f}_{*,v,p} \gtrsim \norm{f}_{*,\BMO_{v,p}}.      
\]

\end{prop}

For the proof of \cite[Theorem 3.1]{Girela-2001}, a dyadic decomposition of \(\T\) is used. To be able to summarize the approximations made from the dyadic decomposition in the classical fashion, the weight needs to satisfy an extra condition stated in \cite[Proposition 2.6]{dyakonov-1998}, the fact that \((v(1-t))^{-1}\) need to be of upper type less than \( 1/p \). In this work, the comparable property found in \cite{Spanne-1965} of almost increasing/decreasing is used. Inspired by \cite[Proof of Theorem 3.1]{Girela-2001}, we have the following proposition, which includes a new, suitable estimate for this work.

\begin{prop}\label{prop:intBMOAimplyConformInvBMOA}
Let \(1\leq p < \infty\) and \(v\colon \D \to ]0,\infty[\) be radial. Assume there is an \(\epsilon_0>0\) such that \(x\mapsto v|_{[0,1[}(1-x) x^{\frac{1}{p}-\epsilon_0} \) is almost increasing. Then
\[
 \norm{f}_{*,v,p}  \lesssim_{v,p,\epsilon_0}  \norm{f}_{*,\BMO_{v,p}}
\]
and for any \(R_{\BMO}\in]0,1[\) and \(a\in\D\) with \(\abs{a}\geq 1-R_{\BMO}/2 \), we have 
\begin{equation}\label{eq:lemR_BMOineq}
 v(a) \gamma(f,a,p)  \lesssim_{v,p,\epsilon_0}     \sup_{m(I) \leq  R_{\BMO}} v(1-m(I)) \eta(f,I,p) + \norm{f}_{*,\BMO_{v,p}} \bigg(  \frac{2(1-\abs{a})}{R_{\BMO}}   \bigg)^{\epsilon_0   }.  
\end{equation}

\end{prop}

\begin{proof}
Let \(a\in \D\) and define \(J_k, \ k=0,1,\ldots,N\) to be the arc with center \(a/\abs{a}\)  and \(m(J_k) = 2^k (1-\abs{a})\), where \(N\) is the number such that \(m(J_N) < 1 \leq  2 m(J_N)\). We also put \(J_{N+1}:=\T\). The relation between \(a\) and \(N\) is given by 
\[
N=N(a) = \max\Big\{  n\in \Z :  n < \ln \frac{1}{1-\abs{a}  } \frac{1}{\ln 2}   \Big\}  . 
\]
We have \(\T\setminus J_0 =  \bigcup_{k=1}^{N+1} (J_k\setminus J_{k-1}) \) and hence, with the aid of  Minkowski's inequality and a variant of \cite[Lemma 3.2]{Girela-2001}, we have

\begin{align*}
 \bigg(    \int_{\T} \abs{f(z) - f(a)}^p P_a(z) \, dm(z)  \bigg)^{\frac{1}{p}}&  \lesssim   \bigg(  \frac{1}{ m(J_0) }\int_{J_0} \abs{  f(z) - m_{J_0}(f)   }^p \, dm(z) \bigg)^{\frac{1}{p}}  \\
&\quad +  \sum_{k=1}^{N+1} \left( \int_{J_k\setminus J_{k-1}}  \abs{  f(z) - m_{J_0}(f)   }^p \inf_{w\in J_{k-1}} P_a(w) \, dm(z)    \right)^{\frac{1}{p}}.
\end{align*}
For the integrals in the second term, we apply Minkowski's inequality and \cite[Lemma 3.4]{Girela-2001} to obtain
\begin{align*}
 \bigg( \int_{J_k\setminus J_{k-1}}  \abs{  f(z) - m_{J_0}(f)   }^p \inf_{w\in J_{k-1}} P_a(w) \, dm(z)    \bigg)^{\frac{1}{p}}  \lesssim_p   \frac{1}{2^{\frac{k}{p}} }   \eta(f,J_k,p)  +       \frac{1}{2^{\frac{k}{p}} }    \abs{  m_{J_k}(f) - m_{J_0}(f)   }.
\end{align*}
Furthermore, a variant of  \cite[Lemma 3.3]{Girela-2001} gives
\begin{align*}
 \abs{  m_{J_k}(f) - m_{J_0}(f)   }  \lesssim \sum_{j=1}^k \eta(f,J_j,1)  
\end{align*}
yielding
\begin{align*}
\sum_{k=1}^{N+1} \left(  \int_{J_k\setminus J_{k-1}}  \abs{  f(z) - m_{J_0}(f)   } \inf_{w\in J_{k-1}} P_a(w) \, dm(z)  \right)^{\frac{1}{p}} &\lesssim_p        \sum_{k=1}^{N+1}   \frac{1}{2^{\frac{k}{p}} }   \sum_{j=1}^k \eta(f,J_j,1) \\
& \lesssim_p   \sum_{j=1}^{N+1}   \frac{1}{2^{\frac{j}{p}} }   \eta(f,J_j,p)  . 
\end{align*}
Since \(m(J_j) = 2^j (1-\abs{a})\) and
\begin{equation}\label{eq:usefulForJohnNirenberg}
v(a) m(J_0)^{\frac{1}{p}-\epsilon_0} = v(1-m(J_0)) m(J_0)^{\frac{1}{p}-\epsilon_0} \lesssim_v v(1-m(J_j)) m(J_j)^{\frac{1}{p}-\epsilon_0} 
\end{equation}
for every \(j=1,\ldots, N+1\) by assumption, we have now obtained
\begin{align*}
  v(a)   \left(    \int_{\T} \abs{f(z) - f(a)}^p P_a(z) \, dm(z)  \right)^{\frac{1}{p}} \lesssim_{v,p}     \sum_{j=0}^{N+1}   \frac{1}{2^{\epsilon_0 j} }  v(1-m(J_j)) \eta(f,J_j,p).
\end{align*}
Using \( v(1-m(J_j)) \eta(f,J_j,p) \leq  \norm{f}_{*,\BMO_{v,p}}  \), the first statement, \(\norm{f}_{*,v,p}  \lesssim_{v,p,\epsilon_0}  \norm{f}_{*,\BMO_{v,p}}\), follows. For the second, fix \(R_{\BMO}\in ]0,1[\). For \(a\in\D\) with \(\abs{a}\geq 1 - R_{\BMO}/2 \), let \(N_I = N_I(a) \in [1,N] \) be the integer such that 
\( m(J_{N_I(a)}) \in ] R_{\BMO}/2,R_{\BMO} ]  \). Furthermore, using \eqref{eq:usefulForJohnNirenberg}, we have
\[
 v(a) \sum_{j=N_I}^{N+1}   \frac{1}{2^{\frac{j}{p}} }   \eta(f,J_j,p)  \lesssim_v  \sum_{j=N_I}^{N+1}   \frac{1}{2^{\epsilon_0 j} }  v(1-m(J_j)) \eta(f,J_j,p) \lesssim_{\epsilon_0}  \norm{f}_{*,\BMO_{v,p}} 2^{-\epsilon_0 N_I } .
\]
For \(\abs{a}\geq 1-R_{\BMO}/2 \), we have now obtained
\begin{align*}
 v(a) \gamma(f,a,p) &=   v(a)   \left(    \int_{\T} \abs{f(z) - f(a)}^p P_a(z) \, dm(z)  \right)^{\frac{1}{p}} \\
&\lesssim_{v,p}    \sum_{j=0}^{N_I}   \frac{1}{2^{\epsilon_0 j} }  v(1-m(J_j)) \eta(f,J_j,p)  +    \sum_{j=N_I}^{N+1}    \frac{1}{2^{\epsilon_0 j} }  v(1-m(J_j)) \eta(f,J_j,p)  \\
&\lesssim_{\epsilon_0}   \sup_{m(I) \leq  R_{\BMO}} v(1-m(I)) \eta(f,I,p)  +  \norm{f}_{*,\BMO_{v,p}} 2^{-\epsilon_0 N_I }.
\end{align*}
Using \( 2^{N_I}( 1-\abs{a}) = m(J_{N_I(a)}) \geq \frac{R_{\BMO}}{2}   \), we can conclude that for any \(R_{\BMO}\in]0,1[\) and \(a\in\D\) with \(\abs{a}\geq 1-R_{\BMO}/2 \), we have 
\[
 v(a) \gamma(f,a,p)  \lesssim_{v,p,\epsilon_0}     \sup_{m(I) \leq  R_{\BMO}} v(1-m(I)) \eta(f,I,p) + \norm{f}_{*,\BMO_{v,p}} \bigg(  \frac{R_{\BMO}}{2(1-\abs{a})}   \bigg)^{-\epsilon_0   }.  
\]
\end{proof}

Finally, the crucial ingredient for independency of \(p\) is a John-Nirenberg type result. We have the following, inspired by \cite{Spanne-1965}:
\begin{lem}[John-Nirenberg]\label{lem:JohnNirenberg}
Let \(0<R<1<M\) and \(f\in L^1(\T)\) with \(0< \norm{f}_{*,\BMO_{1,1}}<\infty\). For any arc \(I\subset \T\) with \(m(I)\leq R\) and \( \lambda >0 \), we have
\[
m(\{ w\in I  : \abs{   f(w) - m_I(f)    } > \lambda  \})  \leq m(I) \sqrt{M}e^{-\frac{ \lambda}{  \sup_{m(I)\leq R}  \eta(f,I,1)}    \frac{\ln M }{2M}  }.
\]
\end{lem}
The proof of Lemma \ref{lem:JohnNirenberg} is the same as in \cite[Theorem 4.1]{Girela-2001} (see also \cite[Theorem 2.1]{Garnett-1981}). Instead of considering the function \(f/ \norm{f}_{*,\BMO_{1,1}}\) as done in the references, one should fix \(0<R<1\) and change the supremum in the denominator to only include arcs \(I\subset\T\) with \(m(I)\leq R\).

We are now ready to prove Proposition \ref{prop:implicationOfJohnNirenberg}:
\begin{proof}[Proof of Proposition \ref{prop:implicationOfJohnNirenberg}]
We begin by proving
\begin{equation} 
\norm{f}_{*,\BMO_{v,p}} \lesssim_{v,p}  \norm{f}_{*,\BMO_{v,1}},      
\end{equation}
which together with Proposition \ref{prop:intBMOAimplyConformInvBMOA} and Proposition \ref{prop:ConformInvBMOAimplyintBMOA} implies the equivalence of norms and that \(\BMOA_{v,p} = \BMOA_{v,1}\) given the assumptions hold. Let \(f\in L^1(\T)\) with \( 0 <  \norm{f}_{*,\BMO_{v,1}} <\infty \). Fix \(I\subset \T \) and put \(R=m(I)\). Using Lemma \ref{lem:JohnNirenberg} for the first inequality and the fact that \(v\) is almost increasing, we get
\begin{align*}
v(1-R )^p \eta(f,I,p)^p & =   \frac{ v(1-R )^p }{m(I)} \int_0^\infty  m(\{ w\in I  : \abs{    f(w) - m_I(f)    }> \lambda  \})  \,   d\lambda^p \\
&\lesssim_p v(1-R )^p  \int_0^\infty    e^{- \frac{ \lambda }{    \sup_{m(I)\leq R}  \eta(f,I,1)}        }  \,   d\lambda^p \lesssim_p  v(1-R )^p  \left( \sup_{m(I)\leq R}  \eta(f,I,1)     \right)^p  \\
&\lesssim_v \sup_{m(I)\leq R}  v(1-m(I) )^p \eta(f,I,1)^p   \leq    \left( \norm{f}_{*,\BMO_{v,1}} \right)^p,
\end{align*}
and hence,
\[
\norm{f}_{*,\BMO_{v,p}}  \lesssim_{v,p}  \norm{f}_{*,\BMO_{v,1}}.  
\]
Moving on, given \(R_{\BMO}\in]0,1[\), we do the same calculations for any arc with \(m(I)\leq R_{\BMO}\) to obtain
\[
v(1-R ) \eta(f,I,p)  \lesssim_{v,p}   \sup_{m(I)\leq R}  v(1-m(I) )  \eta(f,I,1)    \leq  \sup_{m(I)\leq R_{\BMO}}  v(1-m(I) )  \eta(f,I,1),   
\] 
yielding
\[
\sup_{m(I)\leq R_{\BMO}} v(1-m(I) ) \eta(f,I,p) \asymp_{v,p} \sup_{m(I)\leq R_{\BMO}} v(1-m(I) ) \eta(f,I,1).
\]
Combining this with Proposition \ref{prop:ConformInvBMOAimplyintBMOA} proves \eqref{eq:Rineq}. Concerning \eqref{eq:R_BMOineq}, given any \(R_{\BMO}\in]0,1[\), pick \(R_A\in]0,R_{\BMO}/2]\) and apply \(\sup_{\abs{a}\geq 1-R_A} \) to both sides of \eqref{eq:lemR_BMOineq} in Proposition \ref{prop:intBMOAimplyConformInvBMOA} and we are done. Finally, \eqref{eq:VMOequvalentNorms} follows immediately from \eqref{eq:Rineq} and  \eqref{eq:R_BMOineq}.

\end{proof}

The following proposition is a generalization of \cite[Theorem 2.1]{Girela-2001}.

\begin{prop}\label{prop:ConvergenceInVMOAV}
Let \(1\leq p<\infty\), \(v\colon \D \to ]0,\infty[\) be radial and put \(q=\max\{2,p\}\). Assume \(v|_{[0,1[}\) is almost increasing and there is an \(\epsilon_0>0\) such that \(x\mapsto v|_{[0,1[}(1-x) x^{\frac{1}{q}-\epsilon_0} \) is almost increasing. Then, for a given \(f\in\HOLO(\D)\), the following are equivalent:
\begin{itemize}
\item \( f\in \VMOA_{v,p} \)
\item \( \lim_{c \to 1 }\norm{ T_c f - f}_{\BMO_v,p} = 0 \),
\item \(f\) belongs to the \(\BMOA_{v,p} \)-closure of analytic polynomials.
\end{itemize}
where the limit is taken arbitrary inside \(\closed{\D}\).
\end{prop}
\begin{proof}
First, for all \(c\in \D\) and \(f\in \BMOA_{v,p}\), the functions \( T_c f \in \BMOA_{v,p}  \) (see the remark right after \eqref{eq:LittlewoodPaley}). Furthermore, \( (T_c f - f)(0)= 0\). Let \(1\leq p<\infty\), \(q=\max\{2,p\}\) and assume \(v\) satisfies the assumptions and \(\epsilon_0 > 0\) is the \(\epsilon_0\) given in the statement. Then \(x\mapsto v|_{[0,1[}(1-x) x^{\frac{1}{q}-\epsilon_0} \) is almost increasing and bounded. By Proposition \ref{prop:implicationOfJohnNirenberg}, we have \(  \VMOA_{v,p}=\VMOA_{v} = \VMO_{v,p} = \VMO_{v,2} \) with equivalent norms. It is, therefore, sufficient to prove that the three properties are equivalent for \(p=2\). To this end, assume \( f\in \VMOA_{v} \). If \(c\in \T\), it follows from the proof of \cite[Theorem 2.1]{Girela-2001} that we can fix \(0<R\) small such that 
\begin{align*}
\sup_{m(I)\leq R} v(1-m(I)) \eta(f-T_c f,I,2) & \leq \sup_{m(I)\leq R} v(1-m(I)) \eta(f,I,2) +  \sup_{m(I)\leq R} v(1-m(I)) \eta(T_c f,I,2) \\
&= 2 \sup_{m(I)\leq R} v(1-m(I)) \eta(f,I,2) <  \epsilon,
\end{align*}
because \( f\in \VMO_{v,2} \). For any \(I\) such that  \(m(I)\geq R\), we can choose \(c\), independent of \(I\), close enough to \(1\) so that \(\norm{ T_c f - f }_{H^2} <\frac{ R^{\frac{1}{2}} \epsilon }{ v(1-R)} \). It follows that 
\[
\sup_{m(I)\geq R} v(1-m(I)) \eta(T_c f - f,I,2) \lesssim_v  \frac{v(1-R)}{R^{\frac{1}{2}}}  \norm{ T_c f - f }_{H^2}<\epsilon, 
\]
and hence,
\[
\norm{T_c f - f}_{\BMO_v,2}  \lesssim_v \epsilon
\]
when \(c\in \T\) is close to \(1\). Now, if \(c= r w_0\in \D , 0< r < 1 , w_0\in \T \), we first note that since \(f\in H^1\)
\[
\int_{\T} P_{cz} f \, dm = f(cz) = (T_c f)(z) ,\ c\in \D , z\in \closed{\D}.
\]
The reasoning used in \cite[Theorem 2.1]{Girela-2001} combined with Minkowski's inequality gives us that for every \(\delta>0\)   
\begin{equation*}
v(1-m(I))   \eta(  T_c f - f , I , 2 )  \lesssim \kern-1pt \sup_{\abs{\ARG w}<\delta}   \norm{T_{w_0 w} f - f}_{*,\BMO_{v,2} }  \kern-1pt + \int_{ \abs{\ARG w} \geq  \delta }   P_r(w)  \norm{ f}_{*,\BMO_{v,2} }   dm(w).
\end{equation*}
Choosing \(\delta>0\) small enough and \(w_0\) close to \(1\), the first term is less than \(\epsilon\). By choosing \(r\) close to \(1\), the second term is less than \(\epsilon\) and we have proved that \( f\in \VMOA_{v,p} \) implies \( \lim_{c \to 1 }\norm{ T_c f - f}_{\BMO_{v,2}} = 0 \), where the limit is taken arbitrary inside \(\closed{\D}\). Assuming  \( \lim_{c \to 1 }\norm{ T_c f - f}_{\BMO_{v,2}} = 0 \) an application of Proposition \ref{prop:implicationOfJohnNirenberg} yields
\[
\lim_{c \to 1 }\norm{ T_c f - f}_{\BMOA_{v}} = 0.
\]
By choosing \(c\in[0,1[\) close enough to \(1\), we have \( \norm{ T_c f - f}_{\BMOA_{v}}  < \epsilon \). Since \(T_c f \in \HOLO(\closed{\D})\), the function \((T_c f)' \in \HOLO(\closed{\D})\) can be approximated uniformly in \(\D\) by analytic polynomials, and hence, the derivative of an analytic polynomial, say \(\sup_{z\in\D} \abs{p_0'(z) - (T_c f)'(z) } < \epsilon \). Now, using formula \eqref{eq:LittlewoodPaley}, we have
\begin{align*}
\norm{ p_0 - (T_c f) }_{\BMOA_{v}}^2  \leq    \epsilon^2  (1+    \sup_{a\in \D} v(a) (1-\abs{a}^2))
\end{align*}
 and we can conclude that \(f\) belongs to the \(\BMOA_{v} \)-closure of analytic polynomials. Finally, any function in the  \(\BMOA_{v} \)-closure of the polynomials belong to \(\VMOA_{v} \) according to Proposition \ref{prop:polynomialsInBMOAV}.

\end{proof}

For the rest of the paper, it will be assumed \(v|_{[0,1[}\) is almost increasing and (strictly) positive. Using Proposition \ref{prop:ConformInvBMOAimplyintBMOA}, another interesting fact about \(\BMOA_{v,1}\) is the following result by Spanne, \cite[p.~594]{Spanne-1965}.
\begin{prop}\label{prop:SomeBMOAvLipCts}
If \(v|_{[0,1[}\) is increasing and \( \int_0^1 \frac{ dt }{ t\, v(1-t)  } \) is finite, then for every \(f\in \BMOA_{v,1}\), there exists a constant \(C_f\) (depending on \(f\)) and \(\theta_f>0\) such that for \(r<\theta_f\), it holds that
\[
\esssup_{\abs{t_1-t_2}<r}\abs{f(e^{it_1})  -   f(e^{it_2})   } \leq  C_f  \int_0^{r}   \frac{ dt }{ t\, v(1-t)  } .
\] 

\end{prop}
For analytic functions, the statement above is not evidently as close to an if and only if statement 
(compare with \cite[p.~594]{Spanne-1965}). However, under some additional constraints on \(v\), which will be present for the main results in this paper, the function defined in Lemma \ref{lem:GoodOne} could be used to prove a counterpart to Proposition \ref{prop:SomeBMOAvLipCts} (see Corollary \ref{cor:EvaluationBoundedOnBMOAv}). Recall that \(\BMOA_{v,1}\) only consists of constants if \(v(a) \gtrsim (1-\abs{a})^{-c} \) for any \(c>1\) (see e.g. \cite[Theorem 1.2]{Giaquinta-1983}).

\subsection{Some general Banach space theory}\label{ss:GeneralBanach}

\

Let \(X\) be a Banach space and \( T\in \BOP(X) \). 
\begin{itemize}
\item A series \(\sum_n x_n \subset X\) is weakly unconditionally Cauchy (\(\WUC\)) if \(\sum_n l(x_n)\) is unconditionally convergent, equivalently absolutely convergent, for all \(l\in X^*\). 
\item If for every infinite dimensional subspace \(M\subset X\) the operator \(T|_M\colon M \to T(M)\) is not an isomorphism, then the operator is said to be \emph{strictly singular} (also called Kato operator, \cite[1.9.2]{Pietsch-1980}).
\item If for \(\epsilon>0\), there exists \(N_\epsilon\geq 1\) such that for every subspace \(M\subset X\) with dimension greater than \(N_\epsilon\) , there is \(x\in \partial B_M\) such that \( \norm{Tx}_X\leq \epsilon \), then the operator is said to be \emph{finitely strictly singular} (this notion is used e.g. in \cite{XXX}).
\item Let \(M\) be a Banach space. The operator \(T\) \emph{fixes a copy of \(M\)} if there exists a closed subspace \(Y\subset X\) such that \(Y\simeq M\)  (isomorphic) and \(T|_Y\) is an isomorphism onto its image \(T(Y)\subset X\).
\item Let \(M\subset X\) be a subspace. The operator \(T\) is \(M\)-singular if it does not fix a copy of \(M\).
\item The operator \(T\) is \emph{unconditionally converging} if it maps \(\WUC\)- series to unconditionally convergent series.
\item The operator \(T\) is said to be \emph{completely continuous} if it maps weakly convergent sequences to norm convergent sequences.
\end{itemize}

For a normed space \(X\), the closed unit ball is given by \(B_X  :=  \{  f\in X :  \norm{f}_X\leq 1  \}     \).

The following lemma is found in \cite[C. II. Theorem 8.4']{PRZE-1968}. 
\begin{lem}\label{lem:UCfixC0}
Let \(X\) be a Banach space. If \(T\in \BOP(X)\) is not unconditionally converging, then it fixes a copy of \(c_0\). 
\end{lem}

The first statement in the lemma below is found e.g. in \cite[1.11]{Pietsch-1980} and the other follows more or less from the definitions.
\begin{lem}\label{lem:Structure}
Let \(X\) be a Banach space. If \(T\in \BOP(X)\) is weakly compact or completely continuous, then it does not fix a copy of \(c_0\). Moreover, if \(T\) is compact, it is both weakly compact and completely continuous.
\end{lem}

By the bounded inverse theorem, we have the following:
\begin{prop}
Let \(X\) be a Banach space and \( T\in \BOP(X) \). Then the following are equivalent:
\begin{itemize}
\item \(T\) is strictly singular.
\item For every \(\epsilon>0\) and every infinite dimensional subspace \(M\subset X\), there is \(x\in \partial B_M\) such that \( \norm{Tx}_X\leq \epsilon \). 
\end{itemize}
\end{prop}

For more information see, for example, \cite{Diestel-1984} and \cite{YYY}. Inspired by \cite{Leibov-1990} and \cite[Proposition 6]{Laitila-2013}, we have

\begin{lem}\label{lem:EquivBasisC0}
Let \(Y\) be a Banach space on which, for every \(a\in\D\), \(U(\cdot,a)\colon Y\to [0,\infty[\) is a complete norm yielding the evaluation maps bounded. Then the norm \(\norm{f} := \sup_{a\in\D} U(f,a)\) renders \(X\subset Y\) a Banach space for some subspace \(X\subset Y\).

Moreover, if \( (f_n) \subset X \) is a sequence with \( \norm{f_n}\asymp 1 \), \( \lim_{\abs{a}\to 1} U(f_n,a) = 0\) for all \(n\), and for all \(0<R<1\), it holds that  \(\lim_{n\to \infty} \sup_{\abs{a}\leq R}  U(f_n,a) = 0 \). Then, there is a subsequence \( (f_{n_k} ) \) equivalent to the standard basis for \(c_0\), and hence, the identity \(X\to X\) fixes a copy of \(c_0\).
\end{lem}
\begin{proof}
Let 
\[
X:=\{ f\in Y : \norm{f} < \infty  \}.
\]
Clearly \(\norm{\cdot}_X := \norm{\cdot} \) is a norm on \(X\). Let \(f_n\) be a Cauchy sequence in \(X\) with respect to \(\norm{\cdot}_X\). Since \( \norm{f}_X \geq U(f,a) \) for all \(f\in X\) and \((Y,U(\cdot,a))\) is complete, there is a limit  \(g\) in \(Y\). Since the evaluation maps are bounded, the limit, \(g\), is independent of \(a\in\D\). For all \(a\in \D\), we have
\[
U(g-f_n,a)  = \lim_{m\to\infty} U(f_m-f_n,a)  \leq   \lim_{m\to\infty} \norm{ f_m -f_n }_X, 
\] 
and since the right-hand side is independent of \(a\in \D\), we have proved that \(\lim_{n\to\infty} \norm{f_n- g}_X = 0\) and 
\( \norm{g}_X \leq \lim_{n\to \infty} \norm{f_n}_X <\infty \) since \((f_n)\) is Cauchy in \(X\).

For the second statement, applying the standard sliding hump technique to the assumptions (see for example \cite[Proof of Proposition 6]{Laitila-2013}) yields that there exists an increasing sequence \((r_k)\subset[0,1[\) and a subsequence \( (f_{n_k})\subset (f_n)\) such that
\begin{equation}\label{eq:slidingHump1}
\sup_{\abs{a}\leq r_k}  U(f_{n_k},a)  \leq  2^{-k}, \hspace{2cm}  \text{for all k},  
\end{equation}
\begin{equation}\label{eq:slidingHump2}
  \sup_{\abs{a} > r_{k+1}}  U(f_{n_k},a)  \leq  2^{-k}, \hspace{2cm}  \text{for all k}.  
\end{equation}

The sequences \( (r_k) \) and \( (f_{n_k}) \) are obtained, by first choosing e.g. \(r_1= \frac{1}{2}\) and then an element \(f_{n_1}\)  such that \eqref{eq:slidingHump1} is satisfied, which is possible due to  \(\lim_{n\to \infty} \sup_{\abs{a}\leq R}  U(f_n,a) = 0 \). After that, we apply the fact that \( \lim_{\abs{a}\to 1} U(f_{n_1},a) = 0\) to obtain \(r_2\) via \eqref{eq:slidingHump2} and so on.

Now, let \((t_k)\in \ell^\infty\). For every \(a\in\D\), there exists exactly one \(k_a\in \{0,1,2,\ldots\}\) such that 
\[
 a\in A(k_a) = \begin{cases}
]r_{k_a},r_{k_a+1}] , & k_a>0  \text{ and }\\
[0,\frac{1}{2}], & k_a = 0 . 
\end{cases}
\]
On the one hand, condition \eqref{eq:slidingHump1} tells us that for a fixed \(a\in \D\), it holds that \( U(f_{n_k},a) \leq 2^{-k} \), whenever, \(k>k_a\). On the other hand, condition \eqref{eq:slidingHump2} tells us that for a fixed \(a\in \D\), it holds that \( U(f_{n_k},a) \leq 2^{-k} \), whenever, \(k<k_a\). We can now conclude that for every \(K\)
\begin{align*}
\norm{ \sum_{k=1}^K  t_k f_{n_k}  }_X & \leq \norm{(t_k)}_\infty  \sum_{k=1}^K  \norm{f_{n_k}  }_X  = \norm{(t_k)}_\infty   \sup_{a\in\D}  \left( \sum_{k>k_a} + \sum_{k<k_a} +  \sum_{k=k_a}   \right)  U( f_{n_k},a) \\
& \leq \norm{(t_k)}_\infty (\sum_k 2^{-k}+\norm{f_{n_{k_a}}}_X) = \norm{(t_k)}_\infty (1+ \sup_n \norm{f_n}_X). 
\end{align*}

This is a characterization of \( (\sum_k f_{n_k})\) being a weakly unconditionally Cauchy series, which is equivalent to \(\sum_k t_k f_{n_k}\) converging for every \((t_k)\in c_0\). The sequence \( f_{n_k} \) is therefore, a basis, but not necessary Schauder. By the Bessaga-Pe\l czy\'nski selection principle, we can extract a subsequence \( (g_k) \subset (f_{n_k}) \), which is basic, and since \(wuC\)-property and \( \norm{f_{n_k}}_X \gtrsim 1  \) are inherited to (series of) subsequences, we have finally obtained a sequence \( (g_k) \subset (f_n) \), which is equivalent to the standard basis of \(c_0\).

\end{proof}

The following lemma is well known.
\begin{lem}\label{lem:dualOperatorProperties}
Let \(X\) be a Banach space and \(T\in\BOP(X)\). Then \(T^{**}\colon X^{**} \to  X^{**}  \) is weak\(^*\)-weak\(^*\) continuous and \(\iota^-1 T^{**}\iota = T\) on \(X\), where  \(\iota\colon  X \to X^{**} \) is the canonical embedding. 
\end{lem}

\begin{lem}\label{lem:dilationsOfBMOAvUniformlyBounded}
Let \(1\leq p<\infty\) and \(v\colon \D\to ]0,\infty[\) be a radial function. Then
\[
\sup_{c\in\D}\norm{T_c f}_{\BMOA_{v,p} } \leq  \norm{ f }_{\BMOA_{v,p}}.
\]
\end{lem}
\begin{proof}
Let \(c=rw_0, w_0\in \T, r\in [0,1[\). Using the fact that \(f(cz)  = \int_\T  (T_{w_0z} f) (w) P_r(w) \, dm(w)\), an application of Minkowski's inequality gives us
\begin{align*}
\norm{T_c f}_{*,v,p} &=     \sup_{a\in \D } v(a) \left(  \int_{\T} \abs{  f(cz) -   f(ca)     }^p  P_a(z)\,  dm(z)\right)^{\frac{1}{p}} \\
&\leq    \sup_{a\in \D } v(a)  \int_{\T}    \left(  \int_{\T} \abs{   (T_{w_0z} f) (w)  -   (T_{w_0a} f) (w)    }^p  P_a(z) \,  dm(z)\right)^{\frac{1}{p}} P_r(w) \, dm(w) \\
& \leq   \sup_{a,w\in \D } v(aw_0 w)   \left(  \int_{\T} \abs{    f (z)  -   f (aw_0w)    }^p  P_{aw_0w}(z) \,  dm(z)\right)^{\frac{1}{p}} \\
&= \norm{ f}_{*,v,p}.
\end{align*}
For the last inequality, we have approximated the integrand of the outer intergral by its supremum, followed by a variable substitution \(z\mapsto z\CONJ{w_0w}\), and the fact that \(v(a) = v(aw), w\in \T\).

\end{proof}

We are now ready to prove the following result, which does mainly rely on fundamental properties of weighted composition operators on Banach spaces of analytic functions and a certain duality relation. 

\begin{thm}\label{thm:Gantmacher}
Let \(v\) be an admissible weight. The spaces \( \VMOA_v^{**} \) and \( \BMOA_v \) are isomorphic. Moreover, if \(\phi\) and \(\psi\) are such that \(\WCO\in \BOP(\VMOA_v)\), then the domain of \(\WCO\) can be extended to \(\BMOA_v\) and the extension \( \WCO \colon \BMOA_v \to \BMOA_v \) is weakly compact if and only if \(\WCO|_{\VMOA_v}\) is weakly compact if and only if \(\WCO(\BMOA_v) \subset \VMOA_v\). 
\end{thm}

\begin{proof}
By Lemma \ref{lem:dilationsOfBMOAvUniformlyBounded}, we obtain
\[
\sup_{c\in \D} \norm{T_c f}_{\BMOA_v} \leq  \norm{f}_{\BMOA_v}, f\in \BMOA_v.
\]
Since  \( \lim_{r\to 1^-}\norm{T_r f- f}_{H^2} = 0\) \cite[Theorem 2.6]{Duren-1970} and \(T_r f \in \VMOA_v \) for every \(0<r<1\) according to the remark right after \eqref{eq:LittlewoodPaley}, we can use \cite[Theorem 2.2]{Perfekt-2013} with \(X=Y=H^2/\C\) and
\[
\mathcal L = \{ L_{\hat{\phi}} \colon f \mapsto  v(a) (f\circ \sigma_a - f(a))  :  \hat{\phi}\in\AUT \}.
\]
The result is that  \( (\VMOA'_v)^{**} \) and \( \BMOA'_v \) are isometrically isomorphic, where
\[
\VMOA'_v = \{ f\in \VMOA_v : f(0) = 0  \} \text{ and } \BMOA'_v = \{ f\in \BMOA_v : f(0) = 0  \}.
\]
Now it follows that
\[
\VMOA_v^{**}  = (\VMOA'_v \oplus_1 \C  )^{**} \cong  (\VMOA'_v )^{**} \oplus_1  \C \cong  \BMOA'_v \oplus_1  \C = \BMOA_v,
\]
where \(\cong\) stays for isometrically isomorphic and \(X\oplus_1 Y\) means the direct sum is equipped with the norm \( f\mapsto \norm{(\norm{f}_{X},\norm{f}_{Y})}_{\ell^1}\).

Theorem 2.2 in \cite{Perfekt-2013} also yields that the corresponding isometric isomorphism is an extension of the canonical mapping \(\iota\colon \VMOA_v \to (\VMOA_v)^{**}  \). Hereafter, let \(\iota\) be the extension. By the Banach-Alaoglu theorem \((B_{\VMOA_v^{**}}, w^*)\) is compact, and since \(\delta_z\in \VMOA_v^*\), we have that \(  \iota\colon  (B_{\BMOA_v}, \tau_0 ) \to  (B_{\VMOA_v^{**}}, w^* ) \) is a homeomorphism. Applying Lemma \ref{lem:dualOperatorProperties} to the weighted composition operator \(\WCO\in \BOP(\VMOA_v) \), we get that \(\iota^{-1}\WCO^{**}\iota |_{\VMOA_v} = \WCO \) and that \(\iota^{-1} T^{**}\iota |_{B_{\BMOA_v}}\) is \(\tau_0 - \tau_0\) continuous. Since \(\VMOA_v\) is \(\tau_0\) dense in \(\BMOA_v\) and \(\iota^{-1}\WCO^{**}\iota |_{\VMOA_v} = \WCO \), we obtain \(\iota^{-1} T^{**}\iota = \WCO\in \BOP(\BMOA_v)\).

Finally, Gantmacher's theorem yields the equivalences.

\end{proof}

Another equivalent statement of \(\WCO\in\BOP(\VMOA_v)\) being weakly compact is the following \cite[Theorem 3.2]{Perfekt-2015}: for every \(\epsilon>0\) there exists \(N>0\) such that
\[
\norm{\WCO f}_{\BMOA_v} \leq N\norm{f}_{H^2} + \epsilon \norm{f}_{\BMOA_v}  , \quad  f\in \VMOA_v.
\]
See \cite[Corollary 3.3]{Perfekt-2015} for a similar result concerning the operators induced by the same symbols acting on \(\BMOA_v\).

A final remark is that, assuming Theorem \ref{thm:MainThmCpt} holds for \( (\WCO\in \BOP(\VMOA_v) \),  Theorem \ref{thm:Gantmacher} allows us to immediately extend it to operators \( (\WCO\in \BOP(\BMOA_v) \) satisfying the extra assumption \(\WCO|_{\VMOA_v}\in \BOP(\VMOA_v)\). In this case, we can add  \(\WCO(\BMOA_v) \subset \VMOA_v\) to the list of characterizations of compactness given in  Theorem \ref{thm:MainThmCpt}. It is worth noticing that there are weighted composition operators in \(\BOP(\BMOA_v)\) that are not an extension of operators \(\WCO\in\BOP(\VMOA_v)\), in other words, there exists \(\WCO\in \BOP(\BMOA)\) such that \(\WCO|_{\VMOA_v}\notin\BOP(\VMOA)\). An example of an operator \(\WCO\in\BOP(\BMOA)\), which is compact due to Theorem \ref{thm:MainThmCpt}, but for which \(\WCO|_{\VMOA_v}\notin\BOP(\VMOA)\) is obtain using \(\psi= \PEF\in \BMOA_v\) (see Lemma \ref{lem:GoodOne} and its proof) and \(\phi\colon z\mapsto z/2 \).

\subsection{Some comments about the conditions concerning admissible weights} 

The functions \( g(z)= z^c  , \ c \geq 1/2\) satisfy all assumptions except the global growth restriction \ref{eq:mainAssump1}. The functions
\[
z\mapsto (e+z) ^{ c  e^{\cos( \ln(\ln (e+z) )   )}   } ,
\]
where \(0<c<\frac{1}{2e}\) belong to \(\HOLO(\C_{\Re \geq \frac{1}{2}})\) and satisfy \ref{eq:mainAssump1} (it's easy to see that its growth along the positive real line is bounded by \(z\mapsto z^{ce}\)), but not the curvature restriction \ref{eq:mainAssump2}. The intuition for this function follows from the fact that it can be written as an analytic staircase function \(F(z) = z + \cos z\) composed with an analytic magnifier \(\EXP\EXP \) from the left and its inverse \(\ln\ln \) from the right to create the effect of a rapidly growing size of the steps as \(z\) tends toward infinity along the positive real axis. Finally, the input is translated for the function to have the right domain, and compressed from its linear asymptotic mean growth to be dominated by a suitable root-type growth for condition \ref{eq:mainAssump1} to hold. We have,
\[
(e+z) ^{ c  e^{\cos( \ln(\ln (e+z) )   )}   } = \left(  \EXP \left(  \EXP(F(\ln(\ln (e+z))) \right) \right)^{c} .
\]
To see that the function does not satisfy \ref{eq:mainAssump2}, consider \(0<b=a^2\leq a\leq 1\) in condition \ref{eq:mainAssump2}. For condition \ref{eq:mainAssump2} to hold, it is necessary that (substituting \(a\mapsto \frac{1}{a}\)) 
\[
\sup_{a\in]1,\infty[}  \frac{ (e+a^2) ^{ c  e^{\cos( \ln(\ln (e+a^2) )   )}   }    }{ \Big(\left( (e+a)^2\right) ^{ c  e^{\cos( \ln(\ln (e+a) )   )}   }\Big)^2  } <\infty.
\] 
Moreover,
\[
 \frac{ (e+a^2) ^{ c  e^{\cos( \ln(\ln (e+a^2) )   )}   }    }{ \Big(\left( (e+a)^2\right) ^{ c  e^{\cos( \ln(\ln (e+a) )   )}   }\Big)^2  }  \approx \left( (e+a)^{2c} \right)^ {   e^{\cos( \ln(\ln (e+a^2) )   )}  -  e^{\cos( \ln(\ln (e+a) )   )}     }.
\]
For the above to explode as \(a\to \infty\) along some sequence, it is sufficient to prove that
\[
\limsup_{a\to 0}  \cos( \ln(\ln (e+a^2) )   ) - \cos( \ln(\ln (e+a) )   )  >0.
\]
This can be seen from
\[
 \cos( \ln(\ln (e+a^2) )   ) =  \cos( \ln(\ln (e+a ))  +   \ln\frac{\ln (e+a^2) }{ \ln (e+a)   }   ) \stackrel{ a\to \infty }{\sim} \cos (\ln 2 +   \ln(\ln (e+a ))  ).   
\]

\section{Some important lemmas}\label{sec:ImportantLemmas}

\begin{lem}\label{lem:coolStuff}
Let \(g \colon [1,\infty[\to  ]0,\infty[\) be an almost increasing function. 
We have
\[
\sup_{x,y\in ]0,1[}  \frac{  \left( g\frac{1}{1-y}   \right) }{   \left( g\frac{1}{1-x}   \right)   } \left(  \frac{   (  1-y^2  )(   1-x^2  )         }{    ( 1-xy  )^2   }           \right) \asymp  \sup_{t,x\in]0,1]} \frac{  t \,   g\left(\frac{1}{tx} \right)  }{   g\left(      \frac{1}{x}      \right)       } .  
\]
One sufficient condition for the quantity above to be finite is that \(x \mapsto x  \,  g(\frac{1}{x}) \) is almost increasing.   

\end{lem}
\begin{proof}

Some elementary calculations yield
\[
\sup_{x,y\in ]0,1[}  \frac{  \left( g\frac{1}{1-y}   \right) }{   \left( g\frac{1}{1-x}   \right)   } \left(  \frac{   (  1-y^2  )(   1-x^2  )         }{    ( 1-xy  )^2   }           \right)  \asymp \sup_{x,y\in]0,1[} \frac{g\left(  \frac{1}{y}   \right) }{    g\left(  \frac{1}{x}   \right)    } \frac{   x y         }{    (x+y)^2   }.
\]
Put \(y=tx\) to obtain
\begin{align*}
\sup_{x,y\in]0,1[} &\frac{g\left(  \frac{1}{y}   \right) }{    g\left(  \frac{1}{x}   \right)    }  \left(  \frac{   xy        }{    (x+y)^2   }           \right) = \sup_{x,t\in]0,\infty[} \chi_{0<x<1} \chi_{0<t<\frac{1}{x}} \frac{   g\left(\frac{1}{tx} \right)  }{   g\left(      \frac{1}{x}      \right)       }   \left(  \frac{   x^2t        }{    x^2(1+t)^2   }           \right)  \\
 & = \sup_{x,t\in]0,\infty[} \chi_{0<x<\min\{1,\frac{1}{t}\}}  \frac{   g\left(\frac{1}{tx} \right)  }{   g\left(      \frac{1}{x}      \right)       }   \left(  \frac{   t        }{    (1+t)^2   }           \right)  \\ 
& = \max\bigg\{  \sup_{t\in]0,1]} \sup_{x\in]0,1[}  \frac{   g\left(\frac{1}{tx} \right)  }{   g\left(      \frac{1}{x}      \right)       }   \left(  \frac{   t        }{    (1+t)^2   }           \right)  , \sup_{t\in]1,\infty[} \sup_{x\in]0,\infty[} \chi_{0<x<\frac{1}{t}}  \frac{   g\left(\frac{1}{tx} \right)  }{   g\left(      \frac{1}{x}      \right)       }   \left(  \frac{   t        }{    (1+t)^2   }           \right)  \bigg\}. 
\end{align*}
Since \(g\) is almost increasing, we have
\[
 \sup_{t\in]1,\infty[} \sup_{x\in]0,\infty[} \chi_{0<x<\frac{1}{t}}  \frac{   g\left(\frac{1}{tx} \right)  }{   g\left(      \frac{1}{x}      \right)       }   \left(  \frac{   t        }{    (1+t)^2   }           \right)  \asymp_g 1.
\]
Moreover,
\[
 \sup_{t\in]0,1]} \sup_{x\in]0,1[}  \frac{   g\left(\frac{1}{tx} \right)  }{   g\left(      \frac{1}{x}      \right)       }   \left(  \frac{   t        }{    (1+t)^2   }           \right)   \asymp \sup_{t,x\in]0,1]}  \frac{  t  \,  g\left(\frac{1}{tx} \right)  }{   g\left(      \frac{1}{x}      \right)       }.
\]
\end{proof}

The weighted Bloch space will only appear in this section and it is here defined as follows:
\[
\BLOCH_v = \{ f\in \HOLO(\D) : \norm{ f }_{\BLOCH_v}  := \abs{f(0)}+  \sup_{z\in\D} (1-\abs{z}^2)v(z) \abs{f'(z)} <\infty \}. 
\]

The following lemma is a trivial generalization of \cite[Corollary 5.2]{Girela-2001} (see also \cite[Lemma 5.1]{Girela-2001}). 
\begin{lem}\label{lem:BlochVContainBMOAV}
It holds that
\[
\BMOA_{v,1} \subset  \BLOCH_v
\]
and
\[
\norm{f}_ {\BLOCH_v}  \leq  \norm{f}_{\BMOA_{v,1}}, \ f\in H^1(\D).
\]
\end{lem}

\begin{proof}
For \(f\in H^1\) we have 
\[
\abs{f'(0)} \leq \norm{f}_{H^1}.
\]
Applying the above to \(f\circ\sigma_a - f(a)\) yields
\[
\abs{f'(a)} (1-\abs{a}^2) \leq \gamma(f,a,1). 
\]
Multiply both sides by \(v(a)\) and take the supremum over \(a\in\D\) to obtain the result (after adding \(\abs{f(0)}\) to both sides).
\end{proof}

The following Lemma provides a relation between the main condition \ref{eq:mainAssump1} and the weaker condition \ref{cond:Cond1prime} used in Lemma \ref{lem:GoodOne}. 
\begin{lem}\label{lem:ExtraInfoAboutG}
For an almost increasing function \(g:[1,\infty[ \to ]0,\infty[\) and \(p>q>0\), we have
\[
  \sup_{0<x<1}   x\, g(\frac{1}{x})^{p}   < \infty \ \Rightarrow  \  g\in L^q(]1,\infty[, d\arctan)  \ \Rightarrow \   \sup_{0<x<1}   x\, g(\frac{1}{x})^{q}   < \infty.
\]
\end{lem}
\begin{proof}
The first implication follows from
\[
\int_1^\infty  \frac{g(x)^q \, dx}{ 1+x^2 } =  \int_0^1  g(\frac{1}{x})^q \frac{  dx}{ 1+x^2 } =    \int_0^1  x^{-\frac{q}{p}} \left(   x g(\frac{1}{x})^p \right)^{\frac{q}{p}}  \frac{ dx}{ 1+x^2 }. 
\]
To prove the second implication, it is proved that for any almost increasing function \(f\) with \(\sup_{t\in]1,\infty[} t^{-1}f(t) =\infty\), we have \( f\notin L^1(]1,\infty[, d\arctan) \). Since \(\limsup_{t\to\infty}  t^{-1}f(t) =\infty \), there is an increasing sequence \( (t_n) \) tending to infinity such that \( n t_n \leq f(t_n)  \) for all \(n\geq 2\). Furthermore, let \(N\colon \mathbb N \to \mathbb N \) be an increasing function such that \(N(k)\leq k\) for all \(k\), \( \lim_k N(k) = \infty \) and \( \lim_k  \frac{t_{N(k)}}{ t_{k+1} }  = 0  \). Now, for large \(k\) (we can assume \(t_1\geq 1\)),
\begin{align*}
\int_1^\infty \frac{f(t)}{1+t^2} \, dt &\geq\sum_n \int_{t_n}^{t_{n+1}}  \frac{f(t)}{1+t^2} \, dt 
\gtrsim_f \sum_n \int_{t_n}^{t_{n+1}}  \frac{f(t_n)}{t^2} \, dt  
\geq \sum_{n=N(k)}^k  n t_n   \left( \frac{ 1}{t_n}  - \frac{ 1}{t_{n+1}} \right)  \\
&\geq  \sum_{n=N(k)}^k  \frac{n}{ t_{n+1} }  \left(  t_{n+1}  -  t_n \right)  \geq   \frac{N(k)}{ t_{k+1} }  \sum_{n=N(k)}^k   \left(  t_{n+1}  -  t_n \right)  
=   N(k)   \left( 1-   \frac{t_{N(k)}}{ t_{k+1} }  \right)
\end{align*}
and the statement follows from letting \(k\to \infty\).
\end{proof}

The function \(\PEF\) in the following Lemma is the starting point for constructing the test-function associated with \(\beta(\psi,\phi,a)\) (see Subsection \ref{subsec:BetaFunc}). 
\begin{lem}\label{lem:GoodOne}
Let \(g\in \HOLO(\C_{\Re\geq \frac{1}{2}})   \) such that \(g|_{[\frac{1}{2},\infty[} \) is (strictly) positive and almost increasing. Then, for \(v(z) \asymp  g(\frac{1}{1-\abs{z}}) \), all functions \(f\in \BMOA_{v,1}\) satisfy
\begin{equation}\label{eq:EvalEstimate}
\abs{ f(z) - f(0) } \lesssim_{v,g} \norm{f}_{1,*} \PEF(\abs{z}),
\end{equation}
where
\[
\PEF\colon  z\mapsto \int_0^z  \frac{ dt}{   (1-t) g(\frac{1}{1-t})   }.
\]
The constant is \(1\) if \(v(z)\geq g(\frac{1}{1-\abs{z}}) \). Assume also that

\setcounter{Genumi}{0}
\begin{enumerate} 
\GPrimeItem \(g\in L^2(]1,\infty[, d\arctan)  \),  \label{cond:Cond1prime} 
\GItem \( g(1/b) \lesssim g(a/b)  g(1/a)    \) for \(0<b\leq a<2\) and  \label{cond:Cond2}
\GItem \(\abs{g(z)} \gtrsim g(\abs{z}) , \ z\in \C_{\Re\geq\frac{1}{2}} \). \label{cond:Cond3}
\end{enumerate}
\end{lem}
Then \(\PEF_c\in \BMOA_{v}, \ c\in \D\) with a uniform bound for the norm with respect to \(c\), where \(\PEF_c(z) := \PEF(cz)\). Moreover, for a fixed \(c\in \D\) it holds that \( \gamma(\PEF_c,a,2)^2 \lesssim_{c,g} 1-\abs{a} , \ a\in \D\), implying  \(\PEF_c \in \VMOA_v\).
\begin{proof}
First, let \(f\in \BMOA_{v,1}\) and \(z\in\D\). By definition we have
\begin{align*}
\abs{f(z) - f(0) } &= \abs{\int_0^z f'(t)\, dt } = \abs{z\int_0^1 f'(tz)\, dt } \leq  \abs{z}  \int_0^1 \frac{   v(zt) (1-t^2\abs{z}^2) \abs{f'(tz)}   }{      v(zt) (1-t^2\abs{z}^2)    }\, dt  \\
&\leq \abs{z} \int_0^1 \frac{  \norm{f}_{1,*}  }{      v(zt) (1-t^2\abs{z}^2)    }\, dt \lesssim_{v,g} \norm{f}_{1,*}   \PEF(\abs{z}) ,
\end{align*}
where Lemma \ref{lem:BlochVContainBMOAV} gives the second last inequality.

Next, we prove that \(\norm{\PEF_c}_{\BMOA_v}\lesssim \norm{\PEF}_{\BMOA_v} < \infty\), where \(c\in \D \). First, put 
\begin{align*}
C:= \sup_{0<x<1}  x\, g(\frac{1}{x})^2   \text{ and } M:= \sup_{0<b\leq a<2}\frac{   g(1/b) }{        g(a/b)  g(1/a)   } .
\end{align*}
The constant \(C\) is finite due to \ref{cond:Cond1prime} and Lemma \ref{lem:ExtraInfoAboutG}, and \ref{cond:Cond2} yields that \(M\) is finite.
For \( 0<y\leq x<2\) we have
\begin{equation}\label{eq:aConst}
\frac{    x\, g(\frac{1}{x})^2   }{  y\, g(\frac{1}{y})^2     }  \stackrel{ \ref{cond:Cond2} }{\geq} \frac{    1    }{ M \frac{y}{x}\, g(\frac{x}{y})^2     } \geq \frac{1}{M C} >0.
\end{equation}
It follows that for all \(z,c\in \D, \) we have 
\[
\abs{1-c z} g(\frac{1}{  \abs{1-c z}   })^2 \geq   \frac{1}{M C} (1-\abs{z}) g(\frac{1}{1-\abs{z}})^2
\]
so that
\begin{equation}\label{eq:aConstZ}
\abs{\PEF'(c z)}^2 =\left( \abs{1-cz}^2 \abs{g(\frac{1}{1-cz})   }^2  \right)^{-1} \stackrel{\ref{cond:Cond3}}{ \lesssim_g } \frac{    MC  \left(  (1-\abs{z} ) g(\frac{1}{1-\abs{z} })^2  \right)^{-1}  }{     \abs{1-cz}     }.
\end{equation}
Moreover, for \(c,z\in\D\setminus \{0\}\),  
\[
\frac{    \abs{1-(c/\abs{c})z}  }{\abs{ 1-cz  }     } \leq    \frac{    \abs{1-cz}  }{\abs{ 1-cz  }     }  +    \frac{    \abs{cz(1-\frac{1}{\abs{c}})}  }{\abs{ 1-cz  }     }   \leq  1+    \frac{    \abs{z(\abs{c}-1)}  }{\abs{ 1-cz  }     } \leq 2.
\]

For \(a\in \D\) and \( c\in \D\setminus \{0\}\), \eqref{eq:LittlewoodPaley} together with the above estimates yield

\begin{align*}
\gamma(\PEF_c,a,2)^2 &\asymp \abs{c}^2 \int_{\D}  \abs{\PEF'(c z)}^2 (1-\abs{\sigma_a(z)}^2)  \, dA(z) \\
&  \stackrel{ \eqref{eq:aConstZ}  }{ \lesssim_g } \int_{\D } g\left(\frac{1}{1-\abs{z} }\right)^{-2}  \frac{ (1-\abs{a})}{  \abs{1- (c/\abs{c})z} \abs{ 1-\CONJ{a}z}^2   }  \,  dA(z)  \\
&  \leq \int_0^1 \int_0^{2\pi} g\left(\frac{1}{1-r }\right)^{-2}  \frac{  (1-\abs{a})  }{  \abs{1- re^{it}} \abs{ 1-\abs{a}re^{it}}^2   }   \, dt\, dr,
\end{align*}
where the last inequality is due to Hardy-Littlewood's inequality on rearrangements (see e.g. \cite[378, p.~278]{Hardy-1934}). If \(\abs{a} \leq  \frac{1}{2}\), 
\[
v(a)^2 \gamma(\PEF_c,a,2)^2 \lesssim_{v,g}  g(1)^{-2} < \infty.
\]
Henceforth, we assume \(\abs{a}> \frac{1}{2}\). It is evident that
\[
\sup_{a\in\D} v(a)^2\int_0^{\frac{1}{2}} \int_0^{2\pi} g\left(\frac{1}{1-r }\right)^{-2}  \frac{ (1-\abs{a})}{  \abs{1- re^{it}} \abs{ 1-\abs{a}re^{it}}^2   }   \, dt\, dr \lesssim_{g} \sup_{a\in\D}v(a)^2(1-\abs{a})
\]
and 
\[
\sup_{a\in \D } v(a)^2 \int_0^1  \int_{\frac{1}{2} }^{2\pi-\frac{1}{2}}  g\left(\frac{1}{1-r }\right)^{-2}  \frac{  (1-\abs{a})  }{  \abs{1- re^{it}} \abs{ 1-\abs{a}re^{it}}^2   }   \, dt\, dr  \lesssim  \sup_{\abs{a}\in \D } v(a)^2 (1-\abs{a}).
\]
By \ref{cond:Cond1prime} together with Lemma \ref{lem:ExtraInfoAboutG} \(\sup_{a\in\D}v(a)^2(1-\abs{a})<\infty\). The symmetry of the integrand with respect to \(t\) yields, it is sufficient to prove that
\[
\sup_{\frac{1}{2}<\abs{a}<1  } v(a)^2 \int_{\frac{1}{2}}^1  \int_0^{\frac{1}{2} }  g\left(\frac{1}{1-r }\right)^{-2}  \frac{  (1-\abs{a})  }{  \abs{1- re^{it}} \abs{ 1-\abs{a}re^{it}}^2   }   \, dt\, dr <\infty
\]
in order to establish \(\PEF_c\in \BMOA_v\) with uniformly bounded norm with respect to \(c\in \closed{\D}\). To this end, since \(\frac{1}{2}<r<1\), using the estimate \(\cos t \leq 1-t^2/3\) when \( 0<t<1/2\), we have

\begin{align*}
\int_0^{\frac{1}{2}}   \frac{dt}{    \abs{1- re^{it}} \abs{ 1-\abs{a}re^{it}}^2   } &= \int_0^{\frac{1}{2}}   \frac{dt}{     (1+r^2 - 2r\cos t)^\frac{1}{2}    ( 1+ \abs{a}^2 r^2  - 2 \abs{a} r \cos t )  } \\
&\lesssim \int_0^{\frac{1}{2}}  \frac{    dt}{    \left((1-r)^2 +  t^2 \right)^\frac{1}{2}    \left( (1- \abs{a} +\abs{a}(1- r) )^2  + (\sqrt{\abs{a}}t )^2  \right) },
\end{align*}
and we can therefore conclude that

\begin{align*}
 &\int_{\frac{1}{2}}^1  \int_0^{\frac{1}{2} }  g\left(\frac{1}{1-r }\right)^{-2}  \frac{  1  }{  \abs{1- re^{it}} \abs{ 1-\abs{a}re^{it}}^2   }   \, dt\, dr \\
& \lesssim \int_0^{\frac{1}{2}}   g\left(\frac{1}{r }\right)^{-2} \int_0^{\frac{1}{2}}  \frac{    dt}{    \left(r^2 +  t^2 \right)^\frac{1}{2}    \left( (1- \abs{a} +\abs{a}r )^2  + (\sqrt{\abs{a}}t )^2  \right) } \, dr \\ 
& \leq \int_0^{\frac{1}{2}} \int_0^{\frac{1}{2}}   g\left(\frac{1}{ r }\right)^{-2}  \frac{    1 }{    \left( r^2 +  t^2 \right)^\frac{1}{2}    \left( (1- \abs{a})^2 + \abs{a}^2 (r^2  + t ^2  ) \right) } \, dt \, dr.
\end{align*}
Now consider the two integrals as a representation of an area integral over a square in \(\R^2\). We can get a larger integration area by considering a quarter circle in the first quadrant with radius \(1\). Putting \(r^2+t^2 = R^2\) and \(r=R\cos \theta\), we obtain

\begin{align*}
 &\int_0^{\frac{1}{2}} \int_0^{\frac{1}{2}}   g\left(\frac{1}{ r }\right)^{-2}  \frac{    dt \,  dr }{    \left( r^2 +  t^2 \right)^\frac{1}{2}    \left( (1- \abs{a})^2 + \abs{a}^2 (r^2  + t ^2  ) \right) } \\
&=\frac{1}{\pi}   \int_0^1  \int_0^\frac{\pi}{2}    g\left(  \frac{1}{R\cos \theta}          \right)^{-2}      \frac{   d\theta  R \, dR    }{   R   \left( (1- \abs{a})^2 + \abs{a}^2 R^2  \right) }   \lesssim_g  \int_0^1    g\left(   \frac{1}{R}      \right)^{-2}      \frac{       dR    }{  (1- \abs{a})^2 + R^2    }  ,
\end{align*}
where it has been used that \(g\) is almost increasing, \(\cos \theta\leq 1\) and \(\abs{a}\geq 1/2\).

It remains to prove that
\begin{equation}\label{eq:finalStageEq1}
\sup_{ \frac{1}{2} <\abs{a}<1} g\left(\frac{1}{1-\abs{a}}\right)^2 \int_0^1    g\left(   \frac{1}{R}      \right)^{-2}      \frac{   (1-\abs{a})   \,  dR    }{  (1- \abs{a})^2 + R^2    } =   \sup_{  0 < b < \frac{1}{2} }  \int_0^1  \frac{  g\left(   \frac{1}{b}      \right)^2 }{   g\left(   \frac{1}{R}      \right)^2    }      \frac{   b   \,  dR    }{ b^2 + R^2    } 
\end{equation}
is finite. Now since \(g\) is almost increasing on \([2,\infty[\), we have
\[
b \geq R\ \Rightarrow \   g\left(\frac{1}{b}\right)^2  \lesssim
_g    g\left(\frac{1}{R}  \right)^2    
\]
which yields
\[
 \int_0^b    \frac{  g\left(\frac{1}{b}\right)^2  }{    g\left(\frac{1}{R}  \right)^2    }      \frac{    b \,  dR    }{   b^2 +  R^2   } \leq \int_0^b        \frac{    b \,  dR    }{   b^2 +  R^2   }  \leq \int_0^1        \frac{       dR    }{   1 +  R^2   }  =  \frac{\pi}{4} .
\]

Finally, since
\[
R > b\ \Rightarrow \      g\left(\frac{1}{b}\right) \lesssim_g  g\left(\frac{R}{b}\right) g\left(\frac{1}{R}\right),
\] 
we have

\[
 \int_b^1    \frac{  g\left(\frac{1}{b}\right)^2  }{    g\left(\frac{1}{R}  \right)^2    }      \frac{    b \,  dR    }{   b^2 +  R^2   } \lesssim   \int_b^1     g\left(\frac{R}{b}\right)^2      \frac{    b \,  dR    }{   b^2 +  R^2   }   \leq \int_1^\frac{1}{b}        \frac{    g(R)^2    \,  dR    }{   1 +  R^2   } \leq \norm{g}_{L^2(]1,\infty[,d\arctan)}^2<\infty.
\]

The remaining statement, that  \( \PEF_c \in \VMOA_v, \ c\in\D\),   follows from 
\begin{align*}
\gamma(\PEF_c,a,2)^2 &\asymp \abs{c}^2 \int_{\D}  \abs{\PEF'(c z)}^2 (1-\abs{\sigma_a(z)}^2)  \, dA(z) \lesssim  (1-\abs{a}) (\sup_{z\in\D}  \abs{\PEF'(c z)}^2 ) \int_{\D }  \frac{(1-\abs{z}^2)}{      \abs{ 1-\CONJ{a}z}^2   }  \,  dA(z) \\
&= (1-\abs{a}) (\sup_{z\in\D}  \abs{\PEF'(c z)}^2 ) \int_0^1  \frac{(1-r^2)}{      1-\abs{\CONJ{a}r}^2   }  \,   dr^2 \leq (1-\abs{a}) (\sup_{z\in\D}  \abs{\PEF'(c z)}^2 ),
\end{align*}
where the last equality is true, because \(\int_{\T} P_b \, dm = 1\) for every \(b\in\D\). We already concluded \(\sup_{a\in\D}v(a)^2(1-\abs{a})<\infty\) so we are done.

\end{proof}

In combination with Lemma \ref{lem:BlochVContainBMOAV}, we have the following corollary, telling us that the growth restriction of functions as approaching \(\T\) are the same for functions in \(\BLOCH_v\) and \(\BMOA_{v}\).

\begin{cor}\label{cor:EvaluationBoundedOnBMOAv}
Under the assumptions of Lemma \ref{lem:GoodOne} we have, for \(X=\BMOA_{v}\), \(X=\VMOA_{v}\) or \(X=\BLOCH_{v}\),   

\[
\norm{\delta_z}_{ X^* } \asymp_{v,g} (1+\PEF(\abs{z})) =  1 + \int_0^{\abs{z}}  \frac{ dt}{   (1-t) g(\frac{1}{1-t})   } \asymp 1 + \int_0^{\abs{z}}  \frac{ dt}{   (1-t) v(t)   } \lesssim  \ln \frac{e}{1-\abs{z}}.
\]
\end{cor}

Assuming \ref{cond:Cond2} in Lemma \ref{lem:GoodOne}, we have the following result
\begin{lem}\label{lem:compWithAutomorph}
Let \(g:[1,\infty[ \to ]0,\infty[\) be almost increasing and assume \( g(1/b) \lesssim g(a/b)  g(1/a)    \) for \(0<b\leq a\leq 1\). Then, for \(a\in \D\),
\[
g\left(\frac{1}{ 1-\abs{\sigma_a(z)}   } \right) \lesssim_g  g\left(\frac{ 1+ \abs{ a }   }{ 1- \abs{ a }   }    \right)  g\left(\frac{ 1 }{ 1-\abs{z} }    \right).
\]
\end{lem}
\begin{proof}
It holds that
\[
g\left(\frac{1}{ 1-\abs{\sigma_a(z)}   } \right) \lesssim_g g\left(\frac{1-\abs{z}}{ 1-\abs{\sigma_a(z)}   } \right)  g\left(\frac{ 1 }{ 1-\abs{z} }    \right).
\]
To finish the proof, we apply a result for automorphisms of the disk (see e.g. \cite[p.~48]{CowenMacCluer-1995}) and the fact that \(g \) is almost increasing. 

\end{proof}

The next result is an estimate for comparison of evaluation maps at \(z\) and \(\phi(z)\) for an analytic self-map \(\phi\). 
\begin{cor}\label{cor:EvaluationHierarchyOnBMOAv}
Let \(g:[1,\infty[ \to ]0,\infty[\) be almost increasing and assume \( g(1/b) \lesssim g(a/b)  g(1/a)    \) for \(0<b\leq a\leq 1\). If \(v(z)\asymp g(\frac{1}{1-\abs{z}} )\) and \(\phi\colon \D\to \D\) is an analytic self-map, we have
\[
 \norm{\delta_{ \phi(z) }     }_{\BMOA_{v,p}^*} \lesssim_{v,g}  \left( \max\bigg\{v\left(  \frac{2 \abs{\phi(0)}  }{1 + \abs{\phi(0)}   } \right),1\bigg\} +  \PEF(\abs{\phi(0)}) \right) \norm{\delta_{ z }     }_{\BMOA_{v,p}^*}.
\]
\end{cor}
\begin{proof}
Let \(z,a\in\D\). We have for \(1\leq p <\infty\),
\begin{align*}
\norm{  f\circ \sigma_a   }_{*,p} & =  \sup_{\hat{\phi} \in  \AUT } v((\sigma_a\circ\sigma_a\circ\hat{\phi})(0))  \norm{  f\circ \sigma_a\circ\hat{\phi} -   (f\circ \sigma_a\circ\hat{\phi})(0)     }_{H^p} \\
& =  \sup_{\hat{\phi} \in \AUT } v((\sigma_a\circ\hat{\phi})(0))  \norm{  f\circ \hat{\phi} -   (f\circ \hat{\phi})(0)     }_{H^p} \leq \left(  \sup_{\hat{\phi} \in \AUT } \frac{ v( \sigma_a(\hat{\phi}(0)) )   }{ v(\hat{\phi}(0))   }     \right) \norm{  f  }_{*,p}. 
\end{align*}
Applying Lemma \ref{lem:compWithAutomorph} and the fact that \(\frac{1+ \abs{ a }   }{1- \abs{ a }  } =  \frac{1}{1-\frac{2 \abs{a}  }{1 + \abs{ a }   }}\), it follows that
\begin{align*}
\norm{  f\circ \sigma_a   }_{*,p} &\lesssim_{v,g}   v\left(  \frac{2 \abs{a}  }{1 + \abs{ a }   } \right)  \norm{  f  }_{*,p}, \quad a\in \D, \ p\in [1,\infty[. 
\end{align*}

Estimate \ref{eq:EvalEstimate} in Lemma \ref{lem:GoodOne} yields the last inequality below:
\begin{equation}\label{eq:compWithAut}
\begin{split}
\abs{(f\circ \sigma_a)(z)} &\leq  ( \norm{  f\circ \sigma_a   }_{*,p} +\abs{f(a)} ) \norm{\delta_z}_{\BMOA_{v,p}^*}  \\
&\lesssim_{v,g}  \left( \norm{  f   }_{*,p} v\left(  \frac{2 \abs{a}  }{1 + \abs{ a }   } \right)  + \abs{f(0)} +\abs{f(a) - f(0)} \right) \norm{\delta_z}_{\BMOA_{v,p}^*} \\
& \lesssim_{v,g}  \left( \norm{  f   }_{\BMOA_{v,p}} \max\{v\left(  \frac{2 \abs{a}  }{1 + \abs{ a }   } \right),1\} + \norm{  f   }_{\BMOA_{v,p}} \PEF(\abs{a}) \right) \norm{\delta_z}_{\BMOA_{v,p}^*} .
\end{split}
\end{equation}
Let \(\phi\) be an analytic self-map of \(\D\). Then \(\sigma_{\phi(0)}\circ \phi \) is an analytic self-map of \(\D\) with zero as a fixed point. It follows from the Schwarz Lemma that  \( \abs{ \sigma_{\phi(0)}\circ \phi(z)  } \leq \abs{z} \) and by the maximum modulus principle, it follows that
\( \norm{  \smash{  \delta_{ (\sigma_{\phi(0)}\circ \phi )(z) }    }    }_{\BMOA_{v,p}^*} \lesssim_{v,g} \norm{\delta_{ z }     }_{\BMOA_{v,p}^*} \). From \eqref{eq:compWithAut} we now obtain

\begin{align*}
\abs{\delta_{\phi(z)} f} &= \abs{(f\circ \phi)(z)} = \abs{(f\circ \sigma_{\phi(0)})(\sigma_{\phi(0)}( \phi (z)))} \\
&\lesssim_{v,g}   \norm{  f   }_{\BMOA_{v,p}}  \left( \max\{v\left(  \frac{2 \abs{\phi(0)}  }{1 + \abs{\phi(0)}   } \right),1\} +  \PEF(\abs{\phi(0)}) \right) \norm{\delta_{ (\sigma_{\phi(0)}\circ \phi )(z) }     }_{\BMOA_{v,p}^*} \\
&\lesssim_{v,g} \norm{  f   }_{\BMOA_{v,p}}  \left( \max\{v\left(  \frac{2 \abs{\phi(0)}  }{1 + \abs{\phi(0)}   } \right),1\} +  \PEF(\abs{\phi(0)}) \right) \norm{\delta_{ z }     }_{\BMOA_{v,p}^*}
\end{align*}
and the result is proved.
\end{proof}

\section{The functions \(\alpha\) and \(\beta\)}\label{sec:FuncAlphaBeta}

It is time to introduce the test functions of unit norm \( \gAlpha_a \) and \(\gBeta_a\), for \(a\in \D\), which are a key ingredient in the proof of the two main results: Theorems \ref{thm:MainThm} and \ref{thm:MainThmCpt}. The conclusions in this section are made under suitable assumptions (see beginning of Section \ref{sec:WrappingItUp}).

\subsection{The function \(\alpha\)}

For \(\phi\ \colon \D\to \D, \ a\in\D\) and \(\psi\in\BMOA_v\) recall that
\[
\alpha(\psi,\phi,a) =  \frac{v(a)}{  v( \phi(a) )   }  \abs{\psi(a)} \norm{\phi_a}_{H^2},
\]
where \(\phi_a=\sigma_{\phi(a)}\circ \phi \circ \sigma_a\).

 For \( a\in \D \) define  
\[
\fAlpha \colon z\mapsto \frac{   \sigma_{\phi(a)}(z) -  \phi(a)  } {   v(\phi(a))    } , \ z\in \D
\]
and
\[
  \gAlpha_a\colon z\mapsto  \frac{ \fAlpha (z) }{  \norm{ \fAlpha }_{\BMOA_v}  }    , \ z\in \D.
\]
The important properties for these functions are that \(\gAlpha_a\in \VMOA_v, \lim_{\abs{\phi(a)}\to 1} \norm{ \smash{ \gAlpha_a } }_{H^2} = 0 \) and their relation with the function \(\alpha\),
\[
\alpha(\psi,\phi,a)   \lesssim v(a) \norm{ \psi(a) ( \gAlpha_a \circ \phi\circ \sigma_a - \gAlpha_a( \phi(a))   )   }_{H^2} ,\ a\in \D,
\]
which will be used to prove that boundedness of \(\WCO\) implies boundedness of \(\alpha\) and \(\beta\). Their relation with the function \(\alpha\) is also crucial when proving that \( \limsup_{\abs{a}\to 1} \alpha(\psi,\phi,a)>0 \) ensures \(\WCO\) is not \(c_0\)-singular.

Using \( f_a^{(2)}(z) = \sigma_{\phi(a)}(z) -  \phi(a) \), we have
\begin{align*}
\abs{f_a^{(2)}\circ \sigma_b(z) -  f_a^{(2)}(b)}^2 &=  \abs{\sigma_{\phi(a)}\circ \sigma_b(z) - \sigma_{\phi(a)}(b)}^2  =  \abs{ \frac{ \phi(a)\CONJ{b} - 1 }{ 1-\CONJ{\phi(a)}b  } \sigma_{\frac{ b-\phi(a) }{ 1- \phi(a)\CONJ{b} }  }(z) -   \frac{ \phi(a)\CONJ{b} - 1 }{ 1-\CONJ{\phi(a)}b  } \frac{ b-\phi(a) }{ 1- \phi(a)\CONJ{b} }     }^2  \\
&=  \abs{  \sigma_{\frac{ b-\phi(a) }{ 1- \phi(a)\CONJ{b} }  }(z) -  \frac{ b-\phi(a) }{ 1- \phi(a)\CONJ{b} }     }^2  =  \abs{ \frac{               z\left(\abs{   \frac{ b-\phi(a) }{ 1- \phi( a )\CONJ{b} }     }^2-1 \right)                 }{
    1-z\CONJ{     \frac{ b-\phi(a) }{ 1- \phi(a)\CONJ{b} }    }                 }   }^2.
\end{align*}
 Since \(\int_{\T} P_c \, dm = 1\) for every \(c\in\D\), we have
\begin{equation}\label{eq:gAlphaEquation}
\begin{split}
 v(b)^2 \gamma(\fAlpha,b,2)^2  &=   \frac{ v(b)^2  }{  v(\phi(a))^2   }     \left(  1 - \abs{   \frac{ b-\phi(a) }{ 1- \phi( a )\CONJ{b} }     }^2      \right)  \leq    \frac{v(b)^2 }{v(\phi(a))^2 }   \left(  \frac{   (  1-\abs{b}^2  )(   1-\abs{\phi(a)}^2  )         }{    \abs{ 1-\abs{ \phi( a )} \abs{b}   }^2   }           \right),
\end{split}
\end{equation}
and hence, by Lemma \ref{lem:coolStuff}
\begin{equation}\label{eq:coolStuff}
\sup_{b\in\D} \gamma(\fAlpha,b,2)^2 v(b)^2 \leq  \sup_{x,y\in[0,1[} \frac{v(x)^2  }{v(y)^2  }   \left(  \frac{   (  1-y^2  )(   1-x^2  )         }{    ( 1-xy  )^2   }           \right)   <\infty
\end{equation}
proving  \(\fAlpha\in \BMOA_v\) for all \(a\in \D\) and \(\sup_{a\in\D}\norm{  \smash{   \fAlpha   }   }_{\BMOA_v}  < \infty \). From \eqref{eq:gAlphaEquation} it also follows that   \(\fAlpha\in \VMOA_v\), and therefore also \(\gAlpha_a\in \VMOA_v\). Finally, from the equality in \eqref{eq:gAlphaEquation} it follows that \(1\leq \norm{  \smash{\fAlpha} }_{\BMOA_v} \) and using \(b=0\) in \eqref{eq:gAlphaEquation} now leads to \(\lim_{\abs{\phi(a)}\to 1} \norm{ \smash{ \gAlpha_a } }_{H^2} = 0\).

Next, we proceed to the function \(\beta\), which should not be confused with what is usually referred to as the \(\beta\)-function.

\subsection{The function \(\beta\)}\label{subsec:BetaFunc}

For \(\phi \colon \D\to \D, \ a\in\D\) and \(\psi\in\BMOA_v\), recall that

\[
\beta(\psi,\phi,a) =   \norm{  \delta_{\phi(a)}  }_{ (\BMOA_v)^* }   v(a)  \gamma(\psi,a,1).
\]

Corollary \ref{cor:EvaluationBoundedOnBMOAv} yields that

\[
\beta(\psi,\phi,a) \asymp_{v,g} \bigg(1 +  \int_0^{1}  \frac{\abs{\phi(a)} dt}{   (1-t\abs{\phi(a)}) g(\frac{1}{1-t\abs{\phi(a)}})   } \bigg)  v(a)  \gamma(\psi,a,1).
\]

For \( a\in \D \) define  
\[
\fBeta_a \colon z\mapsto   \int_0^{   z \CONJ{\phi(a) } }     \frac{   dt}{   (1-t  ) g(\frac{1}{1-t})   } =  \int_0^{   z }     \frac{ \CONJ{\phi(a)} \, dt}{   (1-t \CONJ{\phi(a)} ) g(\frac{1}{1-t \CONJ{\phi(a) } })   } , \ z\in \D
\]
and
\[
 \gBeta_a \colon z\mapsto  \frac{ ( 1+ \fBeta_a )^2 }{  \norm{ (1+\fBeta_a)^2 }_{\BMOA_v}  }    , \ z\in \D.
\]

The important properties for the functions \(\gBeta_a\) are that \(\gBeta_a\in \VMOA_v \) and their relation with the function \(\beta\),
\begin{equation}\label{eq:relationBeta}
\beta(\psi,\phi,a)   \lesssim_{v,g} v(a)   \abs{ \gBeta_a(\phi(a))    }   \gamma(\psi,a,1) ,\ a\in \D,
\end{equation}
which will be used in a similar fashion to the \(\alpha\)-case. Moreover, if \(\BMOA_v\not\subset H^\infty\), we will need \(\lim_{\abs{\phi(a)}\to 1} \norm{ \smash{ \gBeta_a } }_{H^1} = 0\) to hold.

Hereafter, \(X = \VMOA_v\) or \(X = \BMOA_v\). From the remark right after \eqref{eq:LittlewoodPaley}, it is clear that \(\gBeta_a\in \VMOA_v\), because \((\fBeta_a)^2 = S_{\CONJ{\phi(a)}} \PEF^2 \) (dilation of analytic function), where \(\abs{\phi(a)}<1\). 
Moreover, since \(g\) is almost increasing
\[
  (1-t\abs{\phi(a)}) g(\frac{1}{1-t\abs{\phi(a)}})  \gtrsim_g     (1-t\abs{\phi(a)}^2) g(\frac{1}{1-t\abs{\phi(a)}^2}),
\]
and hence,
\begin{equation}\label{eq:changingArgInh}
\begin{split}
1 + \int_0^1     \frac{   \abs{\phi(a)} \, dt}{   (1-t \abs{\phi(a)} ) g(\frac{1}{1-t \abs{\phi(a) } })   } &\lesssim_g  1+ \int_0^1     \frac{    \abs{\phi(a)}  \, dt}{   (1-t \abs{\phi(a)}^2 ) g(\frac{1}{1-t \abs{\phi(a) }^2 })   } \\
&\lesssim  1+ \int_0^1     \frac{  \abs{\phi(a)}^2 \, dt}{   (1-t \abs{\phi(a)}^2 ) g(\frac{1}{1-t \abs{\phi(a) }^2 })   }.
\end{split}
\end{equation}
from which an application of Corollary \ref{cor:EvaluationBoundedOnBMOAv} gives us

\begin{equation}\label{eq:lowerBoundForfBeta}
\norm{(1+\fBeta_a)^2 }_X \gtrsim_{v,g} \frac{  (1+\fBeta_a (\phi(a)))^2    }{ \norm{\delta_{\phi(a)}}_{X^*}   } \gtrsim_{v,g} \norm{\delta_{\phi(a)}}_{X^*}. 
\end{equation}

The inequality in equation \eqref{eq:aConstZ} in Lemma \ref{lem:GoodOne} followed by \eqref{eq:changingArgInh}, yields 
\[
 1+ \sup_{z\in\D} \abs{ \fBeta_a  ( z)  } \lesssim_g   1 +  \int_0^1     \frac{   \abs{\phi(a)} \, dt}{   (1-t \abs{\phi(a)} ) g(\frac{1}{1-t \abs{\phi(a) } })   }  \lesssim_{v,g}    \norm{\delta_{\abs{\phi(a)}^2}}.
\]
Similarly to the proof of Lemma \ref{lem:GoodOne}, for every \(a,b\in \D\) we have

\begin{align*}
v(a)^2 \gamma((1+\fBeta_a)^2,b,2)^2 &\asymp  v(a)^2 \abs{\phi(a)}^2 \int_{\D}  \abs{ 1+ \fBeta_a( z)  }^2\abs{\PEF'(\CONJ{\phi(a)} z)}^2 (1-\abs{\sigma_b(z)}^2) \, dA(z) \\
&   \lesssim_{v,g}   \norm{\delta_{\abs{\phi(a)}^2}}_{X^*}^2   v(a)^2 \int_{\D } g\left(\frac{1}{1-\abs{z} }\right)^{-2}  \frac{ (1-\abs{b})}{  \abs{1- z} \abs{ 1-\CONJ{b}z}^2   }  \,  dA(z) \\
& \lesssim_{v,g}    \norm{\delta_{\abs{\phi(a)}^2}}_{X^*}^2 \leq    \norm{\delta_{\abs{\phi(a)}}}_{X^*}^2.  
\end{align*}
Combining with \eqref{eq:lowerBoundForfBeta}, we can conclude that
\begin{equation}\label{eq:BoundsForfBeta}
\norm{(1+\fBeta_a)^2 }_X \asymp_{v,g}  \norm{\delta_{\phi(a)}}_{X^*}  \asymp_{v,g}  \norm{\delta_{\phi(a)^2}}_{X^*}. 
\end{equation}
As a consequence, \eqref{eq:relationBeta} holds. Finally, concerning \(\lim_{\abs{\phi(a)}\to 1} \norm{ \smash{ \gBeta_a } }_{H^1} = 0\), we can assume \(\lim_{{\phi(a)}\to 1}  \norm{\delta_{\phi(a)}}_{X^*} = \infty\), so it is sufficient to prove that \(\sup_{a\in\D}\norm{\fBeta_a}_{H^2}<\infty \). By the Littlewood-Paley identity followed by the inequality in equation \eqref{eq:aConstZ}, we have for \(a\in\D\),

\begin{align*}
\norm{\fBeta_a }_{H^2}^2 &\asymp \abs{\phi(a)}^2 \int_{\D}  \abs{     (1-z \CONJ{\phi(a)} ) g(\frac{1}{1-z \CONJ{\phi(a) } })    }^{-2}  (1-\abs{z}^2)  \, dA(z) \\
&\lesssim_g     \int_{\D}   \frac{    g(\frac{1}{1-\abs{z \CONJ{\phi(a) }  }})^{-2}     }{  \abs{   1-z \CONJ{\phi(a)} }   }  \, dA(z)  \lesssim_g  \sup_{a\in\D} \int_{\D}   \frac{    g(1)^{-2}     }{  \abs{   1-z \CONJ{\phi(a)} }   }  \, dA(z) <\infty
\end{align*}
and we are done.

The reason for using the parameter \(1\) in the factor \(\gamma(\psi,a,1)\) of the function \(\beta\) (compared to \(2\), which is used in e.g. \cite{Laitila-2009}) is due to the restrictive connection between the parameters \(p,q\) and \(v\) for the statements in Proposition \ref{prop:implicationOfJohnNirenberg} to hold, which is present via \ref{eq:mainAssump1}. On the one hand, if \(\gamma(\psi,a,2)\) is used in the definition of \(\beta\), the application of Hölder's inequality in e.g. \eqref{eq:TheReasonForDefBeta} (see also proof of Lemma \ref{lem:BetaBounded}) would create quantities involving the \(H^4\)-norm. In this case, to be able to apply Proposition \ref{prop:implicationOfJohnNirenberg}, we need a more restrictive condition instead of \ref{eq:mainAssumpA1}. On the other hand, it is essential to be able to connect the \(\gamma(\psi,a,1)\) factor in \(\beta\) with \(\gamma(\psi,a,2)\), for example, to obtain \eqref{eq:essNormUpperBound}. This is solved by the extensive Proposition \ref{prop:implicationOfJohnNirenberg}.

\section{Weighted composition operators on \(\BMOA_v\) and \(\VMOA_v\)}\label{sec:WrappingItUp}
Recall that an admissible weight \(v\), is a function satisfying: 

\noindent\MainAssumptionsTwo

With these assumptions, \(v|_{[0,1[}\) is almost increasing, and a version of equation \eqref{eq:aConst} in Lemma \ref{lem:GoodOne} shows that 
\[
x\mapsto v|_{[0,1[}(1-x) x^{\frac{1}{p}-\epsilon} 
\]
is almost increasing for some \(\epsilon>0\). Moreover, since \(v\) is equivalent to a radial function Proposition \ref{prop:implicationOfJohnNirenberg} yields that \(\BMOA_{v} = \BMOA_{v,1} \) with equivalent norms. Moreover, due to Lemma \ref{lem:ExtraInfoAboutG}, the assumptions of Lemmas \ref{lem:coolStuff} and \ref{lem:GoodOne} are satisfied and the useful properties obtained in Sections \ref{sec:ImportantLemmas} and \ref{sec:FuncAlphaBeta} hold. The results below, up to Corollary \ref{cor:BoundednessOfWCO_VMOA}, are inspired by \cite[Section 2 and Proposition 4.1]{Laitila-2009}.
 
Boundedness of \(\WCO\) on \(\BMOA_v\) and \(\VMOA_v\) is characterized in Theorem \ref{thm:BoundednessOfWCO} and Corollary \ref{cor:BoundednessOfWCO_VMOA} respectively. Another standard but important type of result is Theorem \ref{thm:fixCopyC0NotUC}.

Concerning the \(\VMOA_v\)-case, some Carleson measure theory (Proposition \ref{prop:PracticalEstimate}
) is sufficient, due to the nature of \(\VMOA_v\), to prove that the function-theoretic condition implies that \(\WCO\) is the uniform limit of a sequence of compact operators (\(\WCO\) composed with dilation operators), see Theorem \ref{thm:SuffForCptVMOA}. A candidate for a sufficient condition for \(\WCO\in\BOP(\BMOA_v)\) to be compact is presented in Theorem \ref{thm:SuffForCpt}. Although the structure of the space makes the Carleson measure approach less fruitful, the unit ball \(B_{\BMOA_v}\) is compact with respect \(\tau_0\). This allows another, classical, approach to be carried out, namely, to prove that \(\WCO\) maps \(\tau_0\)-null sequences to norm-null sequences. Due to some complications involving the weight \(v\), no characterization is achieved in the general case (see also Conjecture \ref{conj:MainBMOA}).

In all of the proofs to these results, Proposition \ref{prop:implicationOfJohnNirenberg} has a crucial part.

\bigskip

\begin{lem}[{\cite[Proposition 2.1]{Laitila-2009}}]\label{lem:laitilaLittlewood}
There is a constant \(C\geq 1\) such that
\[
\norm{f\circ u}_{H^2} \leq C \norm{f}_{H^2} \norm{u}_{H^2} 
\]
for all \(f\in H^2\) and analytic self-maps \(u\) of \( \D \) such that \(f(0) = u(0) = 0\).
\end{lem}

\begin{lem}\label{lem:WCO_Bounded} 
For \(f\in\BMOA_v\) and \(a\in \D\), we have
\begin{align*}
v(a) \gamma(\WCO f ,a,1 ) & \leq  \norm{f}_{\BMOA_v} \left(  \alpha(\psi,\phi,a)   +    v(a)   \norm{   \psi\circ\sigma_a -   \psi(a)   }_{H^2}          \frac{       \norm{\phi_a}_2  }{ v(\phi(a))  }  + \beta(\psi,\phi,a)   \right)
.
\end{align*}

\end{lem}

\begin{proof}
For a fixed \(a\in\D\) and \( f\in\BMOA_{v}\) we have
\begin{align*}
\gamma(\WCO f,a,1) &= \norm{ \psi\circ\sigma_a  f\circ\phi\circ\sigma_a - \psi\circ\sigma_a(0)  f\circ\phi\circ\sigma_a (0)   }_{H^1} \\
&\leq  \abs{\psi(a)}  \norm{f\circ\phi\circ\sigma_a -   f(\phi(a))  }_{H^1}  + \norm{  (  \psi\circ\sigma_a -   \psi(a)   )( f\circ\phi\circ\sigma_a -   f(\phi(a)) )  }_{H^1} \\
&\quad + \abs{f(\phi(a))} \norm{ \psi\circ\sigma_a - \psi(a)    }_{H^1}. 
\end{align*}
For the first term, we have
\begin{align*}
\norm{f\circ\phi\circ\sigma_a -   f(\phi(a))  }_{H^1} &\leq \norm{f\circ\phi\circ\sigma_a -   f(\phi(a))  }_{H^2}= \norm{f\circ\sigma_{\phi(a) }\circ \phi_a -   f(\phi(a))  }_{H^2} \\
&\stackrel{ \mathclap{ \text{Lemma } \ref{lem:laitilaLittlewood}  }  }{\leq} \quad \gamma(f,\phi(a),2) \norm{\phi_a}_{H^2} \leq \frac{      \norm{f}_*   \norm{\phi_a}_{H^2}  }{ v(\phi(a))  }.
\end{align*}

For the middle term, we apply Hölder's inequality and Lemma \ref{lem:laitilaLittlewood} to obtain 
\begin{equation}\label{eq:TheReasonForDefBeta}
\begin{split}
\norm{  (  \psi\circ\sigma_a -   \psi(a)   )( f\circ\phi\circ\sigma_a -   f(\phi(a)) )  }_{H^1}  &\leq  \norm{   \psi\circ\sigma_a -   \psi(a)   }_{H^2} \norm{  f\circ\phi\circ\sigma_a -   f(\phi(a))   }_{H^2}  \\
&\hspace{-5cm}   \leq  \norm{   \psi\circ\sigma_a -   \psi(a)   }_{H^2} \frac{      \norm{f}_*   \norm{\phi_a}_{H^2}  }{ v(\phi(a))  } 
\end{split}
\end{equation}

For the last term, we have

\begin{align*}
 \abs{f(\phi(a))} \norm{ \psi\circ\sigma_a - \psi(a)    }_1 \leq   \norm{\delta_{\phi(a)}}_{(\BMOA_{v})^*}  \norm{f}_{\BMOA_v} \norm{ \psi\circ\sigma_a - \psi(a)    }_{H^1}
\end{align*}
 and the statement follows.

\end{proof}

\begin{lem}\label{lem:AlphaBounded}
Let \(g\colon [1,\infty[\to ]0,\infty[ \) be almost increasing, \smash{\(x\mapsto x\, g(\frac{1}{x})\)} be almost increasing on \(]0,1]\) and \(v(z) \asymp  g(\frac{1}{1-\abs{z}}) \). Let \(X=\BMOA_v\) or \(X=\VMOA_v\). If \(\WCO\in \BOP(X)\), then
\[
 \alpha(\psi,\phi,a)    \lesssim_{v,g}    \norm{\WCO \gAlpha_a}_{\BMOA_v} +   \frac{ v(a) \norm{  (\psi\circ \sigma_a - \psi(a))  }_{H^2}  }{  v(\phi(a))  } \lesssim_v  \norm{\WCO}_{\BOP(X)} .
\]

\end{lem}
\begin{proof}
Invoking theory from Section \ref{sec:FuncAlphaBeta} and using
\[
v(\phi(a)) (\fAlpha  \circ \phi\circ \sigma_a - \fAlpha ( \phi(a))) = \sigma_{\phi(a)} \circ \phi\circ \sigma_a = \phi_a,
\]
we can conclude that

\begin{equation}\label{eq:AnUsefulEq}
\begin{split}
\alpha(\psi,\phi,a)  &= \frac{v(a)}{  v( \phi(a) )   }  \abs{\psi(a)} \norm{\phi_a}_{H^2} = v(a) \norm{ \psi(a) ( \fAlpha  \circ \phi\circ \sigma_a - \fAlpha ( \phi(a))   )   }_{H^2} \\
& \leq   v(a) \norm{ \psi\circ \sigma_a  \fAlpha  \circ \phi\circ \sigma_a  - \psi(a)   \fAlpha ( \phi(a))     }_{H^2} \\
&\quad + v(a) \norm{  (\psi\circ \sigma_a - \psi(a))  \fAlpha  \circ \phi\circ \sigma_a     }_{H^2}.
\end{split}
\end{equation}

For the first term, Lemma \ref{lem:coolStuff} yields:

\begin{align*}
v(a) \norm{ \psi\circ \sigma_a  \fAlpha  \circ \phi\circ \sigma_a  - \psi(a)   \fAlpha ( \phi(a))     }_{H^2} &\leq  \norm{\WCO \fAlpha }_{\BMOA_v} \leq   \norm{\WCO}_{\BOP(X)}  \norm{\fAlpha }_{\BMOA_v} \\
&\lesssim_{v,g}     \norm{\WCO}_{\BOP(X)}.
\end{align*}

For the second term, we have

\begin{align*}
 v(a) &\norm{  (\psi\circ \sigma_a - \psi(a))  \fAlpha  \circ \phi\circ \sigma_a     }_{H^2}  \leq  \norm{  \fAlpha  \circ \phi\circ \sigma_a     }_\infty  v(a) \norm{  (\psi\circ \sigma_a - \psi(a))  }_{H^2}\\
&\leq 2 \frac{ v(a) \norm{  (\psi\circ \sigma_a - \psi(a))  }_{H^2}  }{  v(\phi(a))  }  \leq 2 \frac{ \norm{\psi}_* }{  v(\phi(a))  }  \lesssim_v  \norm{\WCO 1}_*   \leq \norm{\WCO }_{\BOP(X)} .
\end{align*}

\end{proof}

\begin{lem}\label{lem:BetaBounded}
Let \(v\) be admissible and assume \(\WCO\in \BOP(X) \), where \(X=\BMOA_v\) or \(X=\VMOA_v\). Then
\begin{align*}
 \beta(\psi,\phi,a)   &\lesssim_{v,g} \norm{\WCO \gBeta_a}_{*,1} +   \alpha(\psi,\phi,a)   + \frac{ v(a) \norm{  \psi\circ\sigma_a - \psi(a)     }_{H^2}  \norm{\phi_a }_{H^2}       }{ v(\phi(a))   } \lesssim_{v,g}   \norm{\WCO }_{\BOP(X)}.
\end{align*}

\end{lem}
\begin{proof}
Theory from Section \ref{sec:FuncAlphaBeta}, Hölder's inequality and Lemma \ref{lem:laitilaLittlewood} yield
\begin{align*}
\beta(\psi,\phi,a)     &\lesssim_{v,g}   v(a) \norm{  \gBeta_a(\phi(a)) ( \psi\circ\sigma_a - \psi(a)   )      }_{H^1} \\
&= v(a) \norm{   (\psi\circ\sigma_a) \gBeta_a\circ\phi\circ\sigma_a - \psi(a)   \gBeta_a(\phi(a))       }_{H^1}\\ 
&\quad +v(a) \norm{ (\psi\circ\sigma_a - \psi(a))  (\gBeta_a(\phi(a))     -    \gBeta_a\circ\phi\circ\sigma_a)      }_{H^1}\\
&\quad + v(a) \norm{   \psi(a) (\gBeta_a(\phi(a))     -    \gBeta_a\circ\phi\circ\sigma_a)       }_{H^1}\\
&\leq  \norm{\WCO \gBeta_a}_{*,1}  + \bigg(  \frac{ v(a) \norm{  \psi\circ\sigma_a - \psi(a)     }_{H^2}  \norm{\phi_a }_{H^2} }{v(\phi(a))} +  \alpha(\psi,\phi,a) \bigg)  \norm{   \gBeta_a   }_{\BMOA_v}    \\
& \lesssim_{v,g} \norm{\WCO}_{\BOP(X)}  +  \frac{  \norm{ \WCO 1}_{\BMOA_v}        }{ v(0)   } + \alpha(\psi,\phi,a)  .
\end{align*}
Lemma \ref{lem:AlphaBounded} gives the statement. 

\end{proof}

\begin{thm}[Boundedness]\label{thm:BoundednessOfWCO}
For an admissible weight \(v\), we have
\[
\WCO\in \BOP(\BMOA_{v} ) \ \Longleftrightarrow \  \sup_{a\in\D} (\alpha(\psi,\phi,a) + \beta(\psi,\phi,a))<\infty.
\]

More specifically,
\[
\norm{\WCO}_{ \BOP(\BMOA_v) } \asymp_{v,g}    \abs{\psi(0)} \left( 1+h(\abs{\phi(0)}) \right) +    \sup_{a\in\D} \alpha(\psi,\phi,a) +   \sup_{a\in\D} \beta(\psi,\phi,a),
\]
where
\[
\PEF\colon  z\mapsto \int_0^z  \frac{ dt}{   (1-t) g(\frac{1}{1-t})   }.
\]

\end{thm}
\begin{proof}
Since \(\norm{\psi}_* \lesssim_{v,g} \norm{\psi}_{*,1} \leq \sup_{a\in\D}\beta(\psi,\phi,a)   \), Proposition \ref{prop:implicationOfJohnNirenberg} and Lemma \ref{lem:WCO_Bounded} yield
\begin{align*}
\sup_{\norm{f}_{\BMOA_v}\leq 1}  \sup_{a\in\D} v(a) \gamma(\WCO f ,a,2 ) &\lesssim_{v,g}   \sup_{a\in\D} \left(  \alpha(\psi,\phi,a)   +    \norm{   \psi  }_*             \frac{       \norm{\phi_a}_2  }{ v(\phi(a))  }  + \beta(\psi,\phi,a)   \right)\\
&\lesssim_{v,g}    \sup_{a\in\D} \alpha(\psi,\phi,a) + \sup_{a\in\D} \beta(\psi,\phi,a).
\end{align*}
Moreover,
\[
\sup_{\norm{f}_{\BMOA_v}\leq 1}  \abs{(\WCO f)(0)} \leq \abs{\psi(0)} \norm{ \delta_{\phi(0)} }_{\BMOA_v^*} \lesssim_{v,g} \abs{\psi(0)} \left( 1+h(\abs{\phi(0)}) \right)
\]
and we can conclude that
\[
\norm{\WCO}_{\BOP(\BMOA_v)} \lesssim_{v,g}    \abs{\psi(0)} \left( 1+h(\abs{\phi(0)}) \right) +    \sup_{a\in\D} \alpha(\psi,\phi,a) +    \sup_{a\in\D} \beta(\psi,\phi,a) . 
\]
For the lower estimate, we have 
\begin{align*}
 \abs{\psi(0)} \left( 1+h(\abs{\phi(0)}) \right) &\lesssim_{v,g}  \abs{\psi(0)} +   \abs{\psi(0)}  \gBeta_{0}( \phi(0)  )  \leq   \norm{\WCO 1 }_{\BMOA_v}  +  (\WCO \gBeta_{0}) (0)  \\
&\lesssim_{v,g}   \norm{\WCO }_{ \BOP(\BMOA_v) },
\end{align*}
after which Lemmas \ref{lem:AlphaBounded} and \ref{lem:BetaBounded} yield the lower bound for \(\norm{\WCO}_{ \BOP(\BMOA_v) } \).

\end{proof}

Notice that the condition given in Proposition \ref{prop:implicationOfJohnNirenberg} is sufficient to prove that
\[
\WCO\in \BOP(\BMOA_{v} ) \ \Longleftarrow \  \sup_{a\in\D} (\alpha(\psi,\phi,a) +  \sup_{a\in\D}  \beta(\psi,\phi,a))<\infty.
\]

The following corollary can be compared to \cite[Proposition 4.1]{Laitila-2009} with a slightly different proof.

\begin{cor}\label{cor:BoundednessOfWCO_VMOA}
For an admissible weight \(v\), the following statements are equivalent:
\begin{itemize}
\item \( \WCO\in \BOP(\VMOA_{v} ) \)
\item \( \sup_{a\in\D} (\alpha(\psi,\phi,a) + \beta(\psi,\phi,a))<\infty\) and \( \psi,\psi\phi\in\VMOA_v \)
\item \( \sup_{a\in\D} (\alpha(\psi,\phi,a) + \beta(\psi,\phi,a))<\infty\), \( \psi\in\VMOA_v \) and  \( \lim_{\abs{a}\to 1} v(a) \psi(a)  \gamma(\phi,a,2)  = 0\).
\end{itemize}

More specifically, if \(\WCO \colon \VMOA_v \to \VMOA_v \), then
\[
\norm{\WCO}_{\BOP(\VMOA_v)} \asymp_{v,g}    \abs{\psi(0)} \left( 1+h(\abs{\phi(0)}) \right) +    \sup_{a\in\D} \alpha(\psi,\phi,a) +   \sup_{a\in\D} \beta(\psi,\phi,a) . 
\]
\end{cor}
\begin{proof}
Assume first \(\WCO\in \BOP(\VMOA_{v} )\). According to Proposition \ref{prop:polynomialsInBMOAV} the maps \(z\mapsto 1\) and \(z\mapsto z\) belong to \(\VMOA_{v}\) and by assumption, we can conclude \(\psi,\psi\phi \in \VMOA_v\). By Lemmas \ref{lem:AlphaBounded} and \ref{lem:BetaBounded},we have \(\sup_{a\in\D} (\alpha(\psi,\phi,a) + \beta(\psi,\phi,a))<\infty\). On the other hand, assume  \(\psi,\psi\phi \in \VMOA_v\) and \(\sup_{a\in\D} (\alpha(\psi,\phi,a) + \beta(\psi,\phi,a))<\infty\). Theorem \ref{thm:BoundednessOfWCO} yields \(\WCO\colon \VMOA_v \to \BMOA_v\) is bounded. All that is left to prove is that the codomain is \(\VMOA_v\). Proposition \ref{prop:ConvergenceInVMOAV} yields it is sufficient to prove that any analytic polynomial is mapped into \(\VMOA_v\). To this end, let \(f\) be an analytic polynomial. We have,
\[
\gamma(\WCO f,a,2) \leq  \norm{ (\psi\circ \sigma_a - \psi(a) ) f\circ \phi\circ  \sigma_a  }_{H^2} +  \norm{  \psi(a) ( f\circ \phi\circ  \sigma_a - f(\phi (a))   ) }_{H^2}.
\]   
Since \(f\) is bounded, for the first term we have
 \[
 \norm{ (\psi\circ \sigma_a - \psi(a) ) f\circ \phi\circ  \sigma_a  }_{H^2} \leq \norm{f}_\infty \gamma(\psi,a,2),
\]
and since \(\psi\in\VMOA_v\)
\[
\lim_{\abs{a}\to 1}  v(a) \norm{ (\psi\circ \sigma_a - \psi(a) ) f\circ \phi\circ  \sigma_a  }_{H^2} = 0.
\]
For the second term, using Lemma \ref{lem:laitilaLittlewood},
 \[
v(a) \norm{  \psi(a) ( f\circ \phi\circ  \sigma_a - f(\phi (a))   ) }_{H^2}\leq \alpha(\psi,\phi,a)v(\phi(a)) \gamma(f,\phi(a),2).
\]
It is now sufficient to prove that for every sequence \( (a_n)\subset \D\) with \(\lim_n\abs{a_n}= 1\) there is a subsequence \((a_{n_k})\) such that 
\begin{equation}\label{eq:bddVMOAHelpEq}
\lim_k \alpha(\psi,\phi,a_{n_k})v(\phi(a_{n_k})) \norm{   f\circ  \sigma_{\phi(a_{n_k})} - f(\phi (a_{n_k}))   ) }_{H^2} = 0.
\end{equation}
Since \(\phi (a_n)\) is bounded, there is always a subsequence, either entirely in a compact subset of \(\D\) or that converges to a point on \(\T\). If \(\abs{\phi(a_{n_k})}\to 1\) as \(k\to\infty\), \eqref{eq:bddVMOAHelpEq} follows from \(\sup_{a\in\D}\alpha(\psi,\phi,a) <\infty\) and \(f\in \VMOA_v\). If \((\phi(a_{n_k}))\) is contained in a compact subset of \(\D\), we note that 
\[
\norm{\phi_a}_{H^2}\asymp \norm{\phi\circ\sigma_a - \phi(a)}_{H^2}.
\]
To conclude the proof, it is sufficient to prove that
\[
\lim_{\abs{a}\to 1} v(a) \psi(a) \norm{\phi\circ\sigma_a - \phi(a)}_{H^2} = 0.
\]
This follows from the fact that
\begin{align*}
\psi(a) \norm{\phi\circ\sigma_a - \phi(a)}_{H^2} &\leq  \norm{\psi\circ\sigma_a\phi\circ\sigma_a -\psi(a) \phi(a)}_{H^2} + \norm{ ( \psi\circ\sigma_a-\psi(a) )\phi\circ\sigma_a}_{H^2} \\
&\leq \gamma(\psi\phi,a,2) + \gamma(\psi,a,2)
\end{align*}
and \(\psi,\psi\phi\in\VMOA_v\).
\end{proof}

Concerning the proof above, the case where the denseness of polynomials is used is when \(\VMOA_v \not\subset H^\infty\), that is, when \( \int_0^1 \frac{ dt }{ t\, v(1-t)  } \) is infinite (see Proposition \ref{prop:SomeBMOAvLipCts}).

\subsection{Compactness and related properties for \(\WCO\) on \(\BMOA_v\) and \(\VMOA_v\)}\label{subsec:CptAndRelProp}

The test functions that are used in the following theorem can be found in Section \ref{sec:FuncAlphaBeta}.

\begin{thm}\label{thm:fixCopyC0NotUC}
Let \(\WCO\in \BOP(\VMOA_v) \) and \(v\) is an admissible weight (see beginning of Section \ref{sec:WrappingItUp}). If \( \limsup_{\abs{\phi(a)}\to 1} (\alpha(\psi,\phi,a) +\beta(\psi,\phi,a))>0\), then \(\WCO\) fixes a copy of \(c_0\). Moreover, in this case \( \WCO \) is not unconditionally converging. Furthermore, the same statements hold for \(\WCO\in \BOP(\BMOA_v) \) if at least one of the following holds: 
\begin{itemize}
\item \( \BMOA_v \not\subset H^\infty \),
\item \( \psi\in \VMOA_v \).
\end{itemize}
\end{thm}
\begin{proof}
Let \( (a_n) \subset \D \) be a sequence such that  \( \lim_n \abs{\phi(a_n)}\to 1\)  and at least one of the following holds: \( \lim_n \alpha(\psi,\phi,a_n) > 0 \) or \( \lim_n \beta(\psi,\phi,a_n)>0\). If \( \lim_n \alpha(\psi,\phi,a_n) >0\), then Lemma \ref{lem:AlphaBounded} yields 

\[
\lim_n \norm{\WCO \gAlpha_{a_n}}_{\BMOA_v} > 0. 
\]
We can, therefore, by going to a subsequence if necessary, assume
\[
\inf_n \norm{\WCO \gAlpha_{a_n}}_{\BMOA_v} > 0. 
\]
Now, one can apply Lemma \ref{lem:EquivBasisC0}, first for \( (\gAlpha_{a_n}) \) (recall \(  \smash{\norm{ \smash{   \gAlpha_{a_n}  }   }_{\BMOA_v} \!\!= 1}\) for all \(n\)) and then for  \( (\WCO \gAlpha_{a_{n_k}}) \) to obtain the statement of fixing a copy of \(c_0\), where \( (n_k) \) is the sequence of indices obtained after the first application of  Lemma \ref{lem:EquivBasisC0}.

\vphantom{\rule{0cm}{0.35cm} }Since \smash{\( (\sum_k \gAlpha_{a_{n'(k)}}) \)} is \(wuC\), where \( (n'(k)) \) is the sequence of indices obtained after the two applications of the Lemma, it is sufficient to prove that \vphantom{\rule{0cm}{0.35cm} }\smash{\( (\sum_k \WCO \gAlpha_{a_{n'(k)}}) \)} is not unconditionally convergent, but since the terms are bounded from below in norm it can't converge in norm and we are done.

Similarly, if \( \lim_n \alpha(\psi,\phi,a_n) = 0\) and \( \lim_n \beta(\psi,\phi,a_n)>0\), Lemma \ref{lem:BetaBounded} yields 

\[
\inf_n \norm{\WCO \gBeta_{a_n}}_{\BMOA_v} > 0.
\]
If  \( \BMOA_v \not\subset H^\infty \), we can apply Lemma \ref{lem:EquivBasisC0} and the statement follows in the same manner as above. Else \(\sup_{z\in\D} \norm{\delta_z}_{X^*}<\infty\), where \(X=\VMOA_v\) or \(X=\BMOA_v\). If  \(\WCO|_{\VMOA_v} \in  \BOP(\VMOA_v) \), it follows from Corollary \ref{cor:BoundednessOfWCO_VMOA} that \(\psi\in\VMOA_v\), in which case  \( \lim_n \beta(\psi,\phi,a_n)>0\) is impossible and we are done.
\end{proof}

The following proof is a standard procedure.
\begin{lem}\label{lem:Btau0Cpt}
For any weight \(v\) yielding the evaluation maps \(f\mapsto f(z), \ f\in\BMOA_v , z\in\D \) bounded, the norm-closed unit ball \(  B_{\BMOA_v} \) is \(\tau_0\)-compact.
\end{lem}
\begin{proof}
Since the evaluation maps are bounded, the unit ball is bounded with respect to \(\tau_0\) by the Banach-Steinhaus theorem. By Fatou's lemma, we have for any \(\tau_0\)-convergent sequence \( (f_n) \subset B_{\BMOA_v}  \) with an analytic function \(f\) as limit (defined on \(\D\)), 
\[
\norm{f}_{\BMOA_v}  \leq  \liminf_n \norm{ f_n }_{\BMOA_v}  \leq 1.
\]
This shows that \(  B_{\BMOA_v}  \) is \(\tau_0\)-closed in \(\HOLO(\D)\), because \((\HOLO(\D),\tau_0)\) is a Fréchet space. By Montel's theorem  \(  B_{\BMOA_v}  \) is \(\tau_0\)-compact.
\end{proof}

\begin{lem}\label{lem:HelpWithApplyingTheKnown}
Given that \(v\) is admissible, for any analytic self-map \(\phi\colon \D\to \D\), we have 
\[
\inf_{z\in\D} \frac{v(z)}{v(\phi(z)) } \gtrsim_{v,g}  v  \left(      \frac{  2\abs{\phi(0)}  }{ 1+\abs{\phi(0)} }       \right)^{-1}  > 0.
\]
Moreover, if \(\psi\in \BMOA_v\), then
\[
\WCO\in \BOP(\BMOA_v) \ \Longrightarrow \ \WCO\in \BOP(\BMOA),
\]
\[
\alpha(\psi,\phi,a)  \gtrsim_{v,g,\phi} \alpha(\psi,\phi,a) _{\BMOA}  \text{ and }  \beta(\psi,\phi,a)  \gtrsim_{v,g,\phi} \beta(\psi,\phi,a) _{\BMOA},
\]
where the \(\BMOA\) subscript stands for that \(v\equiv 1\) is used. 
\end{lem}
Notice that \(  \beta(\psi,\phi,a) _{\BMOA} \) is not the same as in \cite{Laitila-2009} and \cite{LaitilaLN-2023}.
\begin{proof}
Applying \cite[Corollary 2.40]{CowenMacCluer-1995}, we have
\[
1-\abs{\phi(z)} \geq (1-\abs{z})\frac{ 1-\abs{\phi(0)} }{  1+\abs{\phi(0)}  } ,  
\]
which gives (using \ref{eq:mainAssumpA2} with \(b = (1-\abs{z})\frac{ 1-\abs{\phi(0)} }{  1+\abs{\phi(0)}  } ,\  a = 1-\abs{z}  \))
\[
g\left(    \frac{  1  }{ 1-\abs{\phi(z)} }  \right)  \leq  g\left(    \frac{  1  }{ 1-\abs{z} }       \frac{  1+\abs{\phi(0)}  }{ 1-\abs{\phi(0)} }       \right)  \lesssim_g  g\left(    \frac{  1  }{ 1-\abs{z} }    \right)  g\left(      \frac{  1+\abs{\phi(0)}  }{ 1-\abs{\phi(0)} }       \right).
\] 
Since \( v(z) \asymp g\left(    \frac{  1  }{ 1-\abs{z} }  \right)\) the first statement follows. Now, we obtain
\[
 \alpha(\psi,\phi,a)  = \frac{v(a)}{  v( \phi(a) )   }  \abs{\psi(a)} \norm{\phi_a}_{H^2}  \gtrsim_{v,g,\phi}   \abs{\psi(a)} \norm{\phi_a}_{H^2} =\alpha(\psi,\phi,a) _{\BMOA} \quad \forall a\in \D.
\]
For the function \(\beta\), we apply Corollary \ref{cor:EvaluationBoundedOnBMOAv} to obtain

\begin{align*}
\norm{  \delta_{\phi(a)}  }_{ (\BMOA_v)^* } &\gtrsim_{v,g}   1+  \int_0^{ \abs{\phi(a)} }  \frac{ dt}{   (1-t)    } \frac{1}{  g(  \frac{1}{  1-\abs{\phi(a)}  })  }   \gtrsim_{v,g}    \frac{    \ln \frac{e}{1-\abs{\phi(a)} }     }{  v( \abs{\phi(a)} )  }  \gtrsim_{v,g} \frac{    \norm{  \delta_{\phi(a)}  }_{ \BMOA^* }   }{   v(a)      v    \left(      \frac{  2\abs{\phi(0)}  }{ 1+\abs{\phi(0)} }       \right)    }  ,  
\end{align*}

and hence,
\[
\norm{  \delta_{\phi(a)}  }_{ \BMOA_v^* }   v(a)  \gamma(\psi,a,1)   \gtrsim_{v,g,\phi}  \norm{  \delta_{\phi(a)}  }_{ \BMOA^* }    \gamma(\psi,a,1).
\]
Finally, apply Theorem \ref{thm:BoundednessOfWCO} twice, first for the weighted case, then for the unweighted case to obtain \(\WCO\in \BOP(\BMOA_v) \ \Longrightarrow \ \WCO\in \BOP(\BMOA)\).

\end{proof}

\begin{thm}[Sufficiency for compactness]\label{thm:SuffForCpt}
Assume \(v\) is admissible (see beginning of Section \ref{sec:WrappingItUp}) and that \(\WCO\in \BOP(\BMOA_{v} ) \). If in addition, at least one of the following holds:
\begin{enumerate}
\item \( C_{\phi} \in \BOP(\BMOA_v)\) or \label{eq:xtraCondOne} 
\item \(\WCO|_{\VMOA_v} \in \BOP(\VMOA_{v} )\),  
\end{enumerate}
then \(  \lim_{\abs{\phi(a)}  \to 1  } (\alpha(\psi,\phi,a) + \beta(\psi,\phi,a)) =0  \) is sufficient to ensure \(\WCO\) is compact on \(\BMOA_v\).
\end{thm}
\begin{proof}
If \(v\) is bounded, the result follows from \cite{Laitila-2009} and \cite{LaitilaLN-2023} (notice that \eqref{eq:xtraCondOne} is trivially true), so we can assume \(v\) is unbounded. The major part of the proof is similar to the second part of \cite[Proof of Theorem 3.1]{Laitila-2009}. Let \( (f_n) \) be a bounded sequence in \(\BMOA_v\), which converges to \(0\) with respect to \(\tau_0\) (converges uniformly on compact subsets of \( \D\)). By Corollary \ref{cor:EvaluationBoundedOnBMOAv} and Lemma \ref{lem:Btau0Cpt}, it follows that the unit ball \(B_{\BMOA_v}\) is \(\tau_0\)-compact. Since \( (f_n) \) is contained in a \(\tau_0\)-compact set, it is sufficient to prove that \( \lim_n \norm{ \WCO f_n }_{\BMOA_v}  = 0\) to obtain \(\WCO\) is compact. 

Similarly to \cite[(3.17)]{Laitila-2009}, we have for any \( r\in]0,1[ \)
\begin{equation}\label{eq:compactness}
\norm{ \WCO f_n }_{\BMOA_v}  \lesssim   \abs{  \psi(0) f_n(\phi(0))  } + \sup_{\abs{\phi(a)} > r} v(a) \gamma(\WCO f_n ,a,1 )  +   \sup_{\abs{\phi(a)}\leq r} v(a) \gamma(\WCO f_n ,a,1 ). 
\end{equation}
The first term converges to zero, because \( f_n\to 0 \) as \(n\to \infty\) w.r.t. \(\tau_0\). For the second term, Lemma \ref{lem:WCO_Bounded} gives us
\[
  \sup_{\abs{\phi(a)} > r} v(a) \gamma(\WCO f_n ,a,1 ) \leq  \norm{ f_n }_{\BMOA_v}  \sup_{\abs{\phi(a)} > r}   \left(  \alpha(\psi,\phi,a)   +    \norm{   \psi  }_*             \frac{       \norm{\phi_a}_{H^2}  }{ v(\phi(a))  }  + \beta(\psi,\phi,a)   \right),
\]
which yields
\[  
\lim_{r\to 1} \sup_{\abs{\phi(a)} > r} v(a) \gamma(\WCO f_n ,a,1 ) = 0
\]
by assumption.

Therefore, for a given \(\epsilon>0\), we can choose \(r\in ]0,1[\) large enough to ensure that
\[
\sup_n  \sup_{\abs{\phi(a)} > r} v(a) \gamma(\WCO f_n ,a,1 ) < \epsilon. 
\]
It remains to prove that for  \(r\in ]0,1[\) arbitrarily close to \(1\),
\[
\lim_{n\to \infty} \sup_{\abs{\phi(a)}\leq r} v(a) \gamma(\WCO f_n ,a,1 ) \lesssim \epsilon.
\]
To this end, for the third term in \eqref{eq:compactness}, we have

\begin{equation}\label{eq:compactness2}
\begin{split}
 \sup_{\abs{\phi(a)}\leq r} v(a) \gamma(\WCO f_n ,a,1 ) & \leq    \sup_{\abs{\phi(a)}\leq r} v(a) \norm{  \psi\circ\sigma_a ( f_n\circ\phi\circ\sigma_a -   f_n(\phi(a)) )  }_{H^1}  \\
&\quad +  \sup_{\abs{\phi(a)}\leq r} v(a) \abs{f_n(\phi(a))} \norm{ \psi\circ\sigma_a - \psi(a)    }_{H^1}.
\end{split}
\end{equation}
For the second term above (in \eqref{eq:compactness2}), we have
\[
 \sup_{\abs{\phi(a)}\leq r} v(a) \abs{f_n(\phi(a))} \norm{ \psi\circ\sigma_a - \psi(a)    }_{H^1}  \leq  \sup_{z\in r\closed{\D} } \abs{f_n(z)} \norm{ \psi    }_{\BMOA_{v}} \stackrel{  n\to \infty}{  \longrightarrow} 0,
\]
since \(\lim_n f_n = 0\) w.r.t. \(\tau_0\). Now, for \(a\in\D\) and \( t\in[0,1[\) we define
\[
F_{n,a} := f_n\circ\phi\circ\sigma_a -   f_n(\phi(a)) =  f_n\circ\sigma_{\phi(a)} \circ \phi_a -  f_n(\phi(a))
\]
and
\[
E = E(\phi,a,t) := \{  w \in \T : \abs{ \phi_a(w) } > t   \}.
\]
Following Laitila, let \(\frac{1}{2}<t<1\) and note that by \cite[(3.19)]{Laitila-2009}
\[
F_{n,a}(z) \leq 2\abs{\phi_a(z) } \sup_{\abs{ w }\leq t } \abs{   f_n\circ\sigma_{\phi(a)} (w) -  f_n(\phi(a))   }, \quad z\in\closed{\D}\text{ with } \phi_a(z)\in t\closed{\D}.
\]
It follows that for \(a\in \D\)
\[
\norm{  \chi_{\T\setminus E} \psi\circ\sigma_a F_{n,a}  }_{H^1} \lesssim    \frac{    \norm{\psi}_{*} +   v(\phi(a))  \alpha(\psi,\phi,a)   }{ v(a)  } \sup_{\abs{ w }\leq t } \abs{   f_n\circ\sigma_{\phi(a)} (w) -  f_n(\phi(a))   },
\]
and because \(\WCO\in\BOP(\BMOA_v)\) and
\[
\abs{\sigma_{\phi(a)} (w) }^2  = 1 -  \frac{  (1-\abs{\phi(a)}^2)   (1-\abs{w  }^2)   }{     \abs{ 1 - \CONJ{ \phi(a) } w       }^2   } \leq   1 -  \frac{  (1-r^2)   (1-t^2)   }{    4    }<1,
\]
for \( \abs{w}\leq t \) and \(\abs{\phi(a)}\leq r\) , an application of Theorem \ref{thm:BoundednessOfWCO} yields
\[
\lim_{n} \sup_{\abs{\phi(a)}\leq r}  v(a) \norm{  \chi_{\T\setminus E} \psi\circ\sigma_a F_{n,a}  }_{H^1} = 0,
\]
using \(\lim_n f_n = 0\) w.r.t. \(\tau_0\).

In view of \eqref{eq:compactness2}, all that remains to show is that we can, for \(0<r<1\) arbitrarily close to \(1\), choose \(0<t<1\) near \(1\) to ensure that 
\[
\lim_n  \sup_{\abs{\phi(a)}\leq r}  v(a)  \norm{ \chi_{E}  \psi\circ\sigma_a ( f_n\circ\phi\circ\sigma_a -   f_n(\phi(a)) )  }_{H^1} <\epsilon.
\]
Applying H\"olders inequality twice, we have
\begin{align*}
&\norm{ \chi_{E}  \psi\circ\sigma_a ( f_n\circ\phi\circ\sigma_a -   f_n(\phi(a)) )  }_{H^1} \leq \left(\norm{ \chi_{E}  \psi\circ\sigma_a  }_{H^1}  \right)^{\frac{1}{2}}  \left( \norm{ \chi_{E}  \psi\circ\sigma_a  F_{n,a}^2   }_{H^1} \right)^{\frac{1}{2}}   \\
& \hspace{2cm}    \leq \left(\norm{ \chi_{E}  \psi\circ\sigma_a  }_{H^1}  \right)^{\frac{1}{2}}  \left( \norm{  F_{n,a}   }_{H^2} \right)^{\frac{1}{2}}  \left( \norm{    \psi\circ\sigma_a  F_{n,a}   }_{H^2} \right)^{\frac{1}{2}} .  \\
\end{align*}
For the last factor, we have (note that \(\norm{\psi}_* = \norm{\WCO1}_* \leq   \norm{ \WCO }_{\BOP(\BMOA_{v})} \) )
\begin{align*}
\norm{   \psi\circ\sigma_a  F_{n,a}   }_{H^2} &\leq  \norm{    \psi\circ\sigma_a  f_n \circ \phi \circ \sigma_a  - \psi(a) f_n(\phi(a))   }_{H^2} +  \abs{   f_n(\phi(a) ) } \norm{   \psi\circ\sigma_a   -  \psi(a)   }_{H^2}  \\
&\leq   2\frac{  \norm{ \WCO }_{\BOP(\BMOA_{v})} }{ v(a) }   \sup_n \norm{   f_n }_{\BMOA_{v}}  \norm{ \delta_{\phi(a)} }_{(\BMOA_{v})^* }.  
\end{align*}
Next, we prove that there is a positive number \(M=M(v,\psi,\phi,(f_n),r,\epsilon)\) such that
\begin{equation}\label{eq:ExtraConditionInAction}
\sup_{ \abs{\phi(a)}\leq r }v(a) \norm{ \chi_{E}  \psi\circ\sigma_a  }_{H^1}   \norm{  F_{n,a}   }_{H^2}   \lesssim_{v,\psi,\phi,(f_n)} \epsilon + M \sup_{ \abs{\phi(a)}\leq r }  \norm{ \chi_{E}  \psi\circ\sigma_a  }_{H^1}  .
\end{equation}

Assuming \(C_{\phi}\in\BOP(\BMOA_v) \), Lemma \ref{lem:laitilaLittlewood} and Theorem \ref{thm:BoundednessOfWCO} implies
\begin{equation}\label{eq:ExtraConditionInActionTwo}
\norm{  F_{n,a}   }_{H^2}  \leq  \norm{  f_n\circ\sigma_{\phi(a)}  -  f_n(\phi(a))}_{H^2} \norm{ \phi_a }_{H^2} \leq  \frac{      \norm{ \phi_a }_{H^2}    }{  v(\phi(a)) }   \sup_n \norm{   f_n }_{\BMOA_{v}} \lesssim_{v,\phi,(f_n)} \frac{    1   }{  v(a) } 
\end{equation}
and \eqref{eq:ExtraConditionInAction} follows.

Assuming  \(\WCO|_{\VMOA_v} \in \BOP(\VMOA_{v} )\), the fact that \(\frac{1}{2}<t<1\) gives us
\begin{align*} 
v(a) \norm{ \chi_{E}  \psi\circ\sigma_a  }_{H^1}   &\lesssim  v(a) \norm{   \psi\circ\sigma_a  \phi_a }_{H^1} \leq  v(a) \norm{   (\psi\circ\sigma_a - \psi(a))   }_{H^2} +  v(a)  \psi(a) \norm{   \phi_a }_{H^2}
\end{align*}  
and by Corollary \ref{cor:BoundednessOfWCO_VMOA}, using 
\[
 \sup_{\abs{\phi(a)}\leq r}  \norm{   \phi_a }_{H^2} \lesssim_r  \sup_{\abs{\phi(a)}\leq r} \norm{  \phi\circ \sigma_a  - \phi(a)  }_{H^2}, 
\]
we have
\[
\sup_{\abs{a}\geq s} v(a) \norm{ \chi_{E}  \psi\circ\sigma_a  }_{H^1} <\epsilon 
\]
for some \(s=s(v,\psi,\phi,r,\epsilon)\in]0,1[\). 
Using the first two inequalities in \eqref{eq:ExtraConditionInActionTwo}, we can conclude that \( \sup_{a\in\D}\sup_n \norm{  F_{n,a}   }_{H^2} <\infty \), and hence,
\begin{align*}
\sup_{ \abs{\phi(a)}\leq r }v(a) \norm{ \chi_{E}  \psi\circ\sigma_a  }_{H^1}   \norm{  F_{n,a}   }_{H^2}  &\lesssim_{v,\psi,\phi,(f_n)} \epsilon +  \sup_{  \substack{ \abs{a}\leq s \\
 \abs{\phi(a)}\leq r }   } v(a) \norm{ \chi_{E}  \psi\circ\sigma_a  }_{H^1}  \\
 &\lesssim_{v} \epsilon +   v(s)  \sup_{ \abs{\phi(a)}\leq r }  \norm{ \chi_{E}  \psi\circ\sigma_a  }_{H^1}.
\end{align*}
Summing up, we have proved that
\[
\norm{ \chi_{E}  \psi\circ\sigma_a ( f_n\circ\phi\circ\sigma_a -   f_n(\phi(a)) )  }_{H^1} \lesssim_* \epsilon +  M\sup_{ \abs{\phi(a)}\leq r }  \norm{ \chi_{E}  \psi\circ\sigma_a  }_{H^1}, 
\]
where \(\lesssim_*\) depend on \(v,\psi,\phi\) and \((f_n)\), and  \(M=M(v,\psi,\phi,(f_n),r,\epsilon)\).

The final part of the proof is to obtain
\[
\lim_{t\to 1} \sup_{\abs{\phi(a)}\leq r} \norm{ \chi_{E}  \psi\circ\sigma_a  }_{H^2} = 0.
\]

Following the proof of \cite[Proof of Theorem 2.1]{LaitilaLN-2023}, we have that if \(\WCO\in\BOP( \BMOA)\) and 
\[ 
\lim_{\abs{\phi(a)}\to 1} \norm{ (\psi\circ \sigma_a ) \phi_a }_{H^2} = 0,
\]
then
\[
\lim_{t\to 1} \sup_{\abs{\phi(a)}\leq r} \norm{ \chi_{E}  \psi\circ\sigma_a  }_{H^2} = 0 \quad \forall r\in ]0,1[ .  
\]
By assumption \(\WCO\in \BOP(\BMOA_v)\), and by Lemma \ref{lem:HelpWithApplyingTheKnown} and Theorem \ref{thm:BoundednessOfWCO} (see also \cite[Theorem 2.1]{Laitila-2009}, where a slightly different definition is used for \(\beta\)) the functions \(\psi\) and \(\phi\) in \(\WCO\) give rise to a bounded operator \(\WCO\in \BOP(\BMOA)\) via the function-theoretic characterization. Furthermore, 
\[
\alpha(\psi,\phi,a) \geq    \frac{v(a) \norm{ (\psi\circ \sigma_a ) \phi_a }_{H^2}   }{v(\phi(a))}   -  \frac{   \norm{\psi}_{\BMOA_v}  }{v(\phi(a))} , \quad a\in\D
\]
and with the aid of  Lemma \ref{lem:HelpWithApplyingTheKnown} ,
\[
\alpha(\psi,\phi,a)  +  \frac{   \norm{\psi}_{\BMOA_v}  }{v(\phi(a))}   \gtrsim_{v,g,\phi}   \norm{ (\psi\circ \sigma_a ) \phi_a }_{H^2}.
\]
Using the fact that \(v\) is unbounded, we conclude
\[ 
\lim_{\abs{\phi(a)}\to 1} \norm{ (\psi\circ \sigma_a ) \phi_a }_{H^2} = 0.
\]

\end{proof}

By turning via Carleson measure theory, it is proved that \( \WCO\in \BOP(\VMOA_v) \) is compact given 
\[
\lim_{\abs{a}\to 1} (\alpha(\psi,\phi,a) + \beta(\psi,\phi,a)) = 0.
\]
We will use the following variant of Carleson sets with center \(z\), where \(\frac{1}{2} < \abs{z}<1\): 
\begin{equation}\label{eq:CarlesonSquare}
S(z): = \{ w\in\D : 0<1-\abs{w} <  2(1-\abs{z}) \text{ and } \abs{\ARG w - \ARG z}< 2\pi (1-\abs{z})  \}.
\end{equation}
This is not the standard definition, but these scaled, open sets serves the same purpose as the classical ones when working with the measure \(\mu_f\) in Proposition \ref{prop:PracticalEstimate}. It is clear that the longest euclidean distance from the center of the Carleson set is the distance to one of the corners away from the origin. Based on this, some elementary calculations yield
\begin{equation}\label{eq:PoissonKernelEstimate}
\inf_{w\in S(z)}\abs{ \sigma_z'(w) } \gtrsim  \frac{ 1 }{ 1-\abs{z} }, \quad \frac{1}{2} < \abs{z}<1.
\end{equation}

The following proposition is a straightforward generalization of \cite[Lemma 4.6]{Laitila-2009} and \cite[Lemma 3.3]{Garnett-1981}.

\begin{prop}\label{prop:PracticalEstimate}
Let \(v\) be admissible and
\[
d\mu_f(z) := \abs{f'(z)}^2 (1-\abs{z}^2)  \, dA(z).
\]
We have
\[
\norm{f}_{*,v}^2 \asymp \sup_{\abs{z}\in ]\frac{1}{2},1[} \frac{ v(z)^2  }{ 1-\abs{z} } \mu_f(S(z)). 
\]
Moreover,

\begin{equation}\label{eq:PracticalEstimate}
\norm{f}_{\BMOA_v} \lesssim_v  \frac{ 1 +   \sup_{z\in]0,r[} v(z)  }{(1-r)^{\frac{5}{2}}} \sup_{\abs{z}\leq r} \abs{f(z)} + \sup_{\abs{a}\geq r} v(a) \gamma(f,a,2).
\end{equation}

\end{prop}

\begin{proof}
Let \(z\in \D\)  with \(\abs{z} >\frac{1}{2}\). On the one hand, we have by \eqref{eq:PoissonKernelEstimate} and \eqref{eq:LittlewoodPaley}
\begin{equation}\label{eq:fBMOAvIsCarleson}
\begin{split}
\frac{ v(z)^2  }{ 1-\abs{z} } \mu_f (S(z)) &\lesssim  v(z)^2   \int_{S(z)} \abs{ \sigma_z'(w) }   \abs{f'(w)}^2 (1-\abs{w}^2)  \, dA(w) \\
&\lesssim  v(z)^2   \gamma(f,z,2)^2.  
\end{split}
\end{equation}
On the other hand, if we put 
\[
E_n := E_n(z) := \bigg\{ w\in \D : \abs{w-\frac{z}{\abs{z}}} < 2^n(1-\abs{z}) \bigg\},
\] 
and
\[
N_z := \max\{ n\in \N : 2^n ( 1-\abs{z}) < 1  \}
\]
we obtain for \(n = 1 ,2 ,\ldots , N_z\)
\begin{equation}\label{eq:EnInSn} 
S_n :=  S\Big(\frac{z}{\abs{z}}\Big(1-2^{n-1} ( 1-\abs{z})\Big) \Big) \supset E_n,
\end{equation}
and therefore, 
\[
v\Big(1- 2^{n-1} ( 1-\abs{z})\Big)^2 \mu_f( E_n ) \leq \bigg(    \sup_{ \abs{\zeta}\in ]\frac{1}{2},1[ }  \frac{ v(\zeta)^2  }{ 1-\abs{\zeta} } \mu_f (S(\zeta))   \bigg)    2^{n-1} ( 1-\abs{z}).
\]

We have \( \sup_{w\in \D}  \abs{ \sigma_z'(w) } \lesssim (1-\abs{z})^{-1}\) and for \(n\geq 2\) and \(\abs{z}\in]\frac{1}{2},1[\), we have
\begin{align*}
 \sup_{w\in E_n\setminus E_{n-1}} \abs{ \sigma_z'(w) }  & \lesssim \frac{2^{-2n}}{1-\abs{z}} \quad \text{ and } \quad   \sup_{w\in \D\setminus E_{N_z}} \abs{ \sigma_z'(w) }   \lesssim 1-\abs{z}.
\end{align*}
Combining the above estimate with \eqref{eq:LittlewoodPaley} after which the path of integration, \(\D\), is partitioned into \(E_1,E_n\setminus E_{n-1},n=2,\ldots, N_z\) and \(\D\setminus E_{N_z}\)  yield
\begin{align*}
&v(z)^2   \gamma(f,z,2)^2  \lesssim   \frac{v(z)^2  }{ 1-\abs{z} }  \mu_f(E_1)  + \sum_{n=2}^{N_z}    v(z)^2     \frac{ 2^{-2n}   }{  1-\abs{z} }  \mu_f(  E_n\setminus E_{n-1} ) +   v(z)^2  (1-\abs{z}) \mu_f( \D\setminus E_{N_z} )    \\
&\stackrel{ \eqref{eq:EnInSn} }{\lesssim}     
  \bigg(    \sup_{ \abs{\zeta}\in ]\frac{1}{2},1[ }  \frac{ v(\zeta)^2  }{ 1-\abs{\zeta} } \mu_f (S(\zeta))   \bigg) \bigg(  1  +    \sum_{n=2}^{N_z}  \frac{  v(z)^2   2^{-n} }{    v\Big(1- 2^{n-1} ( 1-\abs{z})\Big)^2 }   \bigg)   +   v(z)^2  (1-\abs{z}) \mu_f( \D) .
\end{align*} 
Clearly, for \(z\in\D\) with \(\abs{z}>1/2\),
\[
 v(z)^2  (1-\abs{z}) \mu_f( \D) \lesssim   \bigg(    \sup_{ \abs{\zeta}\in ]\frac{1}{2},1[ }  \frac{ v(\zeta)^2  }{ 1-\abs{\zeta} } \mu_f (S(\zeta))   \bigg) 
\]
and using \ref{eq:mainAssumpA2}, we have
\[
\frac{  v(z)^2   2^{-n} }{    v\Big(1- 2^{n-1}  ( 1-\abs{z})\Big)^2 }   \lesssim_{v,g} 
g\big(2^n\big)^2 2^{-n},  \quad n\geq 2.
\]
To finish the proof of 
\[
\norm{f}_{*,v}^2 \asymp_{v,g} \sup_{\abs{z}\in ]\frac{1}{2},1[} \frac{ v(z)^2  }{ 1-\abs{z} } \mu_f(S(z)), 
\]
we note that
\[
 \sum_{n=3}^{\infty}  g\big( 2^n \big)^2 2^{-n}  \lesssim_g \int_1^\infty  \frac{  g(e^x)^2 }{ e^x } \, dx  \asymp  \int_e^\infty g(x)^2 \, d\arctan(x),
\]
where the substitution \(x\mapsto \ln x\) and  \(\frac{dx}{x^2} \asymp  d\arctan(x) \) have been applied. Lemma \ref{lem:ExtraInfoAboutG} and \ref{eq:mainAssumpA1} yield that \( \int_e^\infty g^2 \, d\arctan \) is finite. 

To prove the last statement, let \(r\in]\frac{1}{2},1[\) and consider
\[
\norm{f}_{\BMOA_v} \lesssim \abs{f(0)}  +  \sqrt{ \sup_{\abs{z}\in ]r,1[} \frac{ v(z)^2  }{ 1-\abs{z} } \mu_f(S(z))  }  +  \sqrt{ \sup_{\abs{z}\in ]\frac{1}{2},r]} \frac{ v(z)^2  }{ 1-\abs{z} } \mu_f(S(z))  }.
\]
Equation \eqref{eq:fBMOAvIsCarleson} proves that 
\[
 \sup_{\abs{z}\in ]r,1[} \frac{ v(z)^2  }{ 1-\abs{z} } \mu_f(S(z))    \lesssim  \sup_{\abs{z}\in ]r,1[} v(z)^2   \gamma(f,z,2)^2.
\]
For a given \(z\in]\frac{1}{2},r[\), \(S(z)  \setminus \closed{r\D} \) can be covered by a \( N = N(r,z)= \inf \{n\in\mathbb N : \frac{ 1 - \abs{z} }{ 1 - r } \leq  n \}+1 \) number of Carleson sets, \( (S(z_j))_{j=1}^N\), where \(\abs{z_j}=r\) for all \(j\). It follows that for some \(j_0\in[1,N]\),
\begin{align*}
\frac{ v(z)^2  }{ 1-\abs{z} } \mu_f (S(z))   &\leq \frac{ v(z)^2  }{ 1-\abs{z} }  \sum_{j=1}^N  \mu_f (S(z_j))    + \frac{ v(z)^2  }{ 1-\abs{z} } \mu_f ( \closed{r\D}  )    \\
&\leq \frac{ v(z)^2  }{ 1-\abs{z} }  N \mu_f (S(z_{j_0}))    + \frac{ \sup_{z\in]0,r[} v(z)^2  }{ 1-r } \int_{ \closed{r\D} }   \abs{f'(w)}^2 (1-\abs{w}^2)  \, dA(w)  \\
&\leq \frac{ v(z)^2 (1-r) }{v(r)^2 1-\abs{z} }  N \frac{ v(r)^2 }{1-r} \mu_f (S(z_{j_0}))  +     \sup_{\abs{z}\leq r}\abs{f'(z)}^2  \frac{ \sup_{z\in]0,r[} v(z)^2  }{ 1-r }  \\
&\lesssim_v    \sup_{\abs{z}\in ]r,1[} \frac{ v(z)^2  }{ 1-\abs{z} } \mu_f(S(z))   +  \sup_{\abs{z}\leq r}\frac{ \abs{f(z)}^2 }{ (1-r)^4 }  \frac{ \sup_{z\in]0,r[} v(z)^2  }{ 1-r }  ,
\end{align*} 
where the Cauchy formula yields the comparison estimate between \(f'\) and \(f\).
\end{proof}

\begin{thm}[Sufficiency for compactness on \(\VMOA_v\)]\label{thm:SuffForCptVMOA}
Assuming \(v\) is admissible (see beginning of Section \ref{sec:WrappingItUp}), if \(\WCO\in \BOP(\VMOA_{v} ) \) and \(  \lim_{\abs{\phi(a)}  \to 1  } (\alpha(\psi,\phi,a) + \beta(\psi,\phi,a)) =0  \), then  \(\WCO\) is compact.
\end{thm}
\begin{proof}
If \(v\) is bounded, the result is proved in \cite[Theorem 4.3]{Laitila-2009}, hence, we assume \(v\) is unbounded. First, we prove that if  \(\WCO\in \BOP(\VMOA_{v} ) \), then  
\begin{equation}\label{eq:VMOAStrongAssumption}
\lim_{\abs{\phi(a)}  \to 1  } (\alpha(\psi,\phi,a) + \beta(\psi,\phi,a)) =0 \ \Longrightarrow \ \lim_{\abs{a} \to 1  } (\alpha(\psi,\phi,a) + \beta(\psi,\phi,a)) =0. 
\end{equation}
Indeed, if \( (a_n) \) is a sequence with \( \lim_n \abs{a_n}\to 1 \), we extract an arbitrary subsequence, also called \( (a_n) \). If there is a subsequence \(  (a_{n_k}) \subset  (a_n) \) with \( \lim_k\abs{\phi(a_{n_k})} = 1 \), we can conclude that
\[
 \lim_k (\alpha(\psi,\phi,a_{n_k}) + \beta(\psi,\phi,a_{n_k})) = 0
\]
for such a subsequence. If this is not the case,  \(\abs{\phi(a_n)}\in \closed{R\D}\)  for some \(R<1\), and hence, by Corollary \ref{cor:BoundednessOfWCO_VMOA}
\[
\lim_n \alpha(\psi,\phi,a_n) \lesssim_{v,\phi,(a_n)} \lim_n   v(a_n) \psi(a_n)  \norm{\phi(a_n)-\phi\circ\sigma_{a_n}}_{H^2} = 0
\]
and
\[
\lim_n \beta(\psi,\phi,a_n) \lesssim_{v,\phi,(a_n)} \lim_n  v(a_n)  \norm{ \psi\circ\sigma_{a_n} - \psi(a_n) }_{H^2} = 0.
\]
This proves that for an arbitrary sequence \( (a_n) \) with \(\lim_n \abs{a_n} = 1\), every subsequence of \( (\alpha(\psi,\phi,a_n) + \beta(\psi,\phi,a_n))_n \) has a convergent subsequence, with limit zero, which means it converges to zero as \(n\to\infty\).

We are now ready to prove that \(\WCO\) is the uniform limit (in operator norm) of the compact operators \( \WCO K_n , n\in\mathbb N \), where \(K_n := T_{\frac{n}{n+1}} \), \(  f  \mapsto  [z\mapsto f(\frac{n}{n+1} z)],  f\in \VMOA_v \). The fact that \(K_n\in  \BOP(\VMOA_v)\) is compact follows from the following: Let \( (f_k) \) be a bounded sequence in \(\VMOA_v\) and \(f_0\) be the limit w.r.t. \( \tau_0\) (convergence on compact subsets) of some subsequence \((f_{k'})\). Applying \eqref{eq:LittlewoodPaley}, followed by and the Cauchy formula to estimate the \(f'\) by \(f\) yield
\[
\norm{K_n f_{k'} - K_n f_0}_{\BMOA_v}  \lesssim_{v,n}  \norm{K_{n+1} (f_{k'} -  f_0) }_\infty \stackrel{k'\to\infty}{\longrightarrow} 0.
\]
Although \(f_0\) might not be in \(\VMOA_v\), the function \(K_n f_0 \in \VMOA_v\) by the remark after \eqref{eq:LittlewoodPaley}.

Continuing, for \(n\in\mathbb N\), we have by Proposition \ref{prop:PracticalEstimate}, for every \(r\in]\frac{1}{2},1[\),
\begin{equation}\label{eq:EssNormEstimateVMOA}
\begin{split}
 \sup_{f\in B_{\VMOA_v}}  \norm{  (  \WCO  - \WCO K_n)f   }_{\BMOA_v} &\lesssim_v  \sup_{f\in B_{\VMOA_v}}  \frac{ 1 +   \sup_{z\in]0,r[} v(z)  }{(1-r)^{\frac{5}{2}}} \sup_{\abs{z}\leq r} \abs{ (  \WCO  - \WCO K_n)f(z)} \\
&\quad+ \sup_{f\in B_{\VMOA_v}} \sup_{\abs{a}\geq r} v(a) \gamma(  (  \WCO  - \WCO K_n)f   ,a,2).
\end{split}
\end{equation}
Furthermore, as in the proof of \cite[(4.13)]{Laitila-2009} for \(r\in]\frac{1}{2},1[\) an application of Cauchy's integral formula yields
\[
\lim_n \sup_{f\in B_{\VMOA_v}}  \sup_{\abs{z}\leq r} \abs{ (  \WCO  - \WCO K_n)f(z)} = 0.
\]
An application of Proposition \ref{prop:implicationOfJohnNirenberg} yields that for any \(R_{\BMO}\in]0,1[\), \(R_A\in]0,R_{\BMO}/2[\) and \(f\in\VMOA_v\), we have 
\begin{align*}
\sup_{\abs{a}\geq 1-R_A} v(a) \gamma( (  \WCO  - \WCO K_n)f  ,a,2) &  \lesssim_{v,\epsilon_0}     \sup_{\abs{a}\geq 1-R_{\BMO}} v(a) \gamma(  (  \WCO  - \WCO K_n)f  ,a,1) \\
&\quad   + \norm{  (  \WCO  - \WCO K_n)f   }_{*,v,2} \bigg(  \frac{2R_A}{R_{\BMO}}   \bigg)^{\epsilon_0   }.  
\end{align*}

Combining this estimate with Lemma \ref{lem:WCO_Bounded} and the fact that \(\sup_n\norm{K_n}_{\BOP(\VMOA_v)} < \infty\) yield
\begin{align*}
\sup_{f\in B_{\VMOA_v}} & \sup_{\abs{a}\geq 1-R_A} v(a) \gamma(    \WCO( I - K_n)f   ,a,2) \\
&\lesssim_{v,\epsilon_0}  \sup_{\abs{a}\geq 1-R_{\BMO}} \bigg(  \alpha(\psi,\phi,a)   +    v(a)   \norm{   \psi\circ\sigma_a -   \psi(a)   }_{H^2}          \frac{       \norm{\phi_a}_2  }{ v(\phi(a))  }  + \beta(\psi,\phi,a)   \bigg)   \\
&\quad + \norm{  \WCO   }_{\BOP(\VMOA_v)} \bigg(  \frac{2R_A}{R_{\BMO}}   \bigg)^{\epsilon_0   }.
\end{align*}
Using \(r=1-R_A\) in \eqref{eq:EssNormEstimateVMOA} gives us

\begin{align*}
\limsup_n & \sup_{f\in B_{\VMOA_v}}  \norm{  (  \WCO  - \WCO K_n)f   }_{\BMOA_v} \\
& \lesssim_{v,\epsilon_0}  \sup_{\abs{a}\geq 1-R_{\BMO}} \bigg(  \alpha(\psi,\phi,a)   +   v(a)   \norm{   \psi\circ\sigma_a -   \psi(a)   }_{H^2}        + \beta(\psi,\phi,a)   \bigg)  \\
&\quad + \norm{  \WCO  }_{\BOP(\VMOA_v)} \bigg(  \frac{2R_A}{R_{\BMO}}   \bigg)^{\epsilon_0   }.
\end{align*}
Since \( \psi\in \VMOA_v  \) by Corollary \ref{cor:BoundednessOfWCO_VMOA}, letting \(R_A\to 0\) followed by \(R_{\BMO}\to 0\) yield
\begin{equation}\label{eq:essNormUpperBound}
\limsup_n \norm{  \WCO  - \WCO K_n  }_{\BOP(\VMOA_v)} \lesssim_{v,\epsilon_0} \limsup_{\abs{a}\to 1} (\alpha(\psi,\phi,a)   + \beta(\psi,\phi,a) ), 
\end{equation}
where the right-hand side, considering \eqref{eq:VMOAStrongAssumption}, is zero by assumption.
 
\end{proof}

Similarly to \cite[Theorem 4.3]{Laitila-2009}, using Lemmas \ref{lem:AlphaBounded} and \ref{lem:BetaBounded}, we have the following Corollary:
\begin{cor}
Assuming \(v\) is admissible, if \(\WCO\in \BOP(\VMOA_{v} ) \), then the essential norm, is given by  
\begin{align*}
\norm{\WCO}_{e,\BOP(\VMOA_v)} := \inf_{ \substack{ K\in\BOP(\VMOA_v) \\
K \text{ compact} }}  \norm{ \WCO - K }_{\BOP(\VMOA_v)} &\asymp \limsup_{\abs{a}\to 1} (\alpha(\psi,\phi,a) + \beta(\psi,\phi,a) )  \\
&\asymp \limsup_{ \abs{\phi(a)} \to 1} (\alpha(\psi,\phi,a) + \beta(\psi,\phi,a) ).
\end{align*}

\end{cor}

We end this section with showing that for many admissible weights, the function \(\PEF\) defined in Lemma \ref{lem:GoodOne} is not in \( \VMOA_v \). Therefore, \( \VMOA_v \) is a proper subspace of \(\BMOA_v\).

\begin{prop}
If \(v\) is admissible and \(g\) satisfies the reverse inequality \ref{eq:mainAssumpA3}, that is, \(\abs{g(z)} \lesssim  g(\abs{z}) , \ z\in \C_{\Re\geq\frac{1}{2}} \), then \(\PEF \in \BMOA_v\setminus \VMOA_v\), where
\[
\PEF\colon  z\mapsto \int_0^z  \frac{ dt}{   (1-t) g(\frac{1}{1-t})   }.
\]
\end{prop}
\begin{proof}
We begin by proving that \(\PEF\not\in \VMOA_v\). Let \(a\in ]\frac{1}{2},1[\) and \(S(a)\) be the Carleson set defined in \eqref{eq:CarlesonSquare} and put \(S_a = S(a) \cap  a\closed{\D}\). For \(z\in S_a\), it holds that \(1-\abs{z} \geq 1-a \),\(\abs{1-z}\lesssim 1- a\), and from \eqref{eq:PoissonKernelEstimate}, it follows that \(\abs{1-\CONJ{a}z}^2 \lesssim (1-a)^2\). By the Littlewood-Paley identity, \eqref{eq:LittlewoodPaley} with \(p=2\) and the assumption \(\abs{g(z)} \lesssim  g(\abs{z}) , \ z\in \C_{\Re\geq\frac{1}{2}} \), we have 
\begin{align*}
\gamma(\PEF,a,2) &\gtrsim  \int_{S_a} \frac{(1-\abs{z}^2)(1-\abs{a}^2)}{       \abs{1-z}^2 \abs{g(\frac{1}{1-z})}^2         \abs{ 1-\CONJ{a}z}^2   } \, dA(z)  \gtrsim_g \frac{     \int_{S_a}  dA(z)       }{  (1-a)^2g(\frac{1}{1-a})^2   }  \asymp g(\frac{1}{1-a})^{-2}.
\end{align*}
We can conclude that
\[
\lim_{a\to 1}  v(a)^2 \gamma(\PEF,a,2) \gtrsim_g 1, 
\]
proving \(\PEF\not\in\VMOA_v\). By Lemma \ref{lem:GoodOne}  \(\PEF\in\BMOA_v\).

\end{proof}

The condition \(\abs{g(z)} \lesssim  g(\abs{z}) , \ z\in \C_{\Re\geq\frac{1}{2}} \) is trivially fulfilled for the standard weights, \(g(z)=z^c,  0\leq c<1/2\). The condition is also fulfilled for \( g(z) = (\ln (ez))^c, c>0 \). Indeed, for \( z\in \C_{\Re\geq\frac{1}{2}} \) we have \(  (\ln (e \abs{z})) \geq \ln (e/2)\), and hence, 
\[
 (\ln (e \abs{z}))^2  \leq  \abs{\ln (e z)}^2  \leq  (\ln (e \abs{z}))^2 +  (\pi/2)^2 \lesssim  (\ln (e \abs{z}))^2 .
\]

%
%

\section{Examples}\label{sec:Examples}
This section contains some practical examples of symbols \(\psi\) and \(\phi\) making \(\WCO\) bounded and compact on \(\BMOA_v\) and \(\VMOA_v\), where \(v\) is an admissible weight (see beginning of Section \ref{sec:WrappingItUp}). Before we proceed, we have the following useful lemma.

\begin{lem}\label{lem:ExampleNewWeight}
Let \(v\), be admissible. The weight 
\[
w(a) = v(a) \bigg(1+\int_0^{\abs{a}} \frac{  dt  }{ (1-t)  v(t)  } \bigg), \quad a\in\D
\] 
satisfies \( \sup_{0<x<1}   x\, w(1-x)^{2+\epsilon_0}   < \infty\), where \(\epsilon_0>0\) is given in \ref{eq:mainAssumpA1}.
\end{lem}
\begin{proof}
For \(a\in\D\), \ref{eq:mainAssumpA2} gives us
\[
w(a) \asymp_{v,g} v(a) + \int_0^{\abs{a}} \frac{ v(\abs{a})\,  dt  }{ (1-t)  v(t)  } \lesssim_{v,g}  v(a) + \int_0^{\abs{a}}  v\bigg(  \frac{\abs{a}-t}{1-t}  \bigg)  \frac{  dt  }{ (1-t)    } =  v(\abs{a}) + \int_0^{\abs{a}}  v( t )  \frac{  dt  }{ (1-t)    },
\]
where the substitution \(t\mapsto (\abs{a}-t)/(1-t)\) was used to obtain the last equality. Furthermore, by \ref{eq:mainAssumpA1} we have
\[
\int_0^{\abs{a}}  v( t )  \frac{  dt  }{ (1-t)    } \lesssim_{v,g,\epsilon_0}  \int_0^{\abs{a}}  (1-t)^{-\big(1+\frac{1}{2+\epsilon_0}\big)}   \, dt   \lesssim_{\epsilon_0}  (1-\abs{a})^{-\frac{1}{2+\epsilon_0}} 
\]
yielding
\[
\sup_{a\in\D} (1-\abs{a}) w(a)^{2+\epsilon_0} \lesssim_{v,g,\epsilon_0}   \sup_{a\in\D} (1-\abs{a}) v(a)^{2+\epsilon_0} + 1  
\]
and we are done.
\end{proof}

Via the function-theoretic characterization of boundedness and compactness of \(\WCO\), it is clear that if \(\psi\) and \(\phi\) makes \(\WCO\) act in a bounded (compact) manner on \(\VMOA_v\), then \(\WCO\) will  act in a bounded (compact) manner on \(\BMOA_v\) too. By Proposition \ref{prop:polynomialsInBMOAV}, \(\VMOA_v\) contains all analytic polynomials for any weight, \(v\), satisfying \(  \lim_{a\to 1}  v(a)^2  (1-\abs{a}) = 0 \). Therefore, for polynomial symbols \(\psi\) and \(\phi\) for which \(\sup_{a\in\D}\alpha(\psi,\phi,a)<\infty\), \(\WCO\) acts boundedly on \(\VMOA_v\) (and \(\BMOA_v\)). More is true, let \(\psi = q_1\) and \( \phi = q_2\) be two fractions of polynomials, where the denominators have no zeros in \(\closed{\D}\). By \eqref{eq:LittlewoodPaley}, it follows that \(\norm{q_j}_{\BMOA_v} \asymp_{q_j} \norm{p_j}_{\BMOA_v} \) for some polynomials \(p_j\), \(j=1,2\) for any admissible weight \(v\). It also follows that \(q_1,q_2,q_1 q_2\in \VMOA_v\), and by Corollaries  \ref{cor:EvaluationBoundedOnBMOAv}, \ref{cor:EvaluationHierarchyOnBMOAv} and Lemma \ref{lem:ExampleNewWeight}, \(\lim_{\abs{a}\to 1}\beta(q_1,q_2,a) = 0 \). All that remains for boundedness and compactness, is to prove that  \(\sup_{a\in\D}\alpha(q_1,q_2,a)<\infty\) and \(\lim_{\abs{q_2(a)}\to 1}\alpha(q_1,q_2,a) =0\) respectively. Note that \(\lim_{\abs{q_2(a)}\to 1}\alpha(q_1,q_2,a) =0\) grants boundedness. This follows from \(q_2 \in \VMOA_v\), \(q_1\in H^\infty\) and 
\[
\sup_{\abs{q_2(a)}\leq R} \norm{(q_2)_a}_{H^2} \asymp_R  \norm{q_2(a)-q_2\circ\sigma_a}_{H^2}, \quad 0<R<1.
\] 
It is worth noting that if \(\abs{q_2(\eta)}=1\) for some \(\eta \in \T\), the continuity of \(q_2\) ensures there is a disk \(D:=\eta(c+(1-c)\D)\), \( 0<c<1\) such that \(\lim^*_{\abs{q_2(a)}\to 1}\alpha(q_1,q_2,a)=  0\), where the limit is taken outside the disk \(D\).

Turning our attention to the multiplication operator \(M_\psi \colon f\mapsto \psi f \), Corollary \ref{cor:MainCorMult} yields it is bounded but not compact on \(X=\BMOA_v\) or \(X=\VMOA_v\) (\(v\) admissible), for any \( \psi\in H^\infty \cap \BMOA_w\setminus \{0\} \), where \(w(a) = v(a)\norm{\delta_a}_{(X)^*} , a\in\D \). It is worth noting that the product of admissible weights is admissible if and only if \ref{eq:mainAssump1} is satisfied for the product. Moreover, by Corollary \ref{cor:EvaluationBoundedOnBMOAv} the admissible weight 
\begin{equation}\label{eq:wOne}
w_1(a) := v(a) \ln\frac{e}{1-\abs{a}} , a\in\D 
\end{equation} 
dominates \(w\) yielding \(\BMOA_{w_1} \subset \BMOA_w\).

Another interesting fact is the following result concerning the composition operator \(C_\phi\). Let \(v\) be an admissible weight. Assume \(\psi\in X\), where \(X=\BMOA_v\) or \(X=\VMOA_v\) and \(\phi\) is a self-map of \(\D\), continuous on \(\closed{\D}\). If \(\WCO \colon X\to X \) is a bounded (compact) operator, then for every \(n\in\mathbb N\), \(\psi C_{\phi^n}\colon X\to X \) is bounded (compact). To this end, since \(v\) is almost increasing, \ref{eq:mainAssumpA2} yields
\begin{equation}\label{eq:vasympvn}
v(a^n) \lesssim_{v} v(a) \lesssim_{v,g} v\bigg( 1- \frac{ 1 - \abs{a}}{ 1-\abs{a}^n } \bigg) v(a^n) \lesssim_v v(1-\frac{1}{n}) v(a^n) , \quad a\in\D,n\in\mathbb N.
\end{equation}
Furthermore, put \(c_0:= \epsilon_0/(2(1+\epsilon_0)) \). By Lemma \ref{lem:ExampleNewWeight} and Corollary \ref{cor:EvaluationBoundedOnBMOAv}, we have
\begin{equation}\label{eq:extraTermToZero}
\limsup_{\abs{a}\to 1}   (1-\abs{a})^{1-c_0} v(a)^2 \norm{\delta_a}_{X^*}^2 = 0.
\end{equation}
We will establish
\begin{equation}\label{eq:asympSamePhiA}
 \norm{ (\phi^n)_a}_{H^2}^2  \lesssim_n (1-\abs{a})^{1-c_0} + \norm{\phi_a}_{H^2}^2,
\end{equation}
and after multiplying both sides of \eqref{eq:asympSamePhiA} with \(v(a)^2\psi(a)^2/v(\phi(a))^2\)  \eqref{eq:vasympvn} yields
\[
\alpha(\psi,\phi^n,a)^2 \lesssim_{n,v,g} \frac{(1-\abs{a})^{1-c_0}v(a)^2\psi(a)^2}{v(\phi(a))^2} + \alpha(\psi,\phi,a)^2.
\]
 In view of Theorems \ref{thm:MainThm} and \ref{thm:MainThmCpt}, the statement concerning boundedness follows from taking the limit \(\abs{a}\to 1\) together with \eqref{eq:extraTermToZero}, and compactness follows from considering \(\abs{\phi(a)}\to 1\).

To prove \eqref{eq:asympSamePhiA}, let \(n\in\mathbb N\) and \(r\in]0,\frac{1}{2}[\) be small enough so that for \(a\in \D\) with \(\abs{a}\geq 1-r\), \( \abs{ \phi(z)- \phi(a) } < \frac{1}{2n^2}\) for every 
\[
z\in  J(a) :=\bigg\{ w\in\T : \abs{w-\frac{a}{\abs{a}}}<2(1-\abs{a})^{c_0/2} \bigg\}. 
\]
Furthermore, \(\sup_{z\in \T\setminus J(a)} P_a(z) \lesssim (1-\abs{a})^{1-c_0} \). It follows that for \( \abs{a}\geq 1-r \) and \(z\in  J(a) \), we have
\begin{equation}\label{eq:ctsOfPhi}
\abs{ \sum_{k=0}^{n-1} \CONJ{\phi(a)}^k \phi(z)^k  -  \sum_{k=0}^{n-1} \CONJ{\phi(a)}^k \phi(a)^k  }  \leq \sum_{k=0}^{n-1}  \abs{ \phi^k(z)  -  \phi(a)^k  } \leq n^2 \abs{ \phi(z)  -  \phi(a)  } <\frac{1}{2},
\end{equation}
and hence,
\begin{equation}\label{eq:PhiNLessPhi}
\begin{split}
\norm{ (\phi^n)_a}_{H^2}^2   &=  \int_{\T} \abs{  \frac{   \phi(a)^n - \phi^n  }{    1 - \CONJ{\phi(a)}^n  \phi^n    }  }^2  P_a  \, dm \\
& \lesssim (1-\abs{a})^{1-c_0}  + \int_{J(a)} \abs{  \frac{   \sum_{k=0}^{n-1} \phi(a)^k \phi^{n-k-1}  }{   \sum_{k=0}^{n-1} \CONJ{\phi(a)}^k \phi^k    }  }^2   \abs{  \frac{   \phi(a) - \phi  }{    1 - \CONJ{\phi(a)}  \phi    }  }^2  P_a  \, dm \\
&\stackrel{\eqref{eq:ctsOfPhi} } { \leq }  (1-\abs{a})^{1-c_0}  + \int_{J(a)} \abs{  \frac{  n  }{   \sum_{k=0}^{n-1} \abs{\phi(a)}^{2k} - \frac{1}{2}    }  }^2   \abs{  \frac{   \phi(a) - \phi  }{    1 - \CONJ{\phi(a)}  \phi    }  }^2  P_a  \, dm \\
& \leq  (1-\abs{a})^{1-c_0}  + 4n^2 \int_{\T}   \abs{  \frac{   \phi(a) - \phi  }{    1 - \CONJ{\phi(a)}  \phi    }  }^2  P_a \, dm \leq (1-\abs{a})^{1-c_0}  + 4n^2\norm{\phi_a}_{H^2}^2.
\end{split}
\end{equation}

The following partial complement to \eqref{eq:asympSamePhiA} can be proved using similar calculations to \eqref{eq:PhiNLessPhi}: If we only consider sequences \((a_j)\), with \(\inf_j\abs{\phi(a_j)}>0 \), then
\begin{equation}\label{eq:phiAkbddFrB}
\norm{ \phi_{a_j}}_{H^2}^2  \lesssim_{n,(\phi(a_j))} (1-\abs{a_j})^{1-c_0} + \norm{(\phi^n)_{a_j}}_{H^2}^2.
\end{equation}

If \(\phi(\D)\subset b\closed{\D}\) for some \(0<b<1\), then by Corollary \ref{cor:EvaluationBoundedOnBMOAv}
\[
\alpha(\psi,\phi,a)\asymp_{b,v} \abs{\psi(a)}v(a)\norm{\phi\circ\sigma_a - \phi(a)}_{H^2}  \lesssim_{\psi,v} \ln\frac{e}{1-\abs{a}}v(a)\norm{\phi\circ\sigma_a - \phi(a)}_{H^2}
\]
and
\[
\beta(\psi,\phi,a)\asymp_{b} v(a)\norm{\psi\circ\sigma_a - \psi(a)}_{H^1}.  
\]
It follows that if \(\phi\in \VMOA_{w_1}\cap H^\infty\), then \(\phi\) can be scaled to grant that the operator \(\WCO\colon X\to X \) is bounded (compact) if \(\psi\in X\), where \(X=\BMOA_v\) or \(X=\VMOA_v\). Since \(w_1\) is admissible (see \eqref{eq:wOne}), the remark right after \eqref{eq:LittlewoodPaley} yields many examples \(\phi\in \VMOA_{w_1}\cap H^\infty\), e.g. a dilation of an analytic function.

Recall that a Blaschke product is a function of the form
\[
z\mapsto z^m\prod_n\frac{\abs{b_n}}{b_n} \frac{b_n-z}{1-\CONJ{b_n}z},  
\] 
where \(m=0,1,\ldots\) and \((b_n)\subset \D\cup\{1\}\setminus \{0\}\) is a sequence satisfying \( \sum_n (1-\abs{b_n}) <\infty   \) (\cite[p.~20]{Duren-1970}). An interesting fact is that an infinite Blaschke product (\(b_n\neq 1\) for infinitely many \(n\)) as the symbol \(\phi\) gives rise to a bounded composition operator if and only if \(v\asymp 1\). When \(v\asymp 1\), the statement follows from \(\alpha(1,\phi,a) \asymp \norm{\phi_a}_{H^2}\leq 1\). When \(v\) is unbounded, let \((a_n)\subset \D\) be the zeros of \(\phi\). Since for every \(a\in\D\), the functions \(\sigma_{\phi(a)}\) and \(\sigma_a\) are automorphisms on \(\T\) and \(\phi\) has a modulus \(1\) a.e., it follows that \(\norm{\phi_a}=1\) for all \(a\in \D\). Moreover, \(v(a_n)/v(\phi(a_n)) = v(a_n)/v(0) \to \infty\) as \(n\to\infty\), due to \(\lim_n \abs{a_n} = 1\), resulting in \(\sup_{a\in\D} \alpha(1,\phi,a) =\infty\).

Next, we consider the polynomial \(\phi(z) = \frac{1+z}{2}, z\in \D\). To see that it is bounded, it is enough to prove that
\[
\limsup_{ \abs{\phi(a)}\to 1  } \alpha(1,\phi,a)<\infty,
\] 
which is the same as
\[
\limsup_{ a\to 1  } \alpha(1,\phi,a)<\infty.
\] 
In \cite[Example 5.1]{Laitila-2009}, J. Laitila showed that 
\[
\norm{\phi_a}_{H^2}^2 = \frac{1-\abs{a}^2}{2(1-\Re a)} . 
\]
Since \(\abs{\phi(a)}\geq \abs{a}\) for \(a\in 1/2+(1/2)\closed{\D}\), we have for these \(a\in\D\), \(v(a)/v(\phi(a))\lesssim 1\). It remains to examine the tangential limits \(a\to 1\), that lie outside this disk. These are \(a\in \D\) near \(1\) for which \(\Re a \leq \abs{a}^2\). Using this and \ref{eq:mainAssumpA2}, we have

\[
1-\abs{\phi(a)}^2 = 1-\frac{1}{4}(1+\abs{a}^2+2\Re a) \leq\frac{3}{4} (1-\Re a)
\]
and

\[
\frac{v(a)}{v(\phi(a))} \asymp_{v,g} \frac{v(a^2)}{v(\phi(a)^2)} \lesssim_{v,g}  v\bigg( 1- \frac{ 1 - \abs{a}^2}{ 1-\abs{\phi(a)}^2 } \bigg) \lesssim_v v\bigg( 1-\frac{3}{4} \frac{ 1 - \abs{a}^2}{ 1-\abs{\phi(a)}^2 } \bigg) \lesssim_v v\bigg( 1-  \frac{ 1 - \abs{a}^2}{ 1-\Re a  } \bigg). 
\]
It now follows from \ref{eq:mainAssumpA1} that
\begin{equation}\label{eq:ExCphiBounded}
\sup_{a\in\D} \alpha(1,\phi,a) =  \sup_{a\in\D} \frac{v(a)}{v(\phi(a))}\norm{\phi_a}_{H^2} <\infty
\end{equation}
proving \(C_{\phi}\colon \VMOA\to \VMOA\) is bounded. Considering the limit \(a\to 1\) along the real line together with Lemma \ref{lem:HelpWithApplyingTheKnown}, it is clear that \(C_{\phi}\) is not compact.

Let \(\phi(z) = \frac{1+z}{2}, z\in \D\) and \(\psi(z) = 1-z, z\in\D\), which makes \(M_{\psi}\in\BOP(\VMOA_v)\) in accordance with Proposition \ref{prop:polynomialsInBMOAV} and the discussion above. In \cite[Example 5.1]{Laitila-2009}, J. Laitila proved that \(\WCO\) acts as a compact operator \(\VMOA\to\VMOA\) although neither \(M_{\psi}\) nor \(C_{\phi}\) is compact. This example works in the same manner on \(\VMOA_v\). All that remains to prove for this statement is  that \(\WCO\) is compact on \(\VMOA_v\). From the first paragraph after the proof to Lemma \ref{lem:ExampleNewWeight}, it is sufficient to prove that
\[
\lim_{a\to 1} \alpha(\psi,\phi,a) =0.
\]
This is true due to \eqref{eq:ExCphiBounded} and
\[
\alpha(\psi,\phi,a) = \abs{1-a}\alpha(1,\phi,a).
\]

\section{Proof of main results}\label{sec:proofsOfMainResults}

\begin{proof}[Proof of Theorem \ref{thm:MainThm}]

Theorem \ref{thm:BoundednessOfWCO}, Corollary \ref{cor:BoundednessOfWCO_VMOA} and Corollary \ref{cor:EvaluationBoundedOnBMOAv} yield the statements.

\end{proof}

\vspace{0.4cm}

\begin{lem}\label{lem:FuncCharactTwo}
Let \(X=\BMOA_v\) or \(X=\VMOA_v\), where \(v\) is an admissible weight (see beginning of Section \ref{sec:WrappingItUp}). It holds for \(\WCO\in\BOP(X) \) that
\[
\WCO \text{ is compact } \ \ \Longrightarrow \ \ \limsup_{ \abs{\phi(a)} \to 1} \norm{\WCO\fAlpha}_{\BMOA_v}  = 0  \ \ \Longrightarrow \ \  \limsup_{ \abs{\phi(a)} \to 1} \alpha(\psi,\phi,a)  = 0, 
\]

where 
\[
\fAlpha\colon z \mapsto  \frac{\sigma_{\phi(a)}(z) - \phi(a) }{v(\phi(a)}.
\]

\end{lem}
\begin{proof}
Note that \(C:=\sup_{a\in\D}\norm{   \smash{\fAlpha} }_{\BMOA_v}<\infty \) by \eqref{eq:coolStuff} and \(\WCO|_{CB_X}\) is always \(\tau_0-\tau_0\) continuous. If \(X=\BMOA_v\), Lemma \ref{lem:Btau0Cpt} yields \((CB_X,\tau_0)\) is compact. Now, \cite[Lemma 3.3]{Contreras-2016} proves the first implication, noting that for every \(0< R<1\)
\[
 \sup_{\abs{z}\leq R}  \abs{\fAlpha(z)}  \lesssim \frac{  R (1-\abs{\phi(a)}^2)   }{    1-  \abs{\phi(a)}  R  },
\]
and hence, \(\fAlpha \to 0\) w.r.t. \(\tau_0\) as \(\abs{\phi(a)}\to 1\). On closer inspection, the use of \(B_X\) being compact w.r.t. \(\tau_0\) is redundant for the implication used above. Indeed, if \(\WCO\) is compact, \(K := \closed{\WCO(C B_X)}^{\norm{\cdot}_{\BMOA_v}}\) is compact and by a standard argument, the norm topology induced by \(\norm{\cdot}_{\BMOA_v}\) restricted to \(K\)  is equivalent to the topology on \(K\) induced by \(\tau_0\). Therefore, the first implication also holds for \(X=\VMOA_v\).

Concerning the second implication, the boundedness of \(\WCO\) ensures  \(\sup_{a\in\D}\beta(\psi,\phi,a)<\infty\) and \(\psi\in\BMOA_v\). If \(v\) is unbounded, Lemma \ref{lem:AlphaBounded} gives the second implication. If \(v\) is bounded, Lemma \ref{lem:AlphaBounded} combined with \cite[(3.10)]{Laitila-2009} gives the implication.

\end{proof}

\begin{proof}[Proof of Theorem \ref{thm:MainThmCpt}]

In the following, Theorem (*) means either Theorem \ref{thm:SuffForCpt} or Theorem \ref{thm:SuffForCptVMOA} depending on if \(\WCO\) is operating on \(\BMOA_v\) or \(\VMOA_v\). The statements are proved as follows:    
\[
\eqref{eq:Main_FuncTheorCharact}    \  \  \    \stackrel{  \mathclap{   \substack{ \text{ Theorem (*) }   \\   \  }   }     }{\Longrightarrow}   \  \  \    \eqref{eq:Main_Cpt}   \  \  \     \stackrel{  \mathclap{   \substack{ \text{ Lemma } \ref{lem:Structure}  \\   \  }   }     }{\Longrightarrow}    \  \  \  \substack{  \eqref{eq:Main_wCpt} \\
 \text{ and } \\
 \eqref{eq:Main_cCts}   }  \  \  \  \Longrightarrow   \  \  \  \substack{  \eqref{eq:Main_wCpt} \\
 \text{ or } \\
 \eqref{eq:Main_cCts}   }   \  \  \   \stackrel{  \mathclap{   \substack{ \text{ Lemma } \ref{lem:Structure}  \\   \  }   }     }{\Longrightarrow}   \  \   \  \eqref{eq:Main_fixC0}     \  \  \  \stackrel{  \mathclap{   \substack{ \text{ Theorem } \ref{thm:fixCopyC0NotUC}  \\   \  }   }     }{\Longrightarrow}  \   \  \  \eqref{eq:Main_FuncTheorCharact},   
\]

and
\[
    \eqref{eq:Main_fixC0}  \  \  \   \stackrel{  \mathclap{   \substack{ \text{ Lemma } \ref{lem:UCfixC0}  \\   \  }   }     }{\Longrightarrow}  \  \  \  \eqref{eq:Main_UC}  \  \  \  \stackrel{  \mathclap{   \substack{ \text{ Theorem } \ref{thm:fixCopyC0NotUC}  \\   \  }   }     }{\Longrightarrow}  \  \  \  \eqref{eq:Main_fixC0}
\]
followed by
\[
    \eqref{eq:Main_Cpt}  \  \  \   \stackrel{  \mathclap{   \substack{ \text{ \cite[pp.~120-121]{XXX} }  \\   \  }   }     }{\Longrightarrow}  \  \  \     \eqref{eq:Main_FSS}  \  \  \   \stackrel{  \mathclap{   \substack{ \text{ Def. }  \\   \  }   }     }{\Longrightarrow}  \  \  \   \eqref{eq:Main_SS}  \  \  \  \stackrel{  \mathclap{   \substack{ \text{ Def. }  \\   \  }   }     }{\Longrightarrow}  \  \  \  \eqref{eq:Main_fixC0}.
\]

Lemma \ref{lem:FuncCharactTwo} connects \eqref{eq:Main_FuncTheorCharactTwo} with the rest of the equivalent statements in the first list.

Theorem \ref{thm:Gantmacher} together with the rest of Subsection \ref{ss:GeneralBanach} proves the equivalence of   \eqref{eq:Main_wCpt},   \eqref{eq:Main_Gantmacher} and   \eqref{eq:Main_AC}.

\end{proof}
The only implications in the proof that do not hold for an arbitrary bounded operator on a Banach space \( T\in\BOP(X) \) are the ones involving (\ref{eq:Main_FuncTheorCharact}), (\ref{eq:Main_FuncTheorCharactTwo}). The implication \((\ref{eq:Main_UC}) \  \Rightarrow  \ (\ref{eq:Main_fixC0})\) holds in general, which can be seen using the standard basis of \(c_0\) as a wuC series, which does not converge unconditionally.

To finish the section, some open problems and conjectures are gathered.

\begin{conj}
The condition \( x\mapsto v(1-x)x^{\frac{1}{p}-\epsilon}\) is almost increasing for some \(\epsilon>0\) can be replaced by the milder assumption \( \inf_{x\in ]0,1[} v(x) > 0 \) in Proposition \ref{prop:implicationOfJohnNirenberg}. 
\end{conj}
The problematic part is Proposition \ref{prop:intBMOAimplyConformInvBMOA}. An initial idea is to apply Baernstein's approach. However, instead of \(\AUT\) being a group with respect to composition, a straightforward approach would demand that 
\[
 \{  \hat{\phi}\in \AUT : \abs{\hat{\phi}(0)}\geq R  \} 
\] 
is a group, which is not true (it is not closed under composition). The binary nature of integrating over an interval (integration over \(\T\) with the binary function \(\chi_I\)) allows Lemma \ref{lem:JohnNirenberg} and the very beginning of the proof of Proposition \ref{prop:implicationOfJohnNirenberg} regardless of a strictly positive weight \(v\). In the conformal setting (integration over \(\T\) with the Poisson kernel \(P_a\) as a weight) an unbounded weight \(v\) complicate matters.

\begin{conj}\label{conj:invariantCompAut}
There exists an increasing, radial (not admissible) weight \(v\colon \D \to ]0,\infty[\), \(\hat{\phi} \in \AUT\) and \( f\in \BMOA_v  \) such that \(  f\circ \hat{\phi}  \notin \BMOA_v  \).  
\end{conj}
A direct approach yields that this is  the same as proving that there exists \( \hat{\phi}\in\AUT \) and \( f\in \BMOA_v  \) such that
\[
\sup_{a\in\D} \frac{v(\hat{\phi}(a))}{v(a)}    v(a)  \gamma(f,a,2) = \infty.
\] 
Examining the quotient \smash{  \(\frac{v(\hat{\phi}(a))}{v(a)}\) }, we have the following: Using \smash{ \(v(z)=g(\frac{1}{1-\abs{z}^2}) \) }, for some increasing unbounded function \(g\colon [1,\infty[\to ]0,\infty[\), the automorphisms \smash{  \( \hat{ \phi }_c := \sigma_{-\sqrt{1-c}} , \ 0< c<1 \) } yield
\begin{equation}\label{eq:conjNotClosedWrtComp}
 g\left(\frac{1}{1-\abs{ \hat{\phi}_c (a) }^2}\right) = g\left(\frac{    \abs{  1 + \sqrt{1-c}a  }^2  }{  c (1-\abs{ a }^2)   }\right) \geq  g\left(\frac{1}{c} \frac{1}{1-\abs{a}^2}  \right)  , \ a\in ]0,1[.
\end{equation}
For any increasing unbounded function \(G\colon [1,\infty[\to [1,\infty[\)  of any growth rate, we can create a continuous \(g\), dominated by \(G\), as follows: Let \(x_0=1\) and \(g(x) = 1\) for all \(x\in[x_0,x_1]\), where \(x_1>x_0\) is a point where \( G(x_1)\geq 1  \). On \([x_1,x_1+1]\) we define \(g\) to increase to the value \( G(x_1+1)\). Define \(g\) to be constant for all \(x\in[x_1+1,x_2]\), where \(x_2>x_1+1\) is such that \(G(x_2)/G(x_1+1) \geq 2\). Continuing this process we obtain a sequence \((x_n)\), with the property that  \(G(x_{n+1})/G(x_n+1) \geq n+1\) for all \(n\) and \( ([x_n,x_n+1[)_n \cup ([x_n+1,x_{n+1}[)_n\) is a partition of \([1,\infty[\), where \(g\) is increasing on the first family of intervals and constant on the latter. Moreover, as \(x\geq 1\), the function \(g\) is  increasing, unbounded and satisfies 
\[
\frac{  g(2x_{n} )  }{ g( x_{n} )  } = \frac{  g(2x_{n} )  }{ g( x_{n-1}+1 )  }  \geq  \frac{  g(1+x_{n} )  }{ g( x_{n-1}+1  )  } \geq  \frac{  G(x_{n} )  }{ G( x_{n-1}+1  )  } \geq n   
\]
for all \(n\). In combination with \eqref{eq:conjNotClosedWrtComp}, it follows that such a weight, \(a\mapsto \frac{v(\hat{\phi}_{1/2}(a))}{v(a)}\) is unbounded on \([0,1[\). The remaining part would be to find a function \(f\) such that \(\limsup_{\abs{a}\to 1}  v(a)  \gamma(f,a,2) >0 \) if such exists. It would also be interesting if one could find an analytic function \(g\) satisfying the desired properties. Notice that if \(v\) increases too fast, then the space consists only of constant functions, in which case composition with an automorphism will not change the function at all. Recall that if \(v(a) \gtrsim (1-\abs{a})^{-(1+\epsilon)} \) for any \(\epsilon>0\), then \(\BMOA_v\) consists of constant functions.

\begin{conj}\label{conj:MainBMOA}
We can drop the assumption in Theorem \ref{thm:MainThmCpt} that at least one of the following is true:
\begin{enumerate}
\item \( C_{\phi} \in \BOP(\BMOA_v)\), and \(\BMOA_v \not\subset  H^\infty\) or \(\psi\in \VMOA_v\),
\item \(\WCO|_{\VMOA_v} \in \BOP(\VMOA_{v} )\).  
\end{enumerate}
\end{conj}

This extra assumption is only needed in Theorem \ref{thm:SuffForCpt} and Theorem \ref{thm:fixCopyC0NotUC}. Being able to remove it would yield a complete characterization of e.g. compactness of \(\WCO\in\BOP(\BMOA_v)\) for admissible weights.

\begin{conj}\label{conj:PolynomialSymbBdd}
Given an admissible weight \(v\) (see beginning of Section \ref{sec:WrappingItUp}), all analytic polynomial symbols \(\phi\colon \D\to\D\) makes \(C_\phi\) bounded.
\end{conj}
In section \ref{sec:Examples}, a few examples of polynomial symbols rendering \(C_\phi\) bounded are given.

\begin{conj}\label{conj:ComparingBetas}
Given an admissible weight \(v\) and \( X= \BMOA_v \) or \( X= \VMOA_v \) such that \(X\) is not contained in \(H^{\infty}\), it holds that
\[
 \sup_{a\in\D} \norm{\delta_{\phi(a)} }_{X^*} \gamma(\psi , a , 2)  \lesssim_{v,g,\psi,\phi}  \sup_{a\in\D} \norm{\delta_{\phi(a)} }_{X^*} \gamma(\psi , a , 1) \text{ and } 
\]  
and
\[ 
\limsup_{\abs{\phi(a)}\to 1} \norm{\delta_{\phi(a)} }_{X^*} \gamma(\psi , a , 2) \lesssim_{v,g,\psi,\phi}  \limsup_{\abs{\phi(a)}\to 1} \norm{\delta_{\phi(a)} }_{X^*} \gamma(\psi , a , 1).
\]  

\end{conj}

If \(X\subset H^\infty\) Proposition \ref{prop:implicationOfJohnNirenberg} yields the statements above.

\vspace{1cm}

{\bf Acknowledgement} The author was financially supported by the Magnus Ehrnrooth Foundation. I would like to thank the reviewers for their valuable comments, among other things, mentioning that methods from \cite{Perfekt-2013} and \cite{Perfekt-2015} could be applied to \(\VMOA_v\) and \(\BMOA_v\), which led to Theorem \ref{thm:Gantmacher}. I would also like to thank Professor Jani Virtanen for some valuable discussions and feedback on this manuscript.

\end{document}